\title{On graphs containing few disjoint excluded minors. Asymptotic number and structure of graphs containing few disjoint minors $K_4$}
\date{28 October 2013}

  \documentclass[11pt]{article}

  \usepackage{amsmath}
  \usepackage{amssymb}
  \usepackage{latexsym}
  \usepackage{graphicx}

   \usepackage{color}
 
  \usepackage{enumerate}
  \usepackage{setspace}

  \usepackage{cite} 

  \usepackage{hyperref}
  \hypersetup{ 
    colorlinks=false,
    pdfborder={0 0 0},
  }

  \newenvironment{proof}{\noindent{\bf Proof\,}}{\hspace*{\fill}$\Box$}
  \newenvironment{proofof}[1]{%
  \noindent {\bf Proof of #1}}%
  {\hspace*{\fill}$\Box$}

  \newtheorem{theorem}{Theorem}[section]
  \newtheorem{lemma} [theorem] {Lemma}
  \newtheorem{corollary} [theorem] {Corollary}
  \newtheorem{prop} [theorem] {Proposition}

\def\E{{\mathbb E}\,}

\let\eps=\epsilon
\def\enddiscard{}
\long\def\discard#1\enddiscard{}

  \newcommand{\pr}{\mathbb P}

  \newcommand{\cp}{{\mathcal P}}

  \newcommand{\ca}{{\mathcal A}}
  \newcommand{\cb}{{\mathcal B}}
  \newcommand{\cc}{{\mathcal C}}
  \newcommand{\cd}{{\mathcal D}}
  \newcommand{\ce}{{\mathcal E}}
  \newcommand{\cf}{{\mathcal F}}
  
  \newcommand{\ch}{{\mathcal H}}
  \newcommand{\ci}{{\mathcal I}}
  \newcommand{\cL}{{\mathcal L}}
  \newcommand{\cP}{{\mathcal P}}
  \newcommand{\cR}{{\mathcal R}}
  \newcommand{\cs}{{\mathcal S}}
  \newcommand{\ct}{{\mathcal T}}
  \newcommand{\cu}{{\mathcal U}}
  
  \newcommand{\cz}{{\mathcal Z}}

  \newcommand{\ex}{{\rm Ex\,}}

  \newcommand{\apex}{\mbox{apex \!}}

  \newcommand{\apexp}[1] {\mbox{apex}^{\,#1}\,}
  \newcommand{\rd}[1] {\mbox{rd}_{\,#1}\,} 
  \newcommand{\crd}[1] {\mathcal A_{#1}} 

  \newcommand{\tw}{\mbox{tw}}

  \newcommand{\aw}[1] {{\rm aw}_{#1}\,} 
  \newcommand{\caw}[1] {{\rm caw}_{#1}\,} 

  \newcommand{\SET} {{\rm SET}}
  \newcommand{\SEQ} {{\rm SEQ}}
  \newcommand{\Col} {{\rm Col}}
  \newcommand{\col} {{\rm col}}
  \newcommand{\Ext} {{\rm Ext}}
  \newcommand{\ext} {{\rm ext}}
  \newcommand{\gamupper} {{\overline \gamma}}
  \newcommand{\gamlower} {{\underline \gamma}}

\begin{document}

\author{Valentas Kurauskas\footnote{Vilnius University, Faculty of Mathematics and Informatics, Didlaukio 47, LT-08303 Vilnius, Lithuania.\vspace{-1cm}}
}
\maketitle
\begin{abstract}
    Let $\ex \cb$ be a 
    minor-closed class of graphs with a set $\cb$ of
    minimal excluded minors. 
    Kurauskas and McDiarmid (2012) studied classes
    $\ca$ of graphs that have at most $k$
    disjoint minors in $\cb$, 
    that is, at
    most $k$ vertex disjoint subgraphs
    with a 
    minor in $\cb$. 
    Denote by $\ca_n$ the class $\ca$ restricted to graphs on
    the vertex set $\{1,2,\dots,n\}$.
    In the case when all graphs
    in $\cb$ are 2-connected and $\ex \cb$ excludes some fan
    (a path 
    together with a vertex 
    joined to each vertex on the path),
    they determined the asymptotics of $|\ca_n|$
    and properties of typical graphs in $\ca_n$ as $n \to \infty$. 
    In particular, they showed that 
    all but an exponentially small proportion of
    $G \in \ca_n$
    contain a set $S$
    of 
    $k$
    vertices
    such that $S$ is a $\cb$-blocker, 
    i.e., $G - S \in \ex \cb$.  

    Here we consider 
    the case when $\ex \cb$ contains
    all fans. Firstly, for good enough $\cb$ we obtain results on 
    asymptotics of $|\ca_n|$.
    For example, we give a sufficient condition for the sequence $y_n = (|\ca_n|/n!)^{1/n}$
    to have a limit (a growth constant) as $n \to \infty$. 
    A $\cb$-blocker $Q$ of $G$ is redundant if for each $x \in Q$,
    $Q\setminus\{x\}$ is still a $\cb$-blocker. 
    Let $R_n$ be a graph drawn uniformly at random from $\ca_n$.
    For large enough constant $k$
    we show that the upper limit of $y_n$ is realised by the subclass of graphs 
    that have a redundant $\cb$-blocker $Q$ of size $2k+1$, and there are $n' \to \infty$ 
    such that $R_{n'}$ has no $\cb$-blocker smaller than $2k$ with probability $1 - e^{-\Omega(n')}$.
%
    Secondly, we 
    explore the structure
    of graphs that have at most $k$ disjoint minors $K_4$ (i.e., $\cb = \{K_4\})$.
    For $k = 0$ this is the class 
    of series-parallel graphs.
    For $k = 1, 2, \dots$ we show that there are constants $c_k, \gamma_k$, such that $|\ca_n| = c_k  n^{-5/2} \gamma_k^n n! (1+o(1))$. We prove that the random graph $R_n$ with probability $1 - e^{-\Omega(n)}$ has 
    a redundant $\{K_4\}$-blocker $Q$ of size $2k+1$ and each vertex of $Q$ has a linear degree.
    Additionally, we consider the case $\cb = \{K_{2,3}, K_4\}$ related to outerplanar graphs.
\end{abstract}



\section{Introduction}\label{sec.introK4}

In this paper calligraphic letters such as $\ca, \cb, \dots$ will denote classes of objects, mostly classes of labelled graphs or graphs where some vertices are distinguished and/or unlabelled; these classes will always be closed under isomorphism. 
We denote by $\ca_n, \cb_n, \dots$ the respective classes restricted to objects (graphs) of size $n$ with labels $[n] = \{1, \dots, n\}$.

A class of graphs is called \emph{proper}, if it is not the class of all graphs. Each proper minor-closed class of graphs $\ca$ is \emph{small}, that is,
the supremum as $n \to \infty$ of the sequence
\begin{equation} \label{eq.small}
 \left( \frac {|\ca_n|} {n!} \right)^{1/n}
\end{equation}
is finite \cite{nstw06}, see also \cite{dn10}. If the above sequence converges to $\gamma \in [0, \infty)$, we say that $\ca$
has a \emph{growth constant} $\gamma = \gamma(\ca)$.

A minor-closed class of graphs is called \emph{addable}, if 
each excluded minor is 2-connected.
McDiarmid, Steger and Welsh \cite{msw05} showed that
any proper addable minor-closed class 
has a growth constant, 
further properties for such classes were obtained by McDiarmid \cite{cmcd09}. Bernardi, Noy and Welsh \cite{bnw10} asked whether every proper minor-closed class $\ca$ of graphs has a growth constant. 

For any class of graphs $\ca$
we denote the upper and lower limits of (\ref{eq.small}) by $\gamupper(\ca)$ and $\gamlower(\ca)$
respectively.  Also, let $\rho(\ca)$ denote the radius of convergence of the exponential
generating function $A$ of $\ca$. Of course, $\gamupper(\ca) = \rho(\ca)^{-1}$ (if we assume that $0^{-1} = \infty$).

Given a set of graphs $\cb$, 
a set $Q \subseteq V(G)$ is called a 
\emph{$\cb$-blocker} (or a $\cb$-minor-blocker) for a graph $G$ if $G - Q \in \ex \cb$, i.e., $G-Q$
has no minor in the set $\cb$. 
We call a $\cb$-blocker $Q$ of a graph $G$ \emph{redundant} (``$0$-redundant'', in the terminology of \cite{cmcdvk2011}) if for each vertex $v \in Q$ the set $Q\setminus \{v\}$ is still a $\cb$-blocker for $G$. We denote the class of graphs that have a redundant $\cb$-blocker of size $k$ by $\rd k \cb$. 
For a graph $H$, we will often abbreviate $\ex \{H\}$ to $\ex H$, $\rd k \{H\}$ to $\rd k H$, etc.


Let $\apexp k \ca$ denote the class of all graphs such that by deleting at most $k$ vertices we may obtain a graph in $\ca$.
Also, given a positive integer $s$
call a graph $G$ an \emph{$s$-fan} if $G$ is a union of a complete bipartite graph with parts $A$ and 
$B$, where $|A| = s$, and a path $P$ with $V(P) = B$. We call $1$-fans simply \emph{fans}. 
Given a positive integer $k$ and a set of graphs $\cb$ we denote by $k \cb$ the class of graphs
consisting of $k$ vertex disjoint copies of graphs in $\cb$ (with repetitions allowed).
Thus $\ex (k+1) \cb$ is the class of graphs that do not have
$k+1$ vertex disjoint subgraphs $H_1, \dots, H_{k+1}$, each with a minor in $\cb$.

A classical result related to our topic is the generalisation of Erd\H{o}s and P\'{o}sa
theorem by Robertson and Seymour \cite{rs86}: given a planar graph $H$
and any positive integer $k$ there is a number $f(k)$, such that any graph, that has at most
$k$ disjoint minors $H$ has an $H$-blocker of size at most $f(k)$.
Another famous result by Robertson and Seymour says that each minor-closed class
can be characterised by a finite set $\cb$ of \emph{minimal excluded minors}, i.e. $\ca = \ex \cb$ (see, e.g., \cite{diestel}).

The following 
theorem was 
proved in \cite{cmcdvk2012} (see also \cite{mk-cmcd-09, cmcdvk2011}).
\begin{theorem} \label{thm.cmcdvk2012}
  Let $\ca$ be a proper addable minor-closed class of graphs, with a set $\cb$ of minimal excluded minors. 
  If $\ca$ does not contain all fans, then for each positive integer $k$, as $n \to \infty$
\begin{equation} \label{eqn.main1}
  |(\ex (k\!+\!1)\cb)_n| = (1+e^{-\Theta(n)}) |(\apexp {k} \ca)_n|. 
\end{equation}
\end{theorem}
Suppose $\ex \cb$ is addable but contains all fans. Then the class $\apexp k (\ex \cb) \subseteq \ex (k+1) \cb$ still seems a natural candidate to be the dominating
subclass of $\ex (k+1)\cb$.
However, it was shown in \cite{cmcdvk2012} that for such $\cb$ the above theorem fails, at least for large~$k$. 
Our first theorem shows that a very different subclass determines the convergence radius of $\ex (k+1) \cb$, namely,
the class $\rd {2k+1} \cb$. Clearly, $\rd {2k+1} \cb \subseteq \ex (k+1) \cb$: if $Q$ is a redundant blocker for
$G$ and $|Q| = 2k+1$ then each subgraph of $G$ with a minor in $\cb$ uses at least two vertices of $Q$, 
so we can find no more than $k$ disjoint such subgraphs. 
%
\begin{theorem} \label{thm.main1}
    Let $\ca$ be a proper addable minor-closed class of graphs, with a set $\cb$ of minimal excluded minors
    and growth constant $\gamma$. 
    Suppose $\ca$ contains all fans, but not all $2$-fans,
    nor all complete bipartite graphs $K_{3,t}$.
    
    Then there is a positive integer $k_0 = k_0(\cb)$ such that the following holds. 
    Let $k$ be a positive integer. If $k \ge k_0$,
    \[
    \rho(\ex (k+1) \cb )  = \rho (\rd {2k+1} \cb) < \rho\left( (\ex (k+1) \cb) \cap \apexp {2k-1} \ca \right).
    \] 
    If $k < k_0$, the class $\ex (k+1) \cb$ has a growth constant $2^k \gamma$.
    Furthermore, if $\rho(\rd {2k+1} \cb)^{-1} < 2^k \gamma $ then (\ref{eqn.main1}) holds.
\end{theorem}

For $t \ge 4$ we denote by $W_t$ a wheel graph on $t$ vertices.
Some examples of classes $\ca$ for which Theorem~\ref{thm.main1} applies
are $\ex K_4$, $\ex K_{2,t}$ for $t \ge 3$, $\ex W_5$, and $\ex \{ K_{3,t}, F_s\}$, where $t \ge 2$, $s \ge 5$ and $F_s$
is a 2-fan on $s$ vertices. The conditions of Theorem~\ref{thm.main1} are not satisfied for, say, $\ca = \ex W_6$. 
We believe that these conditions 
can be weakened. 
For example, it may be possible 
to drop the requirement that $\ca$ does not contain all graphs $K_{3,t}$,
to get an assumption similar to the one of
Theorem~\ref{thm.cmcdvk2012}. We also believe that for $k \ge k_0$ 
the fraction of graphs in $(\ex (k+1) \cb)_n$ that are not in $(\rd {2k+1} \cb)_n$
is exponentially small.

Let $\ca$ be an addable class of graphs. In \cite{msw05}, two properties are fundamental
in the proof that $\ca$ has a growth constant. First, the class $\ca$ is 
\emph{decomposable},
meaning that $G$ is in $\ca$ if and only if each component of $G$ is. Second, $\ca$ is \emph{bridge-addable}, i.e., it is closed
under adding bridges between distinct components.
Classes $\ex (k+1) \cb$ are bridge-addable, but not decomposable.
Theorem~\ref{thm.main1} reduces the problem of
proving that a growth constant of $\ex (k+1) \cb$ exists, to the analogous problem 
for the class $\rd {2k+1} \cb$. 
Graphs in $\rd {2k+1} \cb$
with a fixed redundant blocker can be represented by a decomposable class of coloured graphs, see Section~\ref{sec.gcgen}.

With stronger conditions on $\ca$, this allows us to prove the following. 
\begin{theorem} \label{thm.gc}
     Let $k$ be a positive integer and let $\ca$ be a proper addable minor-closed class of graphs with a set $\cb$ of minimal excluded minors. 
    Suppose each graph in $\cb$ is 3-connected,
    $\ca$ does not contain all 2-fans, nor all complete bipartite graphs $K_{3,t}$, nor all wheels.
    Then $\ex (k+1) \cb$ has a growth constant.
\end{theorem}
Classes $\ca$ that satisfy the condition of Theorem~\ref{thm.gc} are, for instance, $\ex K_4$ and $\ex W_5$, but not $\ex K_{2,3}$.

Recall that series-parallel graphs are exactly the class $\ex K_4$. Asymptotic counting formulas and other
properties of series-parallel and outerplanar graphs were obtained by Bodirsky, Gim\'{e}nez, Kang and Noy \cite{momm05}; the degree distribution
was studied by Bernasconi, Panagiotou and Steger \cite{bps08} and Drmota, Noy and Gim\'{e}nez \cite{dng2010}. 
Our next main result concerns
the number 
of graphs, not containing a minor isomorphic to $k+1$ disjoint
copies of $K_4$ (i.e., $\cb = \{K_4\}$). 
\begin{theorem} \label{thm.K4}
    Let $k$ be a positive integer. We have
          \[
          |(\ex (k+1) K_4)_n| = (1 + e^{-\Theta(n)}) | (\rd {2k+1} K_4)_n|.
          \]
         There are constants $c_k > 0$ and $\gamma_k > 0$, such that 
         \[
           |(\rd {2k+1} K_4)_n| = c_k n^{-5/2} n! \gamma_k^n (1 + o(1)).
         \]
         Furthermore, $\gamma_1 = 23.5241..$. 
\end{theorem}
The proof of the above theorem yields the following facts about the structure of typical graphs without a minor isomorphic to $(k+1)K_4$. 
Given a class of graphs $\ca$, we write $R_n \in_u \ca$ to mean that $R_n$ 
is a uniformly random graph drawn from $\ca_n$.
\begin{theorem}\label{thm.K4struct}
    Let $k$ be a positive integer and let $R_n \in_u \ex (k+1) K_4$. 
    \begin{enumerate}[{\hskip 5 mm} (a)]
        \item There is a constant $a_k > 0$, such that 
          with probability $1 - e^{-\Theta(n)}$, the graph $R_n$ has a unique redundant 
          $K_4$-blocker $Q$ of size $2k+1$, each vertex in $Q$ has degree at least $a_k n$,
          and any $K_4$-blocker $Q'$ with $|Q \setminus Q'| > 1$, 
          has at least $a_k n$ vertices.
      \item The probability that $R_n$ is connected converges to $p_k = A(\gamma_k^{-1})$
          where $A$ is the exponential generating function of $\ex K_4$ and $\gamma_k$ is as in Theorem~\ref{thm.K4}.
    \end{enumerate}
\end{theorem}
Let us point out that the complete asymptotic distribution of the `fragment' graph of $R_n$ ($R_n$ minus its largest component)
and the asymptotic Poisson distribution of the number of components in $R_n$
can be easily obtained (in terms of $A$ and $\gamma_k$) using results from \cite{cmcd09}. Furthermore, the expected number of vertices not in the largest component of $R_n$ is $O(1)$,
this holds more generally for random graphs from any bridge-addable class \cite{cmcd09}. We provide an approach to evaluate $\gamma_k$, $k = 1, 2, \dots$ numerically to arbitrary precision. Then $p_k$ can be numerically evaluated using results of \cite{momm05}.

The proof of Theorem~\ref{thm.K4} is much more complicated than the proof in the case $\cb = \{K_3\}$ in \cite{cmcdvk2011} or
the more general Theorem~\ref{thm.cmcdvk2012}. When $\ex \cb$ does not contain all fans, 
a random graph from $(\ex (k+1) \cb)_n$ essentially consists of
a random graph in $\ex \cb$ on $n-k$ vertices and $k$ apex vertices with their neighbours chosen independently at random, each
with probability $1/2$ \cite{cmcdvk2012}. 
Meanwhile, if $Q$ is a redundant blocker for $G \in \ex (k+1) K_4$, then the
possible
neighbours of a vertex $v \in Q$ depend on the series-parallel graph $G-Q$. To solve this, we  obtain decompositions
of the dominating subclass of $\rd {2k+1} \cb$ into tree-like structures
and analyse the corresponding generating functions. 

In our last result we look at classes $\ex (k+1) \{K_{2,3}, K_4\}$. For $k=0$, this corresponds to outerplanar graphs. 
\begin{theorem}\label{thm.outerplanar}
    Let $\cb = \{K_{2,3}, K_4\}$.
    The class $\ex (k+1) \cb$ has a growth constant $\gamma_k'$ for each $k = 1, 2, \dots$.
    We have
    \[
    \gamma_1' = \gamma(\apex(\ex \cb))  =  2 \gamma(\ex \cb) > \gamma(\rd 3 \cb)
    \]
    and for a positive constant $c$
    \[
    | (\ex 2 \cb)_n|  = c n^{-3/2} \gamma_1'^n n! (1 + o(1)).
    \]
    However, for any $k \ge 2$
    \[
      \gamma_k' = \gamma(\rd {2k+1} \cb) > \gamma(\apexp k (\ex \cb))
    \]
    and
    \[
    | (\ex (k+1) \cb)_n| = e^{\Omega(n^{1/2})} \gamma_k'^n n!
    \]

    The first few values are $\gamma_1' = 14.642..$, $\gamma_2' = 34.099..$, $\gamma_3' = 130.023..$,
    and for $k \ge 2$ $\gamma_k'$ admits a closed-form expression.
\end{theorem}
The last theorem shows that Theorem~\ref{thm.main1} does not hold in general with $k_0 = 1$.
The unusual subexponential factor for $k \ge 2$ shows up because the underlying structure of typical graphs
in $\rd {2k+1} \{K_{2,3}, K_4\}$ is ``path-like'', whereas it is
``tree-like'' for graphs in $\rd {2k+1} K_4$, see Section~\ref{subsec.illustration} below.


Half of the paper consists of structural results which yield general theorems
(i.e., Theorem~\ref{thm.main1} and Theorem~\ref{thm.gc}) with rough asymptotics. 
The other half is a study of the specific case $\ex (k+1) K_4$, which requires
knowledge of the structure of the class of series-parallel graphs and
analysis of specific generating functions, but yields much sharper conclusions (Theorem~\ref{thm.K4}).

In Section~\ref{sec.struct} we prove our key structural lemmas and Theorem~\ref{thm.main1}. 
Section~\ref{sec.gcgen} is similar, using a superadditivity argument as in \cite{msw05}, we prove Theorem~\ref{thm.gc} there. In Section~\ref{sec.part2}, we explore the rich structure of the classes
$\rd {2k+1} K_4$, which we then translate into generating functions and apply analytic combinatorics to get the growth constant when $k = 1$. In Section~\ref{sec.unrooting} we count graphs obtained from unrooted Cayley trees where edges, internal vertices and leaves
may be replaced by graphs from different classes. Then, in Section~\ref{sec.unrootedstructure} we complete the proof of Theorem~\ref{thm.K4} and present Figure~\ref{fig.Ctree2} illustrating the structure of typical graphs in $\ex 2 K_4$ together with a short intuitive explanation. In Section~\ref{sec.outerplanar} we prove Theorem~\ref{thm.outerplanar}. Finally, in Section~\ref{sec.conclusion} we discuss open questions that arise from this work and give some concluding remarks.

\section{Definitions}

\subsection{Definitions for coloured graphs}
\label{sec.def}

Let $t \ge 0 $ be a fixed integer. We will consider $\{0,1\}^t$-coloured graphs $G$ where each vertex
$v \in V(G)$ is assigned a colour 
\[
\col(v) = \col_G(v) = (\col_1(v), \dots, \col_t(v)) \in \{0,1\}^t.
\]
 We say that $v\in V(G)$ has colour $i$ if $\col_i(v) = 1$.
 We denote by $\Col(v) = \Col_G(v) = \{k: c_k(v) \ne 0\}$ the set of all colours of $v$, similarly let $\Col(G)$ 
 be the 
 union of $\Col_G(v)$ for all $v \in V(G)$. 
 Also, denote by $N(G)$ the uncoloured graph obtained by removing all colours from $G$. 
 Whenever $t$ is clear from the context or not important, we will call
 $\{0,1\}^t$-coloured graphs just \emph{coloured graphs} or simply \emph{graphs}.
We call a vertex $v \in V(G)$ \emph{coloured} if $\Col_G(v) \ne \emptyset$. 


Let $G$ be a $\{0,1\}^t$-coloured graph. Given a set $L=\{s_1, \dots, s_t\}$
such that $s_1, \dots,s_t \not \in V(G)$ and $s_1 < \dots < s_t$, we
can obtain an (uncoloured) graph $G^L$ on vertex set $V(G) \cup L$ by connecting $s_i$ to each vertex $v \in V(G)$ that has colour $i$.
We call $G^L$ an \emph{extension of $G$}. We denote by $\Ext(G)$ the set of all extensions $G^L$ of $G$ such that $|L|=t$, 
and denote by $\ext(G)$ an arbitrary representative of $\Ext(G)$.

For a $\{0,1\}^t$-coloured graph $G$ we
define the contraction operation in the standard way (see, e.g., \cite{diestel}) with the addition that the vertex $w$ obtained from contracting an edge $uv \in E(G)$ has colours $\Col(w) = \Col(u) \cup \Col(v)$.
A $\{0,1\}^t$-coloured graph $H$ is a \emph{subgraph} of $G$ if $H$ is a subgraph of $G$, if the colours are ignored, and for each $v \in V(H)$ we have $\Col_H(v) \subseteq \Col_G(v)$.
$H$~is  a \emph{coloured minor} of $G$ if it can be obtained by contraction and subgraph operations from $G$. 

When a (coloured) graph $G$ has $V(G) \subseteq [n]$ for some positive integer~$n$, we will usually assume that 
the new vertex $w$ resulting from the contraction of an edge $e = xy$ has label $\min(x,y)$, so that
$V(G / e) \subseteq [n]$.
For a (coloured) graph $G$ and $J \subseteq E(G)$ we will denote by $G / J$ the graph resulting from the contraction of all of the edges in~$J$. The operation $G / J$ corresponds to
a partition of $V(G)$ into a set of ``bags'' $\{Bag(v): v \in V(G / J) \}$ where $Bag(v)$ is the set of vertices
that contract to $v$. We call a subgraph $H$ of $G$ \emph{stable} with respect to contraction of $J$ in $G$ if no pair of vertices of $H$ is contracted into the same bag.

We say that two $\{0,1\}^t$-coloured graphs $G'$ and $G''$ are \emph{isomorphic} if there is a bijection $f: V(G') \to V(G'')$ such that $x y \in E(G')$ if and only if $f(x) f(y) \in E(G'')$ and $\Col_{G'}(x) = \Col_{G''}(f(x))$ for each $x \in V(G')$.



%

 For a $\{0,1\}^t$-coloured graph $G$, we say that $S \subseteq V(G)$ has colour $c$ 
 if $c \in \Col_G(v)$ for some $v \in S$. We say that $G$ has colour $c$ if $V(G)$ does. 
 For a vertex $v \in V(G)$ we let $\Gamma(v) = \Gamma_G(v)$ denote the set of neighbours of $v$ in $G$.
It will be convenient to call the colours 1, 2 and 3 \emph{red, green and blue} respectively.

Let $C$ be the set of cut points of $G$ and let $\cb$ be the set of its blocks. Fix $r \in V(G)$.
Then the tree $T_r$ with vertex set  $C \cup \{r\} \cup \cb$ and edges given by $uB$ where $u \in C\cup\{r\}$, $B \in \cb$ and $u \in V(B)$ will be called a \emph{rooted block tree} of $G$, rooted at $r$. (This is a minor modification of the usual block tree, see~\cite{diestel}.) We call graphs that are either 2-connected or isomorphic to $K_2$ \emph{biconnected}.

For a graph $G$ and a set $S$, we write $G \cap S = G[V(G) \cap S]$. For two graphs $G_1 = (V_1, E_1)$ and $G_2 = (V_2, E_2)$ we write $G_1 \cup G_2 = (V_1 \cup V_2, E_1 \cup E_2)$ and $G_1 \cap G_2 = (V_1 \cap V_2, E_1 \cap E_2)$. 

\subsection{Analytic combinatorics}

We will apply the ``symbolic method'' of Flajolet and Sedgewick~\cite{fs09} to study the asymptotic
number of graphs from various classes. 

In this paper we follow the notational conventions used in~\cite{fs09}. 
The size of a graph $G$ is the number of labelled vertices, while $V(G)$
refers to the set of all vertices of $G$, including the unlabelled (pointed) ones.
The exponential generating function of $\ca, \cb, \dots$ is denoted $A(x), B(x), \dots$ respectively.
For instance,
$A(x) = \sum_{n=0}^{\infty} \frac{|\ca_n|} {n!} x^n$.
By $\cz$ we denote the class of graphs consisting of a single vertex with a label, such that $Z(x) = x$.

We use the notation $\ca + \cb$, $\ca \times \cb$, $\ca(\cb)$ to denote the class of graphs obtained
by the (disjoint) union, labelled product and composition operations respectively, see \cite{fs09}.
For a positive integer $k$, $\ca^k$ denotes the class consisting of a sequence of $k$ disjoint members from $\ca$,
and we define $\ca^0$ to be the class with exponential generating function $A^0(x) = 1$.
We also refer to \cite{fs09} for the formal definition of the class $\SET(\ca)$ (obtained by taking
arbitrary sets of elements of $\ca$ and appropriately relabelling), the class $\SEQ(\ca)$ (obtained
by taking any ordered sequence of elements of $\ca$ and appropriately relabelling), and classes
$\SET_{\ge k}(\ca)$ (sets of at least $k$ elements) and $\SEQ_{\ge k}$ (sequences of at least $k$ elements). Given a positive integer $k$, we will denote by $k \times \ca$  a combinatorial class with the counting sequence $(k |\ca_n|, n=0,1,\dots)$.

To denote dependence of a class $\ca$ (or a generating function $A$) on a parameter $l$ we will use either superscript $\ca^{<l>}$, $\ca^l$ or a subscript $\ca_l$.\footnote{When this coincides with the notation for the elements $\ca_n$ in $\ca$ with labels in $[n]$ or with a power of a class $\ca^k$, the meaning should be determined from the context.}

If  $\ca$ and $\cb$ have identical counting sequences $|\ca_n| = |\cb_n|$ for $b = 0, 1, \dots$ 
we call $\ca$ and  $\cb$ \emph{combinatorially isomorphic} and write $\ca = \cb$. We note that most of the
decomposition results of Section~\ref{sec.part2} and onwards yield a stronger kind of isomorphism than just the combinatorial one: we prove unique decompositions of graphs from one class into unions of subgraphs with disjoint sets of labels from other classes. This is important since many of our proofs rely on the structure of graphs.

\section{Structural results for $\ex (k+1) \cb$}
\label{sec.struct}

\subsection{The colour reduction lemma}
\label{subsec.colreduction}

The following simple lemma will be very useful in our structural proofs. 
\begin{lemma}\label{lem.apex2P3}
    Let $l$ be a non-negative integer. Let $G$ be a $\{0,1\}^2$-coloured graph. Suppose $G$ does not have $l+1$ disjoint connected subgraphs
    containing both colours. Then there is a set $S$ of at most $l$ vertices such that each component of $G - S$ has at most one colour.
\end{lemma}
%
%
\begin{proof}
    For two new vertices $s, t \not \in V(G)$ consider the extension $G' = G^{\{s,t\}}$. $G'$ has $l+1$ internally disjoint paths from $s$ to $t$ if and only if $G$ has $l+1$ disjoint connected subgraphs containing both colours. By Menger's theorem 
 we may find a set $S$ of at most $l$ vertices in $V(G)$ such that $S$ separates $s$ from $t$ in $G'$, and hence each component of $G - S$ can have at most one colour.
\end{proof}

\bigskip

Given an integer $s$ and a graph $G$, we define its \emph{apex width of order $s$}, denoted $\aw s G$
as the maximum number $j$ such that $G$ has a minor $H$ on $j+s$ vertices where
$H$ is a union of a tree $T$ with $|V(T)| = j$ and a complete bipartite graph with parts $V(T)$ and $V(H) \setminus V(T)$.
For a class of graphs $\ca$ we define $\aw s (\ca)$ to be the supremum of $\aw s (G)$ over $G \in \ca$.
In this paper we will only use the parameter $\aw 2$. For example, it is easy to check that 
its value
for classes $\ex K_4$, $\ex K_{2,t}$ 
and $\ex K_5$ is $2$, $t-1$ and $\infty$ respectively. In Section~\ref{subsec.finite_aw} below we give
a condition to check if $\aw j (\ca)$ is finite.

Given integers $s$ and $t$, $1 \le s \le t$ and a $\{0,1\}^t$-coloured graph $G$,
define its \emph{coloured apex width of order $s$}, denoted $\caw s (G)$, as the maximum
number $j$, for which there are $j$ pairwise disjoint connected subgraphs $H_1, \dots, H_j$ of $G$
that have at least $s$ common colours, i.e., $|\Col(H_1) \cap \dots \cap \Col(H_j)| \ge s$.
For a class of coloured graphs $\ca$ we define $\caw s (\ca)$ as the supremum of $\caw s (G)$ over the graphs $G \in \ca$.

We state one of our key structural lemmas next. We assume that all graphs here have vertices in $\mathbb{N}$.
\begin{lemma}[Colour reduction lemma] \label{lem.colreduction}
    Let $t \ge 2$ be an integer and let $G$ be a connected $\{0,1\}^t$-coloured graph.
    Suppose $\caw 2 (G) \le j$, for some non-negative integer $j$.
    
    Then there is a connected $\{0,1\}^t$-coloured graph $G'$ and a set $J \subseteq E(G')$ of size at most
    $(j + 1)^{t-1} - 1$ such that a) each component of $G' - J$ has at most one colour,
    b) $G'/J = G$ and c) each component of $G'-J$ is stable with respect to contraction of $J$ in $G'$.
\end{lemma}

\begin{proof}
    We use induction on $t$. For $t=2$ by Lemma~\ref{lem.apex2P3}
    there is a set $B$ of at most $j$ vertices in $G$ such that each component of 
    $G - B$ has at most one colour. Denote by $V_{green}$ the set of vertices that belong to a component of $G - B$ that has the green colour.

    Let $G_0 = N(G[B])$.
    For each vertex $v \in B$ take a new vertex $v' \not \in V(G)$; let $B' = \{v': v \in B\}$.
    Now define a matching $J = \{v v' : v \in B\}$ on $|B| \le j$ edges.    
    Consider
    the $\{0,1\}^2$-coloured graph $G_1$ on the vertex set $B \cup B'$, with edges $E(G[B]) \cup J$
    and colours 
    \[
    \Col_{G_1}(v) = \{red\} \cap \Col_G(v) \quad\mbox{and}\quad \Col_{G_1}(v') = \{green\} \cap \Col_G(v).
    \]
    for each $v \in B$.

    Now let $G'$ be the union of $G_1$, $G - B$ and the set of edges 
    $E_1 \cup E_2$ defined as follows: 
    \begin{align*}
        & E_1 = \{ v' x : v \in B, x \in V_{green}, \mbox{ and } vx \in E(G) \};
        \\  & E_2 = \{ v x:  v \in B, x \in V(G) \setminus (B \cup V_{green}), \mbox{ and } vx \in E(G) \}.
    \end{align*}
    In words, $G'$ is obtained from $G$ by splitting each vertex $v \in B$, so that one of the new vertices
    inherits the green colour of $v$ (if it had that colour) and all neighbours of $v$ in $V_{green}$ 
    while the other vertex inherits the rest of the neighbours of $v$. Obviously, $G' / J = G$ if we make sure that the newly created
    vertices $v'$ have larger labels than those in  $V(G)$. 
    By our construction, $J$ separates $B$ and $B'$ in $G_1$ 
    and each component of $G' - J$ containing green colour can have vertices only in $V_{green} \cup B'$,
    thus each component of $G' - J$ has at most one colour.

    Consider a component $C$ of $G' - J$. Since $J$ is a matching and each edge of $J$ 
    is between different components of $G' - J$, the contraction of $J$ in $G'$ may not put two vertices of $C$ into
    the same bag.
    This completes the proof for the case $t = 2$.

    Suppose $t > 2$. Assuming that we have proved the claim for any $t' < t$, we now prove it with $t' = t$.
    Delete the colour $t$ from $G$ to get a $\{0,1\}^{t-1}$-coloured graph $G_1$.
    By induction, there is a $\{0,1\}^{t-1}$-coloured graph $G_1'$ and a set of edges $J_1$,
    such that $|J_1| \le (j+1)^{t-2}-1$, each component of $G_1' - J_1$ has at most one colour, it is
    stable with respect to contraction of $J_1$ in $G_1'$ and
    $G_1'/J_1 = G_1$.
    For a vertex $v$ of $G_1'/J_1$,
    denote by $Bag(v)$ the set of vertices of $G_1'$ that contract to $v$. 

    Now let us return the colour $t$ back as follows. For each vertex $v$ of $G$ that
    has colour $t$ in $G$ pick one vertex 
    $v' \in Bag(v) \subseteq V(G')$
    and add the colour $t$ to $\Col_{G'}(v')$: we obtain a $\{0,1\}^t$-coloured graph $G_2$
    such that $G_2 / J_1 = G$
    and each component of $G_2 - J_1$
    can have at most two colours.

    Now since each component $C$ of $G_2 - J_1$ is 
    stable with respect to contraction of $J_1$ in $G_2$,
    we have $\caw 2 (C) \le \caw 2 (G) \le j$. Thus, by symmetry, we can apply the already proved case $t = 2$ 
    of the lemma to each such $C$ to obtain a $\{0,1\}^t$-coloured 
    graph $C'$ and a set $J_C \subseteq E(C')$ of at most $j$ edges, such that $\Col(C) = \Col(C')$, every component of $C' - J_C$ has
    at most one colour, is stable with respect to contraction of $J_C$ in $C'$ and 
    $C' / J_C  = C$. 
    We assume that the labels for the new vertices are chosen 
    so that they are larger than any label of $G_2$ and $V(C_1')$ and $V(C_2')$ remain disjoint for distinct components $C_1$ and $C_2$ of $G_2$.

    For
    any $v \in V(C' / J_C)$ denote by $Bag_C(v)$ the set of all vertices of $C'$ that contract to $v$.
    Now for any edge $e = x y \in J_1$, let $C_x$ and $C_y$ be the components of $G_2 - J_1$ containing
    $x$ and $y$ respectively, and define $e' = x'y'$ where $x'$ and $y'$ are any vertices in 
    $Bag_{C_x}(x)$ and $Bag_{C_y}(y)$ 
    respectively. Set $J = \{e': e \in J_1\} \cup \bigcup_{C} J_C$ 
    where the union is over the components of $G_2 - J_1$. Finally, 
    let $G'$ be the graph obtained  by adding $J$ to the union of the disjoint graphs $C'$, for each component $C$ of $G_2 - J_1$.
    
    Clearly, each component of $G' - J$ has at most one colour.
    Now consider the operation $G'/J$ in two stages: in the first stage
    contract all edges $\bigcup_{C} J_C$, in the second stage contract the edges in $J_1$.
    Then at the first step we obtain the graph $G_2$, and in the second step, we obtain the graph $G$.
    Furthermore, if $\tilde{C}$ is a component of $G' - J$, then it is a component of $C' - J_C$ for some $C$.
    $\tilde{C}$ is stable with respect to contraction of $J_C$ in $C'$ and $C$ is stable
    with respect to contraction of $J_1$ in $G_2$, therefore $\tilde{C}$ is stable with respect to contraction of $J$ in $G'$.


    Also, since $G$ is connected $G_2$ has at most $|J|+1$ components. Therefore $|J| \le |J_1| + (|J_1| + 1) j \le (j+1)^{t-1} - 1$.
\end{proof}

\subsection{Redundant blockers}

\bigskip

The main result of this section is the following lemma 
(cf. Lemma~1.6 of~\cite{cmcdvk2012}). 
\begin{lemma}\label{lem.red2}
   Let $k$ be a positive integer, 
   and let $\ca$ be a proper addable minor-closed class with a set $\cb$
   of minimal excluded minors.
   Suppose that $\aw 2 (\ca)$ is finite. 

   Then there is a constant $c = c(\cb, k)$ such that any graph in $\ex (k+1) \cb$
   has a $\cb$-blocker $Q$ of size at most $c$ and a set $S \subseteq Q$ of 
   size at most $2k$ such that any subgraph $H$ of $G$ with $H \not \in \ca$
   that meets $Q$ in at most two points, also meets the set $S$.
\end{lemma}
%
The proof will follow from a slightly more general result, Lemma~\ref{lem.red2gen} below;
we first need a few definitions.
Given a graph $G$, a set of graphs $\cb$ and a set $Q \subseteq V(G)$, we say that $Q$ is
a $(j, \cb)$-\emph{blocker} of $G$ if $G$ contains no subgraph $H$, such that $H \not \in \ex \cb$
and $|V(H) \cap Q| \le j$. We say that 
$Q$ is a $(j, s, \cb)$-\emph{blocker} of $G$ if (a) $Q$ is a $\cb$-blocker for $G$ and (b) $G$ does not contain
$s$ pairwise disjoint subgraphs $H_1, \dots, H_s \not \in \ex \cb$ 
where each $H_i$, $i = 1, \dots, s$ has at most $j$ vertices in $Q$.


A graph $H$ will be called \emph{$\cb$-critical} if $H \not \in \ex \cb$ but $H' \in \ex \cb$ for any $H' \subset H$. 
Notice that if each graph in $\cb$ is 2-connected, then so is each $\cb$-critical graph.

As in \cite{cmcdvk2012}, we will use normal trees for our proofs.
Let $G$ be a graph, and let $T$ be a rooted tree on the same vertex set
  $V(G)$, with root vertex $r$.
  The tree $T$ induces a tree-ordering $\le_T$ on $V(G)$ where
  $u \le_T v$ if and only if $u$ is on the path from $r$ to $v$ in $T$.
  $T$ is a {\em normal} tree for $G$
  if $u$ and $v$ are comparable for every edge $uv$ of $G$ (notice that we do not require that $T$ is a subgraph of $G$).
  For $u \in V(T)$ we will denote by $T_u$ the induced subtree of $T$ on vertices $\{v: u \le_T v\}$.
  
  We say that $u$ is above $v$ (and $v$ is below $u$) in $T$ if
  $u <_T v$. 
  Given a graph $G$ and a normal tree $T$ for $G$, for each vertex $v$ of $G$ we define its
  set $AA_T(v)$ of {\em active ancestors} by
\[
  AA_T(v) =
  \left\{ u <_T v :\ \exists z \geq_T v \mbox{ with } uz \in E(G) \right\}.
\]
  and write $a_T(v) = |AA_T(v)|$. Kloks observed (see Theorem 3.1 of \cite{cmcdvk2012} for a proof) that the treewidth $tw(G)$ of a graph $G$ satisfies
\begin{equation}
  \tw(G) = \min_{T \in \ct} \max_{v \in V(G)} \ a_T(v) \label{eqn.kloks}
\end{equation}
  where $\ct$ is the set of all normal trees for $G$.

Finally, denote by $f^{*n}$ the $n$-th iteration of the function $f$, so that $f^{*0}(x) = x$ and $f^{*(n+1)}(x) = f(f^{*n}(x))$ for $n = 0, 1,\dots$.


\begin{lemma}\label{lem.red2gen} 
 Let $k$ be a positive integer, let $\ca$ be a proper addable minor-closed class with
 a set $\cb$ of minimal excluded minors.
 Suppose that $\aw 2 (\ca) \le j$. 
 
 There are positive constants $c_1 = c_1(\cb)$ and $w = w(\cb)$ such that the following holds. Define a function $f: \mathbb{N} \to \mathbb{R_+}$ by $f(q) = j c_1 q^2 + j c_1 w q$.
 Suppose $Q$ is a non-empty $(2, k+1, \cb)$-blocker for a graph $G$. Then there are sets $S, Q' \subseteq V(G)$, such that 
 $S, Q \subseteq Q'$, 
 $|S| \le 2k$, $|Q' \setminus S| \le f^{*k} (|Q|)$ and $Q' \setminus S$ is a $(2, \cb)$-blocker for $G - S$.
 \end{lemma}

 \begin{proofof}{Lemma~\ref{lem.red2}}
    The assumption that $\aw 2(\ca) < \infty$ implies that some planar graph, a 2-fan, is excluded from $\ca$.
    By the theory of graph minors \cite{rs86},
    see also \cite{diestel} and 
    Proposition~3.6 
    of \cite{cmcdvk2012}, there is a constant $c' = c'(\cb,k)$
    such that every graph $G \in \ex (k+1) \cb$ has a $\cb$-blocker $Q_0$ of size at most $c'$ (we may assume $Q_0$ is non-empty).
    Such a set $Q_0$ is clearly also a $(2, k+1, \cb)$-blocker for~$G$.
    Now Lemma~\ref{lem.red2gen} ensures that there is a $\cb$-blocker $Q$ of size at most $c = c(\cb, k) = f^{*k}(c')$
     and a set $S \subseteq Q$ of size at most $2k$ as required.
\end{proofof}

\medskip

\begin{proofof}{Lemma~\ref{lem.red2gen}} 
We use two results from the theory of graph minors of Robertson and Seymour \cite{rs86}:
since $\ca$ excludes a planar graph (a 2-fan on $j+2$ vertices),
the maximum treewidth  over graphs in $\ca$, denoted $w = w(\cb)$, is finite. Furthermore, 
the set $\cb$ is finite.

Let $c_1 = c_1(\cb)$ be the maximum number of components that can be created by removing three vertices from
a $\cb$-critical graph.
Since there is a finite number of graphs in $\cb$, the number $c_1$ is finite (see Lemma~5.2 in \cite{cmcdvk2012}).
For example, we have $c_1(\{K_4\}) = 4$. Since $\ca$ is addable, we have $j \ge 1$ and $c_1 \ge 1$. 
The case $V(G) = Q$ is also trivial (take $S = \emptyset$ and $Q' = Q$), so we will assume $Q \subset V(G)$. 

    
We will prove the lemma by induction on $k$.
    The case $k=0$ is trivial: we may take $Q' = Q$ and $S = \emptyset$. Assuming the lemma holds
    for $0 \le k < k'$, where $k'$ is a positive integer, we prove it for $k = k'$.

    By 
    (\ref{eqn.kloks}), 
the graph $G-Q$ has a (rooted) normal tree $T$ 
with
$\max_v a_T(v) \le w$.
Let $r$ be the root of $T$.
%
Form a set $U$ of all such vertices $v \in V(G) \setminus Q$ for which there are some vertices $x, y \in Q \cup AA_T(v)$ that the subgraph of $G$ induced on $V(T_v) \cup \{x,y\}$ has a minor in $\cb$.
Choose a vertex $u \in U$ with maximum distance from $r$ in $T$.
Let $R = Q \cup AA_T(u)$.

%

Let $P$ be a set of minimum size such that $G[P \cup V(T_u)] \not \in \ex \cb$. Then $|P| \in \{1,2\}$. Fix 
a $\cb$-critical graph $H$, such that $H \subseteq G[P \cup V(T_u)]$.

Consider the graph $G_1 = G[T_u] - u$. This graph consists of some connected components. 
Since $H \cap V(G_1) = H - (P \cup \{u\})$, the graph $H$ has vertices in $t \le c_1(\cb)$ such components. 
Call these components $C_i$, $i = 1,\dots,t$. 
Fix a pair $x,y \in R$ such that $\{x,y\} \cap P = \emptyset$. We claim, that for $i=1,\dots,t$ there is a set of at most $j$ vertices $D_i(x,y)$, such that in $C_i - D_i(x,y)$ no component has edges to both $x$ and $y$.

Let us show why it is true. In the component $C_i$ colour vertices adjacent to $x$ red, and those adjacent to $y$ green
to obtain a $\{0,1\}^2$-coloured graph $C_i'$
(vertices adjacent to both $x$ and $y$ are coloured $\{red, green\}$, and the remaining vertices are coloured $\emptyset$).
Suppose $C_i'$ has $j+1$ disjoint connected subgraphs containing both colours.
Then since $\aw 2 (\ca) \le j$ we would have that $G[ \{x,y\} \cup V(C_i)] \not \in \ca$. 
But this contradicts to the choice of $u$: since $T$ is normal, the vertices of the component $C_i$ must be entirely contained in 
$V(T_{u'})$ for some 
$u' \in V(T_u - u)$, so $u' \in U$. Thus $C_i'$ 
cannot have 
$j+1$ connected subgraphs containing both colours, so
we may apply Lemma~\ref{lem.apex2P3} to find a suitable set $D_i(x,y)$ of size at most~$j$.


Now define sets $S_0, Q_1$ as follows. If $|P|=1$, let $S_0 = P \cup \{ u\}$, otherwise,
let $S_0 = P$. Set
\[
Q_1 = \left( (Q \cup AA(T_u) \cup \{u\}) \setminus S_0 \right) \cup  \bigcup_{i \in [t], x,y \in R\setminus P, x \ne y} D_i(x,y) .
\]
%
Writing $q = |Q|$ and 
considering the cases $|P|=1$ and $|P|=2$ separately we get that 
\[
  |Q_1| \le jc_1(q -1) (q - 2 + w) + q - 1 + w \le  f(q).
\]

If $Q_1$ is a $(2, k, \cb)$-blocker for $G - S_0$, then we can use induction
to find sets $S', Q' \subseteq V(G) \setminus S_0$ such that $S',Q_1 \subseteq Q'$, 
$|S'| \le 2(k-1)$, $|Q'\setminus S'| \le f^{*(k-1)}(|Q_1|) \le f^{*k}(q)$ and $Q' \setminus S'$ is a $(2, \cb)$-blocker for $G - S_0$. Then the lemma follows with $S = S' \cup S_0$ and $Q'$.

It remains to show that $Q_1$ is a $(2, k, \cb)$-blocker for $G - S_0$. Assume this is not true.
%
%
Let $\tilde{H} \subseteq G$ be a $k \cb$-critical subgraph of $G - S_0$ showing that $Q_1$ is not
a $(2, k, \cb)$-blocker: that is $\tilde{H} = H_1' \cup \dots \cup H_k'$, where
$H_i' \not \in \ex \cb$, $i = 1, \dots, k$ are 2-connected and pairwise disjoint;
and for each $i \in [k]$ we have $|V(H_i') \cap Q_1| \le 2$. 

Now $\tilde{H}$ and $H$ may not be disjoint: otherwise $Q$ would not be a $(2, k+1, \cb)$ blocker for $G$.
Let $H'$ be a component of $\tilde{H}$ which shares at least one vertex with $H$. 
Then $V(H') \cap V(H) \subseteq V(H) \setminus S_0 \subseteq V(T_u)$.

Suppose first that $V(H') \cap Q_1$ consists of a single vertex $v$. Note, that we must have $v \in Q \setminus S_0$. 
The graph $H'-v$ 
cannot be entirely contained in $G[V(T_u)]$.
To see why, observe first that by our construction in this case $u \in S_0$. 
Since $H'-v$ is connected it must be entirely contained in one of 
the proper subtrees of $T_u$, but this contradicts our choice of $u$.

Thus $H' - v$ must have a vertex $a$ in $V(T_u - u)$ and a vertex $b$ in $G - (S_0\cup Q_1 \cup V(T_u))$. 
The set $AA_T(u) \subseteq Q_1 \setminus \{v\}$ separates $V(T_u - u)$ from $G - (S_0\cup Q_1 \cup V(T_u)$. On the other hand, there is a path from $a$ to $b$ in the connected graph $H' - v$: this is a contradiction.

Now suppose $H'$ has exactly two vertices $x,y$ in $Q_1$. 

First consider the case where $x,y \in Q \cup AA_T(u)$. 
Let $a$ be a vertex in $V(H) \cap V(H')$. 
It cannot be $a = u$ since in this case we have that $|V(H') \cap Q_1| \ge 3$. It follows that $a \in V(C')$ where $C'$
is a component of $C_i - D_i(x,y)$ for some $i \in \{1,\dots,t\}$.
But $C'$ 
cannot have edges to both $x$ and $y$. This means that either $x$ or $y$ is a cut vertex in $H'$: 
this contradicts the fact that $\tilde{H}$ is $k \cb$-critical.

If $x \in Q$ and $y \in D_i(x',y')$ for some pair $\{x', y'\}$, then suppose that $H' - x$ is contained in $G[V(T_u)]$. By the choice of the set $P$ 
this means that $u \in S_0$.  
But since $T$ is normal, this contradicts the definition of $u$. 
Otherwise, suppose that $H' - x$ has a 
vertex in $G - (S_0 \cup Q_1 \cup V(T_u))$.
Since $AA_T(u) \subseteq Q_1 \setminus \{x,y\}$ we have that $x$ must be a cut point in $H'$: this is a contradiction to the $k\cb$-criticality of $\tilde{H}$.

Finally, consider the case $x \in Q$ and $y = u$. Note, that the only case when $u \not \in S_0$ by our construction is when there is no vertex $z$ such that $G[V(T_u) \cup \{z\}]$ has a minor in $\cb$. Thus $H$ must contain at least one vertex in $G-(S_0 \cup Q_1 \cup V(T_u))$, and
we saw earlier that it has a vertex in $(V(H) \setminus S_0) \subseteq V(T_u)$. 

Again, each path in $H'$ from $V(H') \cap V(T_u)$ to $V(H') \cap (V(G) \setminus (S_0 \cup Q_1 \cup T_u))$ must use $x$, since $AA_T(u) \subseteq Q_1 \setminus \{x,y\}$ separates $T_u$ from the rest of $G-Q$. So we obtain a contradiction to the fact that $H'$ is 2-connected. In all of the cases we obtained a contradiction, so we conclude that $Q_1$ must be a $(2,k,\cb)$-blocker for $G-S_0$. 
\end{proofof}

\subsection{Blockers of size $2k$}
\label{subsec.b2k}

We will need another definition. We call a $\cb$-blocker $Q$ of a graph $G$ a \emph{$(k, j, \cb)$-double blocker}
if there is a set $S \subseteq Q$ of size at most $k$, which is a redundant $\cb$-blocker for $G - (Q \setminus S)$, and $Q \setminus S$ is a
$(j, \cb)$-blocker for $G - S$. Such a set $S$ is called a \emph{special set} of $Q$. 
\begin{lemma}\label{lem.2canddecomp}
    Let  $k, l$ be positive integers and let $\cb$ be the
    set of minimal excluded minors of a proper addable minor-closed class
    and assume that $\cb$ contains a planar graph.
    Then there is a positive constant 
    $w = w(\cb)$
    such that the following holds.

    Suppose $G \in \ex (k+1) \cb$ has 
    a $\cb$-blocker $Q$ of size at most $q$ and a set $S \subseteq Q$ of size at most $l$ such that
    $Q\setminus S$ is a $(2,\cb)$-blocker for $G-S$.
    Then $G$ can be represented
    as the union of two graphs, $G = G_1 \cup G_2$ where 
    \begin{itemize}
        \item the graph $G_1$ has an $(l, 2, \cb)$-double blocker $Q_1 \supseteq Q$ of size at most 
            $q + w+1$,
            such that $S$
            is the special set;
        \item $G_2 \in \apex (\ex k \cb)$ and $Q \subseteq V(G_1) \cap V(G_2) \subseteq Q_1$. 
    \end{itemize}
\end{lemma}

\begin{proof}
    The set $\cb$ contains a planar graph, so by the theory of graph minors, see \cite{diestel, rs86},
    the treewidth of $G - Q$ is bounded by a constant $w = w(\cb)$. Since the claim is trivial in the case $Q = V(G)$
    (take $G_1 = G$, $Q_1 = Q$ and $G_2 = \bar{K}_Q$, where $\bar{K}_Q$ is the empty graph on $Q$),
    we will assume that $Q \subset V(G)$. By the Kloks theorem (\ref{eqn.kloks}),
    $G-Q$ has a normal tree $T$ such that the number of active ancestors satisfies $a_T(v) \le w$ for each $v \in V(T)$.
    Denote by $r$ the root of $T$.

    Let $U$ be the set of all vertices $v \in V(G-Q)$ such that $G[V(T_v) \cup \{x\}] \not \in \ex \cb$ for some $x \in S$.
    If $U = \emptyset$ then $Q$ is itself an $(l, 2, \cb)$-double blocker for $G$, so we may take $G_1 = G$ and $G_2 = \bar{K}_Q$.
    Now assume that $U$ is non-empty. Let $u \in U$ be a vertex with maximum distance in $T$ from the root $r$
    and let $x_0$ be a vertex in $S$ showing that $u \in U$. 
    Write $A = AA_T(u)$, let $G_1 = G[V(T_u) \cup Q \cup A]$, 
    and let $G_2 = G - V(T_u)$. We claim that $G_1$ and $G_2$ are as required.
    
    We have $V(G_1) \cap V(G_2) = Q \cup A$, so $G_1$ and $G_2$ share $|Q \cup A| \le q+w$ vertices. 

    
    We will show that $Q_1 = Q \cup A \cup \{u\}$
    is an $(l, 2, \cb)$-double blocker for $G_1$, and $S$ is its special set.  
    Indeed, 
    using the assumption of the lemma,
    the set $Q_1 \setminus S = (Q \setminus S) \cup A \cup \{u\}$ is a $(2, \cb)$-blocker for $G_1 - S$. 
    Now suppose that $G[V(T_u-u) \cup \{z\}] \not \in \ex \cb$
    for some $x \in S$. Let $H$ be a $\cb$-critical subgraph of $G[V(T_u-u) \cup \{x\}]$. Then, since $T$ is normal, all vertices of the connected graph $H - x$
    must be contained in $V(T_v)$ for some $v$ strictly below $u$ in $T$. 
    This is a contradiction 
    to the choice of $u$. 


    

    Finally observe that if $G_2 - x_0$ contains a minor in $k \cb$ then
    since $V(T_u) \cup \{x_0\}$ and $V(G_2 - x_0)$ are disjoint, $G \not \in \ex (k+1) \cb$. So $G_2 \in \apex ( \ex k \cb)$ is as claimed.
\end{proof}

\medskip

In the proof of the next lemma and in much of the remaining part of the paper,
we will find it more convenient to represent graphs with small blockers 
as coloured graphs, where 
a colour class
corresponds to a vertex in the blocker, and the set of colours of a vertex corresponds to the set of its neighbours in the blocker.
What follows is an attempt to capture this formally.

Let $r$ be a fixed integer. We call a graph $G$ with $r$ distinct distinguished
vertices, or roots, an \emph{$r$-rooted graph}. The roots will be labelled and ordered.
We say that a class $\Pi$ of $r$-rooted graphs  
is an $r$-property if
$\Pi$ is closed under isomorphism and under deleting edges between the roots.
We say that an unrooted graph $G$ has $r$-property $\Pi$ if it is possible to root $r$ of its vertices, so
that the resulting $r$-rooted graph is in $\Pi$.
We associate two classes with $\Pi$: the class $\ca_\Pi$ of uncoloured, unrooted graphs that
have $r$-property $\Pi$ and the class $\tilde{\ca}_\Pi$ of $\{0,1\}^r$-coloured graphs $G$, such that 
if $q_1 < \dots < q_r$ are not elements of $V(G)$, then $G^{\{q_1, \dots, q_r\}}$ with roots $(q_1, \dots, q_r)$
belongs to $\Pi$.

The next proposition just spells out the well known fact that a class of rooted
graphs has the same radius of convergence as the class of corresponding unrooted graphs.
\begin{prop}\label{prop.r-property}
    Let $\Pi$ be an $r$-property for some positive integer $r$. Then the sequence (\ref{eq.small}) for the class
    of rooted graphs $\Pi$, the class of unrooted graphs $\ca = \ca_\Pi$ and the class of $\{0,1\}^r$-coloured graphs $ \tilde{\ca} = \tilde{\ca}_\Pi$ has the same set $L$ of limit points, $L \subseteq [0; \infty]$.
\end{prop}
\begin{proof}
The claim follows, since
\[
|\tilde{\ca}_n| \le |\ca_{n+r}| \le |\Pi_{n+r}| \le 2^{\binom r 2} (n+r)_r |\tilde{\ca}_n|,
\]
and if (\ref{eq.small}) has a limit for a subsequence $(n_k, k=1,\dots)$ for any of the classes, it has the same limit for the other two.
\end{proof}

\medskip
The most important $r$-property for us will be the property of having a redundant blocker of size $r$.
Formally, given a set $\cb$ of graphs and a positive integer $r$, $\Pi_0$ is the set of all $r$-rooted graphs $G$, such 
that any subgraph of $G$ containing just one of the roots of $G$ is in $\ex \cb$.

Define 
$\crd {r,\cb} = \tilde{\ca}_{\Pi_0}$ and notice that $\ca_{\Pi_0} = \rd r \cb$. To keep the notation simpler,
below we will omit the subscript $\cb$, since the class $\cb$ will always be fixed. Proposition~\ref{prop.r-property} implies that the $\rho(\crd r) = \rho(\rd r \cb)$. We will also denote by $\cc^r = \cc^{r,\cb}$ the class of connected graphs in $\crd r$.

Given a set of graphs $\cb$ and a coloured graph $G$ with $N(G) \in \ex \cb$, call a colour $c$ \emph{bad} for $G$ (with respect to $\cb$), if $N(G) \in \ex(\cb)$, but adding to $G$ a new vertex $x_c$ connected to every vertex $v \in V(G)$ which has colour $c$ we produce a graph with a minor in $\cb$.
Otherwise, call $c$ good for $G$.
Notice, that $G \in \crd k$ if and only if $N(G) \in \ex \cb$ and  $\Col(G) \subseteq [k]$ and every colour is good for $G$.
%
%
%
%
\begin{lemma} \label{lem.doubleblockers}
    Let $k$ and $r$ be positive integers, such that $k < r$, and let $\cb$ be a finite set of graphs. 
    Suppose $\aw 2 (\ex \cb)$ is finite.
    Let $\ca$ be the class of graphs that have a $(k, 2, \cb)$-double blocker of size $r$.
    Then $\rho(\ca) = \rho(\rd {k+1} \cb)$.
\end{lemma}

\begin{proof}
    Let $\Pi$ be the $r$-property for ``containing a $(k,2,\cb)$-double blocker of size $r$'', i.e.,
    $\Pi$ is the set of all graphs  $G \in \ca$ with $r$ distinct roots $q_1, \dots, q_r$ so that $\{q_1, \dots, q_r\}$ is a $(k, 2, \cb)$-double blocker for $G$ with a special set $\{q_1, \dots, q_k\}$. 
    Then $\ca = \ca_\Pi$ and $\tilde{\ca} = \tilde{\ca}_\Pi$ have the same radius of convergence by Proposition~\ref{prop.r-property}.

    Let $\tilde{\cc}$ be the class of connected graphs in $\tilde{\ca}$. 
    The exponential formula, see, i.e., \cite{fs09}, gives that for $n = 0, 1, 2, \dots$ 
    \[
        [x^n] \tilde{C}(x) \le [x^n] \tilde{A}(x) \le [x^n] e^{\tilde{C}(x)}
    \]
    so $\rho(\tilde{\cc}) = \rho(\tilde{\ca})$. 
    Similarly $\rho(\cc^{k+1}) = \rho (\crd {k+1})$. Therefore by Proposition~\ref{prop.r-property} it suffices to prove that
    \begin{equation}
        \rho(\tilde{\cc}) = \rho(\cc^k).
    \end{equation}
    We have $\cc^k_n \subseteq \tilde{\cc}_n$, so $\rho(\cc^k) \ge \rho(\tilde{\cc})$. The difficult part is the opposite inequality.
    Our idea is to use the ``Colour reduction lemma'', Lemma~\ref{lem.colreduction}, to represent
    each graph in $\tilde{\cc}$ as a transformation of a finite set of disjoint graphs in $\cc^k$.

    Write $a = \aw 2 (\ex \cb)$.
    Consider a $\{0,1\}^r$-coloured graph $G \in \tilde{\cc}$. Let $G'$ be a $\{0,1\}^r$-coloured graph obtained by removing the
    colours $\{1, \dots, k\}$ from~$G$. 
    Suppose $\caw 2 (G') > a$. Then for any set $L = \{l_1, \dots, l_r\}$
    such that $l_1 < \dots < l_r$ and $L \cap V(G) = \emptyset$, the graph 
    $G^L - \{l_1, \dots, l_k\}$ has a subgraph $H \not \in \ex \cb$ such that $H$ has
    at most two vertices in $\{l_{k+1}, \dots, l_r\}$. But by the definition of $\tilde{\ca}$ (and $\tilde{\cc}$), $\{l_{k+1}, \dots, l_r\}$
    is a $(2, \cb)$-blocker for $G^L - \{l_1,\dots,l_k\}$, a contradiction. Therefore $\caw 2 (G') \le a$.

    By Lemma~\ref{lem.colreduction}, there is a graph $G_1$ obtained from the union of $\kappa \le N = (a+1)^{r-k-1}$ disjoint graphs,
    each with at most one colour in $\{k+1, \dots, r\}$,
    and a set  $J$ of $m \le N-1$ edges between these graphs, such that $G_1 / J = G'$ 
    and each component of $G_1 - J$ is stable with respect to $G_1 / J$.

    Now return the colours $\{1, \dots, k\}$ back: starting with the coloured graph $G_1$, for each $c \in \{1, \dots, k\}$ and each $v \in V(G_1 / J)$, 
    add $c$ to the set of colours for one of the vertices $v' \in Bag(v)$. Denote the newly obtained graph by $G''$.
    Then 
    $G'' / J = G$. 
    Each component $C$ of $G'' - J$ can have at most one colour $c \in \{k+1, \dots, r\}$, and  so
    at most $k+1$ colours in total. Crucially, if $C$ contains a colour
    $c \in \{k+1, \dots, r\}$,
    we can map the colour $c$
    to $k+1$, and this yields a graph $C'$ in $\cc^{k+1}$. Why? 
    Since $C$ is stable with respect to contraction of $J$ in $G_1$ (and also in $G''$),
    $C$ is isomorphic to a (coloured) subgraph of $G$.
    If there was a colour $j \in \col(C)$ which was bad for $C$, then it would be bad for $G$.
    But $G \in \tilde{\ca}$,
    this gives a contradiction.

    Recall that we assume that a contraction of an edge $x y$
    produces a new vertex with label $\min(x,y)$.
    Thus,
    each graph $G \in \tilde{\cc}_n$ (as well as other graphs) can be obtained
    by choosing integers $\kappa, l, m$ with $0 \le l, m, \kappa-1 \le N-1$,
    a graph $G_0 \in \left(\cc^{k+1}\right)^\kappa_{n+l}$, 
    for each component of $G_0$, mapping the colour $k+1$ to an arbitrary
    colour in $\{k+1, \dots, r\}$, adding a set $J$ of $m$ edges to $G_0$ and finally contracting them.   
    Therefore we have
    \begin{align*}
    \left[\frac {x^n} {n!}\right]\tilde{C}(x) \le \sum_{l = 0}^{N-1} \sum_{\kappa = 1}^N \sum_{m = 0}^ {N - 1} (n+l)^{2m} \left[\frac {x^{n+l}} {(n+l)!}\right]  \left( (r-k) C^{k+1} \right)^\kappa,
    \end{align*}
    from which it follows that $\rho(\tilde{\cc}) \ge \rho(\cc^{k+1})$.
\end{proof}

\begin{lemma}\label{lem.gammaHiolder}
   Let $\cb$ be any set of graphs.
   Let $k$ and $r$ be positive integers, $k \ge r$. 
   Then
   \[
   \gamupper(\rd k \cb)^2 \le \gamupper (\rd {k+r} \cb) \gamupper(\rd {k-r} \cb).
   \]
\end{lemma}

\begin{proof}\,
  By Proposition~\ref{prop.r-property} it suffices to show that $\gamupper(\crd k)^2 \le \gamupper(\crd {k+r}) \gamupper(\crd {k-r})$ (see above for the definition of $\crd k$).
  Fix positive integers $l, n$. Note that $\crd 0 = \ex \cb$. We can partition the class $\crd {l,n}$ (of $\{0,1\}^l$-coloured graphs on vertex set $[n]$) into $|\crd {0,n}|$ disjoint subclasses according to the underlying uncoloured graph $N(G)$ of
    $G \in \crd {l,n}$.
    
    Given 
    an uncoloured graph $G \in \crd 0$,
    let $X_G$ be the number of ways to add colour $1$ so that the resulting $\{0,1\}^1$-coloured graph is in $\crd 1$.
    Since the constraint for redundant blockers involves only individual vertices, we can
    pick the sets of vertices coloured 
    $1, 2, \dots, l$ 
    independently, in $X_G^l$ ways. Therefore
    \[
        |\crd {l,n}| = \sum_{G \in \crd {0,n}} X_G^l.
    \]
    We see that the equality also holds when $l = 0$.
    Choose $G$ from $\crd {0,n}$ uniformly at random. Then $X = X_G$ is a random variable and $|\crd {l,n}| = |\crd {0,n}| \E X^l$. 
    By the
    Cauchy-Schwarz
    inequality applied with random variables $X^{ (l - r)/2}$ and $X^{(l+r)/2}$ we have
    \[
      (\E X^l)^2 \le \E X^{l + r} \E X^{l -r},
    \]
    or
    \[
    |\crd {l,n}|^2 \le |\crd {l+r,n}| |\crd {l-r,n}|.
    \]
    Now the claim follows by dividing each side by $(n!)^2$, raising to $1/n$ and considering
    the subsequence that realizes the upper limit of the left side.
\end{proof}

\bigskip

\begin{lemma}\label{lem.gamupper}
    Let $\ca$ be a proper addable minor-closed class of graphs with a set $\cb$ of minimal excluded minors.
    Suppose $\aw 2 (\ex \cb)$ is finite and there is a positive integer $k_0$
    such that
    \[
    r = \frac {\gamupper( \ex (k_0+1) \cb)} {\gamupper(\ex k_0 \cb)} > 2.
    \]
    Then for any $k \ge k_0$
    \[
    \gamupper(\ex (k+1) \cb) = \gamupper(\rd {2k+1} \cb) \ge r \gamupper (\ex k \cb).
    \]
\end{lemma}
\begin{proof}
    In the proof we will need the following important consequence of the preceding lemmas.
    
    Let $t$ be a positive integer. By Lemma~\ref{lem.red2}, every graph $G \in \ex (t+1)\cb$ has a $\cb$-blocker $Q$, such that $Q$ contains a
    set $S$ of size at most $2t$ and $Q \setminus S$ is a $(2, \cb)$-blocker for $G - S$.
    Furthermore, the size of $Q$ is bounded by a constant $c = c(\cb, t)$. 

    For any integer $j \ge 0$ write $\gamupper_j = \gamupper (\rd j \cb)$.  Then
    \begin{equation}\label{eq.min}
     \gamupper(\ex (t+1) \cb) = \max (\gamupper_{2t+1}, \gamupper (\apex (\ex t \cb))).
    \end{equation}
    Let us prove (\ref{eq.min}). Write $d = c+w$ where $w=w(\cb)$ is as in  Lemma~\ref{lem.2canddecomp}. 
    We claim that for $n \ge d+1$ we have
    \begin{equation}\label{eq.dom6}
        [x^n] A(x) \ge [x^n] A_1(x) A_2(x),
    \end{equation}
    where $A(x), A_1(x), A_2(x)$ are exponential generating functions of $\ex (t+1) \cb$,
    the class $\ca_1$ of graphs that have a $(2t, 2, \cb)$-double 
    blocker of size 
    $d+1$ with 
    $d$ rooted vertices,
    and the class $\ca_2$ of
    graphs in $\apex( \ex t \cb)$
    which 
    have
    $d$ pointed (i.e., unlabelled, but distinguishable) vertices respectively. 

    (\ref{eq.dom6}) can be seen as follows. Given graphs $G_1 \in \ca_1$ and $G_2 \in \ca_2$ 
    with disjoint labels
    we can obtain a new graph
    by identifying the $i$-th distinguished vertex of $G_1$ with the $i$-th distinguished
    vertex of $G_2$ for $i = 1, \dots, d$ and 
    removing repetitive edges.
    By Lemma~\ref{lem.2canddecomp},
    the set of all resulting graphs will contain
    all graphs in
    $\ex (k+1) \cb$ of size at least $d+1$. Note that if $G \in \ex (k+1) \cb$ has at least $d+1$ vertices,
    we may assume that $G_1, G_2$ given by Lemma~\ref{lem.2canddecomp} are such that $G_1$ has exactly $d+1$
    vertices, $|Q_1| = d+1$ and $|V(G_2) \cap V(G_2)| = d$; otherwise $G_1$, $G_2$ can be extended by including extra isolated vertices, and $Q_1$ can be extended by adding more vertices from $G_1 - Q_1$.
    

    Now rooting or pointing a fixed number of vertices does not change the convergence radius of a class (see, i.e.,  Proposition~\ref{prop.r-property}),
    therefore $A_1$ has the convergence radius $\rho_1 
    = \gamupper_{2t+1}^{-1}$ by Lemma~\ref{lem.doubleblockers},
    $A_2$ has convergence radius $\rho_2 = 
    \gamupper(\apex(\ex t \cb))^{-1}$, 
    and by the theory of generating functions, see  i.e., \cite{fs09},
    the convergence radius $\rho$ of $A$ is at least $\min(\rho_1, \rho_2)$. Finally, $\rho$ is exactly this, since
    both $\rd {2t+1} \cb$ and $\apex (\ex t \cb)$ are contained in $\ex (t+1) \cb$.

    We will also use a simple bound
    $\gamupper(\apex(\cd)) \le 2 \gamupper(\cd)$, which is valid
    for any class of graphs $\cd$, since 
    $|\apex(\cd)_n| \le n 2^{n-1} |\cd_{n-1}|$.
   
    Let us now prove the lemma. We use induction on $k$.  First consider the case $k = k_0$. 
    We have $\gamupper(\apex(\ex k_0 \cb)) \le 2 \gamupper(\ex k_0 \cb) < \gamupper(\ex (k_0+1) \cb)$
    and so by (\ref{eq.min}), only one candidate
    to realize $\gamupper(\ex (k_0 +1) \cb)$ remains:
    \[
    \gamupper(\ex (k_0 +1) \cb) = \gamupper_{2k_0+1}.
    \]

   Now let $k' > k_0$ be an integer, assume we have
    proved the lemma with $k < k'$, let us now prove the case $k = k'$.

    We have $\gamupper(\ex k \cb) = \gamupper_{2k-1}$ by induction, therefore
    \[
    \gamupper(\apex (\ex k \cb)) \le 2 \gamupper_{2k-1}.
    \]
    Now Lemma~\ref{lem.gammaHiolder} and induction yields
    \[
    \gamupper_{2k+1} \ge \gamupper_{2k-1} \frac {\gamupper_{2k-1}} {\gamupper_{2k-3}} \ge r \gamupper_{2k-1}.
    \]
    So finally
    \[
    \gamupper_{2k+1} \ge r \gamupper_{2k-1} >
     2 \gamupper_{2k-1} \ge \gamupper(\apex(\ex k \cb))
    \]
    and the claim follows by (\ref{eq.min}).
\end{proof}

\bigskip

We are now ready to prove our first main theorem. 

\medskip

\begin{proofof}{Theorem~\ref{thm.main1}} \,
    We use the notation from Lemma~\ref{lem.gamupper}. By Lemma~\ref{lem.bounded_aw} below,
    $\aw 2 (\ca) < \infty$.
    As has been noticed in \cite{cmcdvk2012}, since $\ca$ contains all fans, $\apexp {2k+1} (\cp) \subseteq \ex (k+1) \cb$,
    where $\cp$ is the class of paths. So we have
    $\gamlower(\ex (k+1) \cb) \ge 2^{2k+1}$ for $k = 1, 2, \dots$. Also by Theorem~1.2 of \cite{cmcdvk2012},
    $\gamma (\apexp k (\ca)) = 2^k \gamma$. 

    Let $k_0$ be the smallest positive integer, such 
    that 
    \[
    \gamupper(\ex (k_0+1) \cb) 
    >
    \gamma(\apexp {k_0} (\ca) ). 
    \]
    Then for  $1 \le j < k_0$, by applying (\ref{eq.min}) $j$ times we have
    \[
    \gamupper(\ex (j+1) \cb) 
    = 2^j \gamma.
    \]
    Since $\apexp j (\ca) \subseteq \ex (j+1) \cb$, it follows that $2^j \gamma$ is the growth 
    constant of $\ex (j+1) \cb$.
    Thus
    \[
    \gamupper(\ex (k_0+1) \cb) > 2^{k_0} \gamma = 2\gamma(\ex k_0 \cb), 
    \]
    therefore $\gamupper (\ex (k+1) \cb) = \gamupper( \rd {2k+1} \cb)$ for all $k \ge k_0$ by Lemma~\ref{lem.gamupper}.

\medskip

Let us show that for $k \ge k_0$,  
\[
\gamupper \left( \left(\ex (k+1) \cb\right) \cap  \apexp {2k-1} (\ca)\right)  < \gamupper_{2k+1}.
\]
By 
Lemma~\ref{lem.gamupper}
we have
\[
\gamupper(\apex(\ex k \cb)) \le 2 \gamupper _ {2k-1} < \gamupper _ {2k+1}.
\]
So, using Lemma~\ref{lem.gammaHiolder} 
\[
\gamupper_{2k} \le \sqrt{\gamupper_{2k+1} \gamupper_{2k-1}} \le 2^{-1/2}\gamupper_{2k+1}.
\]
Now apply Lemma~\ref{lem.2canddecomp} (with $k, \cb$ and $l = 2k-1$), Lemma~\ref{lem.doubleblockers} and an inequality analogous to (\ref{eq.min}),
to get
\[
\gamupper \left( \left(\ex (k+1) \cb\right) \cap  \apexp {2k-1} (\ca)\right) = \max(\gamupper_{2k}, \gamupper(\apex(\ex k \cb))) < \gamupper _ {2k+1}.
\]

Finally, let us show that (\ref{eqn.main1}) holds in the case $\gamupper_{2k+1} < 2^k \gamma$. Note that, in such case $1 \le k < k_0$. It cannot be that for some $j \in \{1, \dots, k-1\}$ we have $\gamupper_{2j+1} > 2^j \gamma$, since Lemma~\ref{lem.gamupper}
would imply that $\gamupper_{2k+1} > 2^k \gamma$. Similarly, if $\gamupper_{2j+1} = 2^j \gamma$,
we would have $\gamupper_{2k+1} \ge 2^{k-j} \gamupper_{2j+1} \ge 2^k \gamma$ by Lemma~\ref{lem.gammaHiolder}.
Thus $\gamupper_{2j+1} < 2^j \gamma$ for all $j = 1, \dots, k - 1$.

Trivially, $|(\ex (0 + 1) \cb)_n| = |\ca_n| = | (\apexp 0 (\ca))_n|$. Using Lemma~\ref{lem.apexdomination} (given in Section~\ref{subsec.turningpoint} below) and induction,
we get that for $j \in \{1, \dots, k\}$
\[
|(\ex (j+1) \cb)_n|= | (\apexp j (\ca))_n| (1 + e^{-\Theta(n)}).
\]
\end{proofof}

\subsection{When is apex width finite?}
\label{subsec.finite_aw}

In Section~\ref{subsec.colreduction} we introduced the apex width parameter, which is not standard. Here
we present a characterisation of classes with bounded apex width in terms of excluded minors.
\begin{lemma} \label{lem.bounded_aw} Let $\ca$ be a minor-closed class. Then
    $\aw j (\ca) < \infty$ if and only if some $j$-fan and some bipartite graph $K_{j+1, t}$ does not belong to $\ca$.
\end{lemma}

For stating Theorem~\ref{thm.gc} in terms of minors, we will need another lemma.
\begin{lemma}\label{lem.bounded_aw_no_wheel}
 Let $\ca$ be a minor-closed class such that $\aw 2 (\ca) < \infty$.
 Then the following two statements are equivalent.
 \begin{enumerate}[(1)]
     \item There is a constant $c$ such that for each $G \in \ca$ and each
         $v \in V(G)$, if $G-v$ is 2-connected then the degree of $v$ in $G$
         is at most $c$.
     \item Some wheel does not belong to $\ca$.
 \end{enumerate}
\end{lemma}

To prove the above lemmas, we need a few simple preliminary results. Recall
that the \emph{height} of a rooted tree is the number of edges in the longest path starting from the root.
A \emph{leaf} of a rooted tree is a vertex of degree $1$, which is not the root.
A straightforward fact is:
\begin{lemma}\label{lem.leaves_bounded_height}
    Let $T$ be a rooted tree of size $n$, height $h$ and with $l$ leaves. Then
    $l h \ge n-1$.
\end{lemma}
For positive integers $j$ and $s$, denote by $F^j_s$ the $j$-fan on $s$ vertices. Also, let $K_{j,s}^*$ denote a graph obtained from the union of a $K_{j,s}$
and a $(j-1)$-star on the part of size $j$. Note, that for $j \le s+1$, $\aw j (K_{j+1,s}^*) = s+1$, and $K_{j+1, s}^*$ is isomorphic to a minor of $K_{j+1, s+j}$.

\bigskip

\begin{proofof}{Lemma~\ref{lem.bounded_aw}}
    If $\ca$ contains all $j$-fans then $\aw j (\ca) = \infty$ by definition.
    If $\ca$ contains all graphs $K_{j+1,t}$, then it contains all graphs $K_{j+1, t}^*$, where $t$ is arbitrary
    and $j$ is fixed, so 
    again $\aw j (\ca) = \infty$.

    Now suppose there are $t \ge 1$ and $s \ge 2$ such that
    $F^j_{j+s+1}, K_{j+1,t} \not \in \ca$.
    Let $T$ be any tree on at least $st$ vertices, let $S$ be a set of $j$ vertices,
    disjoint from $T$, and let $H$ be the union of $T$ and the complete bipartite graph with parts $S$ and $V(T)$. 
    If $T$ has a path of length at least $s$, then $H$ has a subgraph isomorphic to $F^j_{j+s+1}$.
    Otherwise we can root $T$ at a vertex $r$, so that the resulting rooted tree $T_r$ has height
    at most $s-1$. By Lemma~\ref{lem.leaves_bounded_height}, $T_r$
    has at least $\left \lceil \frac {s t - 1} {s-1} \right \rceil \ge t$ leaves. Therefore, contracting all internal vertices of $T_r$ into a single vertex and using vertices in $S$, we obtain a minor of $K_{j+1, t}$. 
    We have shown that $\aw j (\ca) < st$.
\end{proofof}
\bigskip

It is well known that each 2-connected graph on at least 3 vertices has a contractible edge (so that contracting
this edge yields again a biconnected graph), see, e.g., \cite{diestel}. We need a simple
refinement of this.
\begin{lemma}\label{lem.2conn_contractions} Let $G$ be a 2-connected graph on at least 3 vertices, and let $x \in V(G)$. There is an
    edge $xy \in E(G)$, such that $G/xy$ is biconnected.
\end{lemma}
\begin{proof}
    Assume the claim is false: then for each neighbour $u$ of $x$, $\{u,x\}$ must be a cut in $G$. Denote by $C(u)$ a component of $G - \{x,u\}$ of minimal size,
    and let $u'$ minimize $|V(C(u))|$ over the neighbours $u$ of $x$. Since $G$ is 2-connected, 
    Both $x$ and $u$ must have neighbours in $C(u')$. Let $z \in C(u')$ be a neighbour of $x$ in $G$. 
    Suppose $\{x,z\}$ is a cut in $G$. The graph $(G - C(u')) - x$ is connected (this follows using Menger's theorem),
    so $G - \{x,z\}$ must have a connected component which is strictly contained in $C(u')$. But this contradicts to
    the definition of $u'$. We conclude that $\{x,z\}$ is not a cut in $G$, so $G / xz$ must be 2-connected.
%
\end{proof}

\bigskip

The following ``simple fact'' about the size of largest cycle (circumference) of a 2-connected graph is stated in \cite{thomas93}. For completeness, we include a proof.
\begin{lemma}\label{lem.thomas}
    Let $k \ge 3$ be an integer. There is a positive integer $N = N(k)$ such that each 
    2-connected graph with at least $N$ vertices contains either a cycle of length $k$ as a subgraph
    or the complete bipartite graph $K_{2, k}$ as a minor.
\end{lemma}
\begin{proof}
    Write $N = k^{3k^2} + 1$, $\Delta = k^3 + 1$, and let $G$ be a 2-connected graph
    of size at least $N$.

    Let $P$ be the longest path in $G$, and let $x,y$ be its endpoints. Suppose $P$ has
    length at least $k^2$.
    By Menger's theorem, $G$ has a cycle $C$ containing $x$ and $y$. We can
    assume that $|V(C)| \le k-1$. The vertices in $V(C) \cap V(P)$ partition
    $P$ into at most $k-1$ subpaths, with internal vertices disjoint from $C$ (and endpoints in $C$).
    One of these subpaths must have at least $k$ vertices. This subpath,
    together with a part of the cycle $C$ yields a cycle of length at least $k$ in $G$.

    Therefore we may assume that that each path of $G$ has length at most $k^2-1$.  Let $T_r$ be a rooted spanning tree of $G$.
    A rooted tree with maximum degree at most $\Delta - 1$
    and height $h$ can have at most
    \[
    1 + (\Delta - 1) + \dots + (\Delta -1)^{h} \le (\Delta-1)^{h+1}
    \]
    vertices. Since $|V(G)| \ge N$, $T$ (and $G$) has a vertex $v$ of degree at least $\Delta$.

    Consider the $\{0,1\}^1$-coloured connected graph $G'$ obtained from $G-v$ by setting $C_{G'}(u) = \{red\}$ for
    each neighbour $u$ of $v$ in $G$. Let $T'$ be a minimal (Steiner) subtree of $G'$ containing all the red
    vertices. $|V(T')| \ge \Delta$ and $T'$ has diameter at most $k^2$. By Lemma~\ref{lem.leaves_bounded_height},
    $T'$ has at least $\lceil (\Delta - 1) / k^2\rceil \ge k$ leaves. By the minimality of $T'$,
    each leaf has the red colour. Contract the internal vertices of $T'$ to a single vertex;
    this vertex, the leaves of $T'$ and the vertex $x$ demonstrate that $G$ has a minor $K_{2,k}$.
\end{proof}

\bigskip

\begin{proofof}{Lemma~\ref{lem.bounded_aw_no_wheel}}
   It is trivial to see that if $\ca$ has arbitrarily large wheels, then (1) does not hold: the ``hub'' vertex $x$ of a wheel $W_t$ has $t-1$ neighbours, and 
   $W_t - \{x\}$ is 2-connected.
   Suppose a wheel $W_r$ is excluded from $\ca$ and (1) does not hold: we will obtain a contradiction. Set $j = \aw 2 (\ca)$. Take a graph $G \in \ca$ and a vertex $x$ such that
   $G - \{x\}$ is $2$-connected, and $x$ has  $d \ge N(k) + 3$ neighbours, where $k = \max(r-1, j+2)$ and $N(k)$ is as in Lemma~\ref{lem.thomas}.
   Let $G'$ be the $\{0,1\}^1$-coloured graph obtained from $G-\{x\}$ by colouring the former neighbours of $u$ $\{red\}$.
   By Lemma~\ref{lem.2conn_contractions}, the graph $G'$ has an all-red
   2-connected minor $H$ of size $d$ (repeatedly contract a contractible edge incident to an uncoloured vertex,
   until no uncoloured vertices remain). Now by Lemma~\ref{lem.thomas}, $H$ either has a cycle $C$ of length $k \ge r-1$,
   or $K_{2,k}$ as a minor. Recalling that $v$ in $G$ is incident to each red vertex of $G'$,
   we get that $G\in\ca$ has a minor $W_{k+1}$ in the first case (contradiction to the fact that $W_r \not \in \ca$). 
   In the second case, we see that $G$ has a minor $K_{3,j+2}$, therefore
   also a minor $K_{3,j}^*$, so $\aw 2 (\ca)  \ge j+1$, a contradiction.    
\end{proofof}

\section{Growth constants for $\ex (k+1) \cb$}

\label{sec.gcgen}

\subsection{Proof of Theorem~\ref{thm.gc}}
\label{subsec.proofgc}

In this section we prove Theorem~\ref{thm.gc}. Similarly as McDiarmid, Steger and Welsh \cite{msw05},
we will make use of a version of Fekete's lemma.
\begin{lemma} \label{lem.fekete}
    Suppose $(f(n), n= 1, 2, \dots)$ is a sequence of real numbers such that for any positive integers
    $n, m$
    \[
    f(n+m + 1) \ge f(n) + f(m).
    \]
    Then
    \[
    \sup \frac {f(n)} {n+1} \le \lim \inf \frac {f(n)} n; \quad \lim \sup \frac {f(n)} n \le \sup \frac {f(n)} {n}. 
    \]
\end{lemma}

\begin{proof}
    The second inequality is obvious. 
    Fix any $d \in \mathbb{N}$. We will show that
    \[
    \frac {f(d)} {d+1} \le \lim \inf \frac {f(n)} n.
    \]
    Define $f(0) := f(d+1) - f(d)$.
    For any $n \in \mathbb{N}$, let $k$ and $r$ be such integers that $n = (d+1)k + r$ and $0\le r \le d$. Then using the assumption
    \[
    \frac {f(n)} n \ge \frac {k f(d) + f(r)} {(d+1) k + r} \to \frac{f(d)} {d+1}, \quad \mbox{as } n\to \infty.
    \]
\end{proof}


We will use the following lemma multiple times in the following sections (cf. Lemma~4.4.4 of \cite{diestel}).
\begin{lemma}\label{lem.2cuts}
    Let $\{v_1, v_2\}$ be a cut in a graph~$G$. Suppose $G_1$ and $G_2$
    are subgraphs of $G$ with $V(G_1) \cap V(G_2) = \{v_1, v_2\}$ and $G_1 \cup G_2 = G$.
    Let $H'$ be a subdivision of a 3-connected graph $H$, and suppose $H'$ is not a subgraph of $G_1$ or $G_2$. Then either $G_1 \cap H'$ or $G_2 \cap H'$ is a
    path from $v_1$ to $v_2$.
\end{lemma}

\begin{proof}
  For $i = 1,2$ write $G_i' = G_i - \{v_1, v_2\}$. By the assumption, $H'$ must have vertices both in $G_1'$ and in $G_2'$,
  and so $\{v_1, v_2\}$ is a cut in $H'$ (both of these vertices must belong to $H'$ since it is 2-connected).

  Suppose first, that $v_1$ and $v_2$ are both on a path $P'$ of $H'$ which is a subdivided edge of $H$.
  Let $P$ be the path connecting $v_1$ with $v_2$ in $P'$. 
  If $P$ has no internal vertices, then because $H$ is 3-connected, the graph $H - \{v_1, v_2\}$ is connected, a contradiction.
  So $P$ has at least one internal vertex, so that $H - \{v_1, v_2\}$ has exactly two parts, one of which is the path $P - \{v_1, v_2\}$. 
  Since $H'$ has vertices both in $G_1'$ and $G_2'$, one of these parts is contained in a component of $G_1'$ and 
  another in a component of $G_2'$. The path $P$ is a subgraph of $P'$, so $H[V(P)] = P$, and in particular $v_1v_2 \not \in E(H)$.
  The claim follows.


  Now suppose there is no path $P$ in $H'$ which is a subdivided edge of $H$ and contains both $v_1$ and $v_2$.
  Using the fact that $H$ is 3-connected we see easily that $H' - \{v_1, v_2\}$ must be connected, a contradiction.
\end{proof}

\bigskip

Let a positive integer $l$ and a class $\cb$ be fixed.
As mentioned in the introduction, the class of $\{0,1\}^l$-coloured graphs $\crd l$ obtained from $\rd l \cb$ 
in Proposition~\ref{prop.r-property} 
is decomposable. Unfortunately, this class is not bridge-addable: in the case $\cb = \{K_4\}$, consider,
for example, the coloured graph $H$ obtained from $K_3$, by colouring each vertex with a distinct pair from $\{1,2,3\}$, see Figure~\ref{fig.unrootable}.
Then the graph $2 H \in \crd 3$, but no bridge can be added between the two components.

\begin{figure}
     \begin{center}
    \includegraphics[height=3cm]{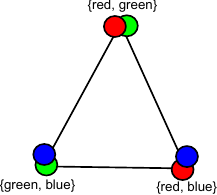}
    \end{center}
    \caption{This graph is not $3$-rootable with respect to $\cb = \{K_4\}$.}
    \label{fig.unrootable}
\end{figure}

Let, as before, $\cc^l$ be the class of connected graphs in $\crd l$. 
We call a graph $G \in \cc^l$ \emph{$l$-rootable} at a vertex $x \in V(G)$, if 
colouring $x$ with $[l]$ 
yields a graph 
still in $\cc^l$. $G$ is $l$-rootable if it is $l$-rootable at some vertex $x \in V(G)$.
For example, in the case $\cb = \{K_4\}$, the coloured graph $H$ from Figure~\ref{fig.unrootable} is not $3$-rootable.
Denote by $\cc^{\bullet l}$ the class of all rooted graphs (with at least one vertex) that can be obtained 
by declaring an $l$-rootable vertex of a graph in $\cc^l$ the root.

We will see next, that
when $\cb$ consists of $3$-connected graphs, 
it is possible 
to partly restore the property of bridge-addability: if we add an edge $xy$ connecting different components of $G \in \crd l$, we obtain a graph in $\crd l$, provided that $x$ and $y$ are rootable in their respective components.
\begin{lemma}\label{lem.unrootable} 
    Let $l$ be a positive integer, and let $\cb$ be as in Theorem~\ref{thm.gc}.
    Then the class $\cu'$ of graphs in $\cc^l$
    that are not $l$-rootable has $\gamupper(\cu') \le \gamupper(\cc^{l-1})$.
\end{lemma}

\begin{proof} 
    Call $G \in \cc^l$ \emph{nice}, if there is $x \in V(G)$ such that $G-x$ has at least
    two components containing all $l$ colours. 
    We claim that in this case, $G$ is rootable at $x$.
    Suppose the contrary. 
    
    Then there is a colour $i \in [l]$ such the graph $G'$ obtained by adding a new vertex $s$
    to $G$ connected to $x$ and each vertex of $G$ that has colour $i$, contains a
    $\cb$-critical subgraph $H$.

    Clearly $i \not \in \Col_G(x)$ and $sx \in H$, otherwise we would have 
    that $G \not \in \cc^l$.
    Suppose $H$ shares vertices with more than one component of $G-x$.
    Let $C$ be one such component. 
    By Theorem~9.12 of \cite{bondymurty} an expansion of a 3-connected graph at a vertex of degree at least 4 is 3-connected, 
    and so $H$ is a subdivision of a 3-connected graph.
    We may apply Lemma~\ref{lem.2cuts} with $G_1 = G'[V(C) \cup \{s,x\}]$,
    $G_2 = G' - V(C)$ and the cut $\{s,x\}$ to get that $H$ meets either $G_1$ or $G_2$ just by a path
    from $s$ to $x$: since $sx \in H$ this path must be $sx$, a contradiction.

    Thus $V(H) \setminus \{s,x\}$ must be completely contained in $V(C)$ for a component $C$ of $G - x$.
    But this means that $G \not \in \cc^l$: we may replace the edge $sx$ by a path from $x$ to $s$ in
    $G'-C$, since $G - s$ has a component, disjoint from $C$ with the colour $i$. We conclude
    that indeed $G$ is rootable at $x$. 

    Denote by $\cu$ the class of graphs in $\cc^l$ that are not nice. Then $\cu' \subseteq \cu$, and we need
    to show that $\gamupper(\cu) \le \gamupper(\cc^{l-1})$.
    Each graph $G \in \cu$ contains a block 
    $B$, such that for any cut vertex $y \in V(B)$, the components of $G - y$
    disjoint from $B$ can have at most $l-1$ colours.
    (This can be seen as follows. Consider the rooted block tree $T$ of $G$, and let $r$ be its root.
    For any block $B'$ let $r(B') \in V(G)$ be its parent in $T$, and denote
     by $G_{B'}$ the component of $G - r(B')$ containing $B - r(B')$. Define the set $\cs$ of
     of all blocks $B'$, such that $\Col(G_{B'}) = [l]$.
     We can assume it is non-empty, otherwise any block containing $r$ has the required property. 
     Pick a block $B \in \cs$ 
     with maximum distance from $r$ in $T$. Then each component $C$ of $G - V(B)$
     contains at most $l-1$ colours: otherwise, if $C \subseteq G_{B}$ we would
     have contradiction to the choice of $B$, and if $C \subseteq G - G_{B'}$,
     then $r(B)$ would be a vertex showing that $G$ is nice.)

    By 
    Lemma~\ref{lem.bounded_aw_no_wheel} for each colour $i \in [l]$ there can be at most a constant number $c$ of
    vertices $x \in V(B)$ such 
    that either $x$ has colour $i$ or the graph ``attached'' to $B$ at $x$ has colour $i$.
    Each graph in $\cc^l$ with at most $l-1$ colours can be obtained from a pair $(j, G_1)$, where $j \in [l]$ and $G_1 \in \cc^{l-1}$
    by mapping the colour $j$ to $l$ in $G_1$. 
    Therefore, for $n \ge cl$
    the coefficients $\left [\frac {x^n} {n!} \right]U(x)$ are bounded from above by
    \[
    \left [ \frac {x^n} {n!}  \right ]x^{cl} A^{(cl)}(x)\left(l (C^{l-1}(x))' \right) ^ {cl}
    \]
    where $A(x)$ is the generating function of $\ca = \ex \cb$ and $A^{(cl)}(x)$ is its $cl$-th derivative. 
    The convergence radius of $U(x)$ is at least $\min(\rho(\ca), \rho(\cc^{l-1}))$.
    The class $\cc^{l-1}$ contains all connected graphs in $\ca$; by the exponential formula
    we have $\rho(\ca) \ge \rho(\cc^{l-1})$.
    So $\gamupper(\cu) \le \gamupper(\cc^{l-1})$. 
\end{proof}    

\bigskip

We will need a simple technical lemma next.
\begin{lemma}\label{lem.gcunion}
    Suppose classes of graphs $\ca$ and $\cb$ both have growth constants.
    Then $\gamma(\ca \cup \cb)$ exists and is equal to $\max(\gamma(\ca), \gamma(\cb))$.
\end{lemma}
\begin{proof}
    We may assume that $\gamma(\ca) \ge \gamma(\cb)$. 
    Clearly, $\gamlower(\ca \cup \cb) \ge \gamma(\ca)$. Suppose there is a subsequence
    $(n_k, k=1,2,\dots)$ such that 
    \[
    \left(\frac {|\ca_{n_k} \cup \cb_{n_k}|} {n_k!} \right)^{1/{n_k}} \to a > \gamma(\ca).
    \]
    If $\gamma(\ca) > \gamma(\cb)$ then $|\ca_n| / |\cb_n| \to \infty$ and there is $k_0$ such that for $k \ge k_0$ we have
    $|\ca_{n_k}| \ge |\cb_{n_k}|$. If $\gamma(\ca) = \gamma(\cb)$ then either $|\ca_{n_k}| \ge |\cb_{n_k}|$
    or  $|\ca_{n_k}| \le |\cb_{n_k}|$ for infinitely many $k$. In this case, rename $\ca$ and $\cb$ if necessary, 
    so that the former holds. We get that $(n_k)$ contains a subsequence $(n'_l, l=1,2,\dots)$ such that
    $|\ca_{n_l'} \cup \cb_{n_l'}| \le 2 |\ca_{n_l'}|$ and
     \[
      \left(\frac {|\ca_{n_l'} \cup \cb_{n_l'}|} {n_l'!} \right)^{1/{n_l'}} \le \left(\frac {2|\ca_{n_l'}|} {n_l'!} \right)^{1/{n_l'}} \to \gamma(\ca). 
    \]
    This is a contradiction.
\end{proof}

\begin{lemma}\label{lem.gcindstep} Let $l$ be a positive integer and let $\cb$ be as in Theorem~\ref{thm.gc}. Suppose that $\cc^{l-1}$ and
    $\cc^{\bullet l}$ have growth constants. Then $\cc^l$ has a growth constant 
    \[
     \gamma(\cc^l) = \max(\gamma(\cc^{l-1}), \gamma(\cc^{\bullet l})).
    \]
\end{lemma}

\begin{proof}
    Denote by $\cc'$ the class of all $l$-rootable graphs in $\cc^l$, and denote by $\cc''$ the class
    of graphs $G \in \cc^{l}$ that are either not $l$-rootable or have $\Col(G) \subseteq [l-1]$.

    Then $\cc'$ has a growth constant, since $\cc^{\bullet l}$ does, and 
    \[
     |\cc_n'| \le |\cc^{\bullet l}_n| \le n |\cc_n'|.
    \]
    The class $\cc''$ also has a growth constant: $\gamma(\cc'') = \gamma(\cc^{l-1})$. This is because
    by Lemma~\ref{lem.unrootable}, $\gamupper(\cc'') \le \gamma(\cc^{l-1})$, and 
    $\cc^{l-1} \subseteq \cc''$. 

    Now,  by Lemma~\ref{lem.gcunion}, the class $\cc^l = \cc' \cup \cc''$ is as claimed.
\end{proof}

\bigskip

The next lemma shows that for good enough $\cb$, the class
$\cc^{\bullet l}$ is closed under joining smaller rooted graphs into ``strings''. 
\begin{lemma}\label{lem.rootablestring}
    Let $l$ be a positive integer and let $\cb$
    be a set of 3-connected graphs.
    Let $G$ be a graph obtained from a non-empty set $\cs$ of disjoint graphs in $\cc^{\bullet l}$, and a path on the set $R(\cs)$
    of the roots of the graphs in $\cs$; and let $G$ be rooted at $r \in R(\cs)$. Then $G \in \cc^{\bullet l}$.
\end{lemma}
\begin{proof}
    We first note that $\ex \cb$ contains all fans: indeed fans are series-parallel graphs,
    and if a fan $F$ had a minor in $\cb$, then since each $3$-connected graph has $K_4$
    as a minor (Lemma~3.2.1 of \cite{diestel}), $F$ would have $K_4$ as a minor, a contradiction.

    Suppose the claim does not hold.
    Call $\cs$ bad if the lemma fails for $\cs$, a path $P$ on $R(\cs)$ and a vertex $r \in V(P)$.
    Let $N(\cs)$ be the number of graphs in $\cs$ that have size at least~$2$. 
    Let $\nu$ be the size of the smallest bad set $\cs$, and 
    consider
    a set $\cs'$ which minimizes $N(\cs)$ over bad sets $\cs$ of size $\nu$,
    let $P'$ be a path on $R(\cs')$ and let $r' \in V(P')$ be the root for which the lemma fails.

    Then there is a colour $i \in [l]$ such that if we add a new vertex $s$ to $G$,
    connect $s$ to $r'$ and every vertex of $G$ that has colour $i$ and remove all the colours, the resulting graph $G' \not \in \ex \cb$.
    Let $H$ be a $\cb$-critical graph in $G'$.

    Suppose all graphs in $\cs$ are of size $1$. Then $G'$ is isomorphic to a minor of a fan, and $G' \in \ex \cb$, a contradiction. 

    Let $G_1$ be  a graph in $\cs'$ with root $r_1$ and size at least $2$. If $H$ has no vertices in $G_1 - r_1$,
    then we could replace $G_1$ with $G_1[\{r_1\}]$ and obtain a bad set $\cs''$ with $N(\cs'') < N(\cs')$. Thus,
    we may assume $G_1$ has at least one vertex coloured $i$, other than $r_1$. Since $G_1 \in \cc^{\bullet l}$,
    $H - s$ also has a vertex not in $G_1$. Now $\{r_1, s\}$ is a cut in $H$, so by Lemma~\ref{lem.2cuts},
    either $H \cap (V(G_1) \cup \{s\}) $ or $H \cap (V(G - G_1) \cup \{s,x\})$ is a path $P_1$ from $r_1$ to $s$.
    In the first case we may replace $G_1$ with the graph 
    consisting of a single vertex $r_1$ coloured $\{i\}$ 
    to obtain a set $\cs''$ with $N(\cs'') < N(\cs')$.
    In the second case,  we may replace $P$ with the edge $r_1 s$ to show that $G_1 \not \in \cc^{\bullet l}$.
    
    In each case we obtained a contradiction, so it must be that $G \in \cc^{\bullet l}$.
\end{proof}
\medskip

\begin{lemma}\label{lem.gcrooted}
    Let $l$ be a positive integer and let $\cb$ be as in Theorem~\ref{thm.gc}. Then $\cc^{\bullet l}$ has a growth constant.
\end{lemma}
\begin{proof}
    Let $n, m$ be positive integers. We claim that
    \begin{equation}\label{eq.rootableconstr}
    |\cc^{\bullet l}_{n+m+1}| \ge (n+m+1) \binom {n+m} n |\cc^{\bullet l}_n| |\cc^{\bullet l}_m|.
\end{equation}
    The above formula follows from the following construction for graphs on $n+m+1$ vertices. In the case $n\ne m$,
    pick a root vertex $r$, divide the remaining $n+m$ vertices into two parts of sizes $n$ and $m$,
    and add a graph $G_1 \in \cc^l$ of size $n$ on the first part and a graph $G_2 \in \cc^l$ of size $m$ on the second part.
    Connect $r$ to the roots $r_1$ and $r_2$ of $G_1$ and $G_2$ respectively, and declare $r$ the root of the
    formed graph $G$. Each construction gives a unique graph, because given a graph $G$ obtained in this way
    we may recover $G_1$ and $G_2$ uniquely by deleting the root of $G$ and declaring the vertices adjacent to 
    the two neighbours of $G$ the roots of the respective components.

    In the case $n = m$, we have to avoid obtaining each graph twice because of symmetry. So if $V(G_1)$ is lexicographically smaller
    than $V(G_2)$ output the graph $G$ as above, otherwise output the graph $G - r r_1 + r_1 r_2$. 
    To finish the proof of (\ref{eq.rootableconstr}), note the constructions $G$ are always in the class $\cc^{\bullet l}$ by
    Lemma~\ref{lem.rootablestring}. 

    
    Now let $d, k$ be positive integers, $k\ge 2$ and set $n = kd$.
    Form a graph $G$ by taking an arbitrary set $\cs$ of $k$ disjoint graphs in $\cc^{\bullet l}$ of size $d$,
    adding a path, $P$ rooted at one of the endpoints $r$ and
    with $V(P)$ consisting of all roots of the graphs in $\cs$. Declare $r$ the root of $G$.
    By Lemma~\ref{lem.rootablestring}, $G \in \cc^{\bullet l}$.

    Note also, that we never construct a graph $G$ twice: it is always possible to recover the path $P$
    and the set $\cs$ uniquely from $G$. (Start with the root $r$ of $G$. There can be only one edge $rx \in G$,
    such that $G-rx$ has a component $C$ of size $d$: $rx$ is the first edge of $P$. Delete $rx$ 
    and proceed in the same way with the component of $G-rx$ containing $x$, rooted at $x$.) 
    Let $\cp$  be the set of all graphs constructed in this way. Then
    \[
      |\cp_n| = \frac {(dk)!} {d!^k} |\cc^{\bullet l}_d|^k
    \]
    Next we observe, that since $\ex \cb$ contains all apex paths, the class $\cR$ of
    rooted uncoloured cycles is contained in $\cc^{\bullet l}$. This class is clearly disjoint from $\cp$,
    in which every graph has a bridge.

  For $t = 1,2, \dots$ let $f(t) = \ln ( |\cc^{\bullet l}_t| / t!)$. 
    We have
    \begin{equation}\label{eq.increase}
    \frac {f(n)} n \ge \frac 1 n \ln \left( \frac {|\cp_n| + |\cR_n|} {n!}   \right) > \frac {f(d)} d. 
    \end{equation}
    From (\ref{eq.rootableconstr})
    it follows that
    \[
       f(n+m+1) \ge f(n) + f(m).
    \]
    Thus by the modification of Fekete's lemma, Lemma~\ref{lem.fekete}
    \[
       \sup \frac {f(n)} {n+1} \le \lim \inf \frac {f(n)} n; \quad \lim \sup \frac {f(n)} n \le \sup \frac {f(n)} n.
    \]
    By (\ref{eq.increase}), for $n = k d$ and any integer $k = 2, 3,\dots$
    \[
    \frac {f(n)} {n+1} - \frac {f(d)} {d+1} > \frac {d} {d+1} \left( \frac {f(n)} n - \frac {f(d)} d \right) > 0.
    \]
    Therefore 
    \[
    \sup \frac {f(n)} {n+1} = \lim \sup \frac {f(n)} {n+1} = \lim \sup \frac {f(n)} n,
    \]
    and 
    \[
    \frac {f(n)} n \to \lim \sup \frac {f(n)} n \in [0; \infty].
    \]
    Since $\ex \cb$ is small by \cite{dn10}, we conclude that $\cc^{\bullet l}$ has a growth constant.
    
    Also, because $\ex \cb$ includes all graphs without a 3-connected minor, by Lemma~3.2.1 of \cite{diestel},
    $\gamma(\cc^{\bullet l}) \in [\gamma(\ex K_4); \infty)$, where $\gamma(\ex K_4) = 9.073..$, see Section~\ref{sec.part2}.
\end{proof}

\begin{lemma}\label{lem.gcrd}
    Let $l$ be a non-negative integer and let $\cb$ be as in Theorem~\ref{thm.gc}.
    Then the class $\cc^l$ has a growth constant.
\end{lemma}
\begin{proof}
    We use induction on $l$. The class $\cc^0$ is the class of connected graphs in
    $\ex \cb$, this class has a growth constant by \cite{msw05, cmcd09}.
    Suppose now that $l > 0$ and assume that we have proved the claim for each class $\ex (l'+1) \cb$, where $l' < l$,
    we now prove it for $l' = l$.

    The class $\cc^{\bullet l}$ has a growth constant by Lemma~\ref{lem.gcrooted}. The class $\cc^{(l-1)}$ has a growth constant
    by induction. So $\cc^l$ has a growth constant by Lemma~\ref{lem.gcindstep}.
\end{proof}

\medskip
We can now combine the lemmas of this section to finish the proof of Theorem~\ref{thm.gc}.
\medskip

\begin{proofof}{Theorem~\ref{thm.gc}}
    The class $\cc^{2k+1}$ has a growth constant $\gamma$ by Lemma~\ref{lem.gcrd}. Since
    \[
    [x^n]C^{2k+1}(x) \le [x^n]A_{2k+1}(x) \le [x^n]e^{C^{2k+1}(x)},
    \]
    we get, see i.e. \cite{fs09}, that $\crd {2k+1}$ also has growth constant $\gamma$. Using Proposition~\ref{prop.r-property},
    we see that $\rd {2k+1} \cb$ has growth constant $\gamma$ as well.
    By the assumption of the theorem, there must be a constant $c$, such 
    that $\cb$ does not contain a wheel $W_{c+1}$ as a minor (which is a planar graph).
    Now Theorem~\ref{thm.main1} completes the proof.
\end{proofof}

\subsection{Small blockers and small redundant blockers}
\label{subsec.turningpoint}

In this section we collect several auxiliary lemmas. 
For $\cb$ and $k_0$ as in Theorem~\ref{thm.main1}, we can often conclude that $R_n \in \ex (k+1) \cb$ either has a blocker of size $k$ (if $k < k_0$)  with probability $1 - e^{-\Omega(n)}$
 or (if $k \ge k_0$) it has a constant size $(2k, 2, \cb)$-double blocker  with probability $1 - e^{-\Omega(n)}$
. 
Using results of this section, it can be shown that this happens, for example, when $\gamma(\rd {2k+1} \cb) \ne 2^k \gamma(\ex \cb)$ exists for all $k$.

\begin{lemma}\label{lem.notcomplex} Let $k$ be a positive integer and let 
    $\cb$ be the set of minimal excluded minors for a proper addable minor-closed class of graphs.
    Suppose $\aw 2 (\ex \cb)$ is finite. 
    
    If $\gamlower(\ex (k+1)\cb) > 2 \gamupper(\ex k \cb)$ 
    then there is a constant $r = r(k, \cb)$ such that
    all but at most $e^{-\Omega(n)}$ fraction of graphs in $(\ex (k+1) \cb)_n$ have
    a $(2k, 2, \cb)$-double blocker of size $r$.
\end{lemma}
\begin{proof}
    By Lemma~\ref{lem.red2} and Lemma~\ref{lem.2canddecomp} there is a constant $r = r(k,\cb) > 2k$
    such that every graph in $G \in \ex (k+1) \cb$ 
    is a union of two graphs $G_1$ and $G_2$,
    where 
    $G_1$ has  a $(2k,2,\cb)$-double blocker $Q$ of size at most $r$ with a special set $S$, 
    $S \subseteq V(G_1) \cap V(G_2) \subseteq Q$, $G_2 \in \apex(\ex k \cb)$ and $Q$ is a $\cb$-blocker for $G$. 
    
    We may assign each $G \in \ex (k+1) \cb$ a unique tuple $t(G) = (G_1, G_2, Q, S)$ as above.
    Call $G$ \emph{complex}, if 
    $G_2 - (Q\setminus S)$ contains a subgraph $H \not \in \ex \cb$
    which has only one vertex $z \in S$. Observe, that if $G$ is not complex, then 
    $Q$ is a $(2k, 2, \cb)$-double blocker for $G$, and $S$ is its special set. 

    Suppose $G$ is complex, and let $H$ be a subgraph of $G_2 - (Q\setminus S)$ such that
    $V(H) \cap S = \{z\}$ for $z \in S$. Then $G_1 - z$ is disjoint from $H$, so $G_1 -z \in \ex k \cb$
    and $G_1 \in \apex(\ex k \cb)$. In this case, if $G$ has
    at least $r$ vertices, it can be obtained from a graph ${\tilde G}_1$ in $\apex(\ex k \cb)$ which
    has $s = |V(G_1) \cap V(G_2)|$ distinguished vertices (roots) and another graph ${\tilde G}_2$ in $\apex(\ex k \cb)$ which has $s$ pointed vertices, by identifying the $i$-th rooted vertex with the $i$-th pointed vertex and merging edges between the distinguished vertices.

    Thus the $n$-th coefficient, of the exponential generating function for 
    the complex graphs is bounded by 
    \[
    [x^n] \sum_{s = 0}^r x^s \left(A^{(s)}(x) \right)^2,
    \]
    where $A$ is the exponential generating function of $\apex(\ex k \cb)$ and $A^{(s)}$
    is the $s$-th derivative of $A$. 
    This shows that the inverse radius of convergence for
    the class of complex graphs is at most
    \[
    \gamupper(\apex(\ex k \cb)) \le 2 \gamupper(\ex k \cb) < \gamlower(\ex (k+1) \cb),
    \]
    see Proposition~\ref{prop.r-property} and the proof of Theorem~\ref{thm.main1}. Hence, all but at most $e^{-\Omega(n)}$ fraction of graphs in $| (\ex (k+1) \cb)_n |$ are not complex, and therefore have a $(2k, 2, \cb)$-double blocker.
\end{proof}

\bigskip

We call a connected subgraph $H$ of $G$ a \emph{pendant subgraph}, if there is exactly one edge 
in $G$ between $V(H)$ and $V(G) \setminus V(H)$.

\begin{lemma}\label{lem.apexring}
    Let $\ca$ be a proper addable minor-closed class of graphs.
    Let $H \in \ca$ be a connected graph
    and let $k$ be a non-negative integer.
    
    There is a constant $c > 0$, such that the random graph $R_n \in_u \apexp k (\ca)$
    with probability $1 - e^{-\Omega(n)}$ has a set $S$ of $k$ vertices,
    such that 
    $R_n - S$ contains a family $\ch$ of at least $c n$
    pairwise disjoint 
    pendant subgraphs isomorphic to $H$,
    and each vertex of $S$ is incident to all vertices of every graph $\tilde{H} \in \ch$.
\end{lemma}

\begin{proof}
    This fact is proved in the proof of Theorem~1.2 of \cite{cmcdvk2012}.
\end{proof}

\bigskip

\begin{lemma}\label{lem.apexdomination}
    Let $k$ be a positive integer and let $\cb$ be the set of minimal excluded minors
    for a proper addable minor-closed class of graphs.
    Suppose $\ex \cb$ contains all fans, $\aw 2 (\ex \cb)$ is finite,
    \[
    |(\ex k \cb)_n| \le | (\apexp {k-1} (\ex \cb))_n| \left(1 + e^{-\Omega(n)}\right)
    \]
    and $\gamma_2 > \gamma_1$, where
    $\gamma_2=\gamupper(\apex (\ex k \cb))$, $\gamma_1 = \gamupper(\rd {2k+1} \cb)$.
    Then 
    \[
     | ( \ex (k+1) \cb )_n | = |(\apexp k (\ex \cb))_n| \left(1 + e^{-\Theta(n)}\right).
    \]
\end{lemma}

\begin{proof}
     By Lemma~\ref{lem.red2} and Lemma~\ref{lem.2canddecomp} there is a constant $r = r(k,\cb) > 2k$
    such that  every graph in $G \in \ex (k+1) \cb$ with at least $r$ vertices can be generated as follows.
%
    \begin{enumerate}[\quad 1)]
        \item Pick $n_2 \in \{0, \dots, n\}$.
        \item Pick a set $V_2 \subseteq [n]$ of size $n_2$.
        \item Pick $q \in \{0, \dots, n_2 \wedge r\}$. 
        \item Pick a set $Q \subseteq V_2$ of size $q$.
        \item Put any graph $G_2 \in \ca$ on $V_2$. Here $\ca = \apex (\ex k\cb)$.
        \item Add edges of any graph $G_1 \in \cd$ on $V_1 = ([n] \setminus V_2) \cup Q$ (merge
            repetitive edges, if necessary). Here 
            $\cd$ is the class of graphs with a $(2k, 2, \cb)$-double blocker of size at most $r$. 
    \end{enumerate}

    Let $u_n$ be the total number of constructions, i.e., 
    the total number of different tuples $(n_2,q, V_2, Q, G_1, G_2)$ that can be generated
    by the above procedure. Denote by $\cu$ be the combinatorial class with the counting sequence
    $(u_n, n = 0, 1, \dots)$. Also, let $R_n$ be the graph obtained by taking the tuple
    $(n_2, q, V_2, Q, G_1, G_2)$ uniformly at random from all $u_n$ possible tuples. (In the rest of
    the proof $n_2, q, V_2, Q, G_1, G_2$ will be random variables.)

By Lemma~\ref{lem.doubleblockers},
$\gamupper(\cd) = \gamma_1$.
Similarly as in the proof of (\ref{eq.min})
\[
u_n \le [x^n/n!] \sum_{q = 0}^{r} x^q A^{(q)}(x) D^{(q)}(x),
\]
so, see \cite{fs09}, $\gamupper(\cu) \le \gamma_2$.
    Fix $\eps \in (0, 0.5)$ and $\delta >0$ such that
    \[
    (\gamma_1+\delta)^{\eps} (\gamma_2 + \delta)^{1-\eps} < \gamma_2.
    \]
    Let $H$ be a graph of minimal size that can
    be obtained by removing one vertex from a graph in $\cb$.
    Let $c$ be a constant as in Lemma~\ref{lem.apexring} applied with $\ex \cb$ and $H$.
    We say that a set $S \subseteq V_2$ is \emph{good} if
    $|S| = k$ and 
    there is a family $\ch_S$ of at least $c n/4$ disjoint pendant subgraphs $\tilde{H}$ 
    in $G_2 - (S \cup Q)$ such that $\tilde{H}$ is isomorphic to $H$
    and  every vertex $v \in S$ is incident to every vertex
    of $V(\tilde{H})$.

    Define the following events:
    \begin{align*}
           &A = \{\mbox{$G_2$ has at least $(1-\eps)n$ vertices}\};
        \\ &B = \{G_2 \in \apexp k (\ex \cb)\};
        \\ &C = \{G_2 \mbox{ has a good set }S \}.
    \end{align*}
    We will show that
    \begin{align}
      \pr(\bar{A}) \le e^{-\Omega(n)}; \quad \pr(\bar{B}) \le e^{-\Omega(n)}; \quad \pr(\bar{C}) \le e^{-\Omega(n)}. \label{eq.e1}
    \end{align}
    and
    \begin{equation}\label{eq.e4}
    \gamma(\cu) = \gamma(\apexp k (\ex \cb)) = \gamma_2.
\end{equation}
%
    For $n$ large enough, $A, B$ and $C$ imply that either $R_n \in \apexp k (\ex \cb)$ 
    or $R_n$ has $k+1$ disjoint subgraphs not in $\ex \cb$. Indeed, by  Lemma~5.3 of \cite{cmcdvk2012}, if $S$ is a good set,
    then for a $\cb$-critical subgraph $H_1$ of $G - S$, there is
    at most a constant number $N_{k,\cb}$ of subgraphs $\tilde{H} \in \ch_S$,
    which are not disjoint from $H_1$.
    For $n$ large enough, $k < c n/4 - N_{k,\cb}$, so we can construct $k$ disjoint subgraphs not in $\ex \cb$, each containing
    one vertex from $S$, and each disjoint from $H_1$, producing $k+1$ disjoint forbidden subgraphs in total.

    Denote $a_n' := |\apexp k(\ex \cb)_n|$. 
    Assuming (\ref{eq.e1}) and (\ref{eq.e4}) hold,
    we have for $n$ large enough
    \begin{align}
     &|(\ex (k+1) \cb)_n \setminus (\apexp k (\ex \cb))_n| \le u_n (\pr(\bar{A}) + \pr(\bar{B}) + \pr(\bar{C})) \nonumber
     \\ &= n! \gamma_2^{n - \Omega(n)} = e^{-\Omega(n)} a_n'. \label{eq.omegaa}
    \end{align}
    Let us show (\ref{eq.e1}) and (\ref{eq.e4}).
Recall that by Theorem~1.2 of \cite{cmcdvk2012}, $\gamma(\apex^l(\ex \cb))$ $= 2^l \gamma(\ex \cb)$
    for any $l = 0, 1,2 ,\dots$.     
    By the definition of apex classes and the assumption of the lemma
    \begin{align*}
        &|\apex(\ex k\cb)_n \setminus (\apexp k (\ex \cb))_n | 
        \\ &\le n 2^{n-1} |(\ex k\cb)_{n-1} \setminus (\apex^{k-1} (\ex \cb))_{n-1}|
        \\ & \le n 2^{n-1} e^{-\Omega(n)} | (\apexp {k-1} (\ex \cb))_{n-1}|
\end{align*}
The last line is $e^{-\Omega(n)} a_n'$ again by Theorem~1.2 of \cite{cmcdvk2012}. So
\begin{equation}\label{eq.e5}
    a_n := |\apex(\ex k\cb)_n| \le a_n' (1 + e^{-\Omega(n)})
\end{equation}
and $\gamma_2 = 2^k \gamma(\ex \cb)$ is the growth constant of $\apex(\ex k \cb)$.
Since $u_n \ge a_n'$, we have $\gamlower(\cu)\ge \gamma_2$ and (\ref{eq.e4}) follows. 

Let $d_n = |\cd_n|$. 
There is a constant $C$, such that for any $n = 1, 2,\dots$ 
\[
a_n \le C n! (\gamma_2 + \delta)^n; \quad d_n \le C n! (\gamma_1 + \delta)^n.
\]

Using (\ref{eq.e4}), there is a constant $C'>0$, such that the number of constructions with $n_2 < (1-\eps) n$ is at most
\begin{align*}
        &\sum_{n_2 = 0}^{\lfloor (1-\eps)n \rfloor} \sum_{q = 0}^{r \wedge n_2} \binom n {n_2} \binom {n_2} q a_{n_2} d_{n-n_2 + q}
     \\ &\le C' n^{2q + 1} n! \max_{n_2<(1-\eps)n} (\gamma_1 + \delta)^{n-n_2} (\gamma_2+\delta)^{n_2}
     \\ &\le  C' n^{2q + 1} n! \left( (\gamma_2 + \delta)^{1-\eps}  (\gamma_1+\delta)^\eps \right)^n
     \\ &\le e^{-\Omega(n)} u_n,
\end{align*}
and the first bound of (\ref{eq.e1}) follows.

The second bound of (\ref{eq.e1}) follows by the first one and (\ref{eq.e5}) since
\[
\pr(\bar{B}) \le \pr(\bar{A}) + \pr(\bar{B}|A) = e^{-\Omega(n)}.
\]
Fix an integer $t$, $(1-\eps) n \le t \le n$, and a subset $V_2 = \tilde{V}$ of size $t$.
Conditionally on $V_2 = \tilde{V}$ and the event $B$, 
the random graph $G_2$ is a uniformly random graph on $\tilde{V}$
from $\apexp k (\ex \cb)$.

By Lemma~\ref{lem.apexring} there is a constant $c_1 > 0$
such that for all large enough $n$, conditionally on $V_2 = \tilde{V}$ and $B$,
the graph $G_2$ with probability at least $1 - e^{-c_1 (1-\eps) n}$
has a set $S$, where every vertex in $S$ is incident to 
every vertex of at least $c (1-\eps) n \ge cn/2$
disjoint pendant subgraphs of $G_2$ isomorphic to $H$. 
At most $q$ such subgraphs
can have vertices in $Q$, so if $n$ is large enough then $cn/2 - q > cn/4$ and $S$ is good.

Now
\[
\pr(\bar{C}) \le \pr(\bar{A}) + \pr(\bar{B}) + \pr(\bar{C} |A,B).
\]
%
For large enough $n$, by the above argument and symmetry the last term on the right side is
\begin{align*}
 &\frac 1 {\pr(A, B)} \sum_{\tilde{V} \subseteq [n], |\tilde{V}| \ge (1-\eps)n} \pr(\bar{C} |V_2 = \tilde{V}, B) \pr(V_2 = \tilde{V}, B)
 \\ &\le e^{-c_1 (1-\eps) n} \frac {\pr(A,B)} {\pr(A,B)} = e^{-\Omega(n)},
\end{align*}
and the last bound of (\ref{eq.e1}) follows.

Finally, the fact that $| (\ex (k+1) \cb)_n \setminus (\apexp k (\ex \cb))_n| \ge e^{-\Theta(n)} a_n'$
follows by Lemma~5.5 of \cite{cmcdvk2012}.
\end{proof}

\bigskip
For $\cb$ as in Theorem~\ref{thm.gc} and sufficiently large $k$ we
have, using Theorem~\ref{thm.main1},
that $R_n \in_u \rd {2k+1} \cb$ belongs to $\apexp {2k-1} \cb \supseteq \apexp k \cb$
with probability $e^{-\Omega(n)}$.
The next lemma shows that the two candidates for 
the main subclass of $\ex (k+1) \cb$ studied so far 
essentially do not overlap.
\begin{lemma} \label{lem.apex_not_rd}
    Let $\ca$ be a proper addable minor-closed class. Let $\cb$ be its set of minimal excluded minors. 
    There is a constant $c > 0$
    such that with probability $1 - e^{-\Omega(n)}$ every redundant blocker 
    $R_n \in_u \apexp k (\ex \cb)$ is of size at least $c n$.
\end{lemma}

\begin{proof}
    Fix a graph $H \in \cb$. Let $H_0 = H - v$, where $v \in V(H)$ is any vertex.
    Fix $\eps \in (0,1)$, and let $A_{\eps} = A_{\eps}(n)$ be the event
    that the random graph $R_n$ has a unique blocker $S$ of size $k$
    and the graph $R_n - S$ has at least $\eps n$ pendant appearances $\tilde{H}$ of the graph $H_0$, such 
    that every vertex of $\tilde{H}$ is connected to every vertex of $S$ (call such pendant appearances \emph{good}),
    and any $\cb$-blocker not containing $S$ has at least $\eps n$ vertices.
    By Lemma~\ref{lem.apexring} and Theorem~1.3  of \cite{cmcdvk2012} we can choose
    $\eps$ so, that $A_\eps$ occurs with probability $1 - e^{-\Omega(n)}$.

    Let $c = \eps / 2$. Let $n$ be sufficiently large, so that $(\eps - c) n > 2$.
    Suppose $A_\eps$ occurs and $R_n$ has a redundant blocker $Q$ of size at most $c n$.
    Then $Q$ must contain $S$ and there must be at least one good appearance $\tilde{H}$ disjoint from $Q$.
    Now any vertex $x \in S$
    together with $\tilde{H}$ induces a graph containing $H$ in $R_n - (Q \setminus \{x\})$, thus  $Q$ is not a redundant blocker. 
    So the probability that $R_n$ has a redundant blocker of size at most $cn$ is
    at most $\pr(\bar{A}_\eps) = e^{-\Omega(n)}$.
\end{proof}

\section{Analytic combinatorics for $\ex 2 K_4$}   
\label{sec.part2}



In this section we focus on the case $\cb = \{K_4\}$ and the class $\cc^3$. 
Recall that $\crd l$ denotes the set of $\{0,1\}^l$-coloured graphs $G$ such that
if $G \in \crd {l,n}$ then $\{n+1, \dots, n+l\}$ is a redundant $\cb$-blocker for~$G^{\{n+1, \dots, n+l\}}$ 
and $\cc^l$ is the class of connected
graphs in $\crd l$.
The main result of this section is the following.
\begin{lemma} \label{lem.part2main}
    Let $\cb = \{K_4\}$. The class $\cc^3$ 
    has growth constant
    $\gamma(\cc^3) = 23.5241..$\,.
\end{lemma}

This shows that in Theorem~\ref{thm.main1} and Theorem~\ref{thm.gc} we have $k_0 = 1$ for $\cb  = \{K_4\}$:
\begin{corollary}\label{col.part2main}
    Let $\cb = \{K_4\}$. For any $k = 1, 2, \dots$ 
    \[
    \gamma(\ex (k+1) K_4) = \gamma(\rd {2k+1} K_4) = \gamma (\crd {2k+1}).
    \]
\end{corollary}
\begin{proof}
    Bodirsky, Gim\'{e}nez, Kang and Noy \cite{momm05} showed that $\gamma(\ex K_4) = 9.073\dots$. 
    By the exponential formula
    \[
      [x^n]C^3(x) \le [x^n]A_3 (x) \le [x^n]e^{C^3(x)},
    \]
    so by Proposition~\ref{prop.r-property} 
    and Lemma~\ref{lem.part2main} 
    \[
        \gamma(\rd 3 K_4) = \gamma(\crd 3) = \gamma(\cc^3) > 2 \gamma(\ex K_4).
    \]
    Now since $\rd 3 K_4 \subseteq \ex 2 K_4$, 
    the claim follows by Lemma~\ref{lem.gamupper} and Theorem~\ref{thm.gc}.
\end{proof}


\subsection{Series-parallel networks}

Recall that a graph $G$ is called series-parallel if $G \in \ex K_4$.
A series-parallel graph $G$ with an ordered pair of distinguished vertices $s$ and $t$ 
is called an  \emph{SP-network} if $G$ is connected and adding an edge $st$ to $G$, the resulting 
multigraph is 2-connected and series-parallel (so a network isomorphic to $K_2$ is also an $SP$-network). $s$ and $t$
are called the \emph{poles} of $G$; $s$ is the \emph{source} of $G$ and $t$ is the \emph{sink} of $G$. 
The poles have no label and do not contribute to the size of $G$. A vertex $v \in V(G)$ that is not a 
pole, is called an \emph{internal vertex}.

Denote by $\cd$ the class of all $SP$-networks and by $\ce_2$ the class of $SP$-networks consisting of a single edge between the source and the sink.
Also, denote by $\ce_1$ the class of degenerate 
networks with source and sink represented by the same vertex and zero internal vertices.
The corresponding exponential generating functions are $E_2(x) = E_1(x) = 1$.
\begin{lemma}[Trakhtenbrot 1958, \cite{trakhtenbrot}, see also \cite{walsh82}]\label{lem.SP}
 We have
 \[
 \cd = \ce_2 + \cs + \cP.
 \]
 Here $\cs$ and $\cp$ are defined by $|\cs_0| = |\cp_0| = 0$ and
 \begin{align*}
     &\cs = (\cP + \ce_2) \times \SEQ_{\ge 1} (\cz \times (\cP + \ce_2) );
     \\ &\cP = \ce_2 \times \SET_{\ge 1} (\cs) + \SET_{\ge 2} (\cs).
 \end{align*}
 Furthermore, the classes $\cs$ and $\cp$ correspond to disjoint classes of networks
 and the above relation corresponds to 
 a unique decomposition of a graph $G \in \cs$ (respectively, $G \in \cp$) into subgraphs in $\cp \cup \ce_2$ (respectively, $\cs \cup \ce_2$)
 with pairwise disjoint sets of labels.
\end{lemma}


\begin{figure}
     \begin{center}
    \includegraphics[height=5cm]{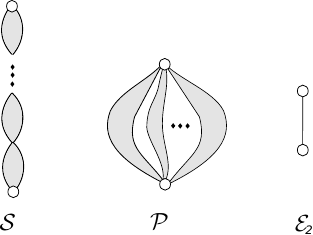}
    \end{center}
    \caption{The structure of the class  $\cd$ provided by Lemma~\ref{lem.SP}.}
    \label{fig.SP}
\end{figure}
The last statement of the lemma asserts that there is a stronger kind of isomorphism than just combinatorial one.
More precisely, the classes $\cs$, $\cp$ and $\ce_2$ can (and will) be considered as classes of graphs, which naturally partition the class of all $SP$-networks, 
see Figure~\ref{fig.SP}. A network $G \in \cs$ is
called a \emph{series SP-network}. $G$ can be decomposed uniquely into $k \ge 2$ networks
$H_1, \dots, H_k \in \cp + \ce_2$, where the sink of $H_i$ is the source of $H_{i+1}$ for $i = 1,\dots,k-1$, the source of $G$ is
the source of $H_1$, the sink of $G$ is the sink of $H_k$, and
the sets of internal vertices are disjoint for $H_i, H_j$, $i \ne j$. We say that $G$ is obtained from $H_1, \dots, H_k$
by \emph{series composition}.

A network $G \in \cp$ with source $s$ and sink $t$ is called a \emph{parallel SP-network}.  $G$ can be decomposed uniquely
into $k \ge 2$ networks $S_1, \dots, S_k \in \cs \cup \ce_2$, where at most one network is in $\ce_2$. 
In such a decomposition,
the sets of internal vertices of $S_1, \dots, S_k$ are pairwise disjoint, the source of $S_1, \dots, S_k$ is $s$, and the
sink of $S_1, \dots, S_k$ is $t$. We say that $G$ is obtained from $S_1, \dots, S_k$ by \emph{parallel composition}.
The above decomposition also implies that for any internal vertex $v$ of $G \in \cp$, we may represent $G$ as a parallel composition
of a network $S \in \cs$ (where $S = S_j$ with $v \in V(S_j)$) and a network $D \in \cd$ (where $D = \cup_{i \ne j} S_i$).


It has been shown, see \cite{momm05} and \cite{bps08}, that the exponential generating functions of $\cd$ and $\cp$ satisfy 
\begin{align}
    \frac {x D(x)^2} {1 + x D(x)} &= \ln \left( \frac {1 + D(x)} 2 \right); \label{eq.D}
    \\ P(x) + 1  &= \frac {D(x)} {1 + x D(x)}. \label{eq.P}
\end{align}
To keep formulas shorter, for exponential generating functions $A(x)$ we will
often skip ``$(x)$''; $x \in \mathbb{C}$ will usually be fixed, and its value should be clear from the context.
Identities where the range of $x$ is not explicitly stated, will hold for some $\delta > 0$ and
any $x \in \mathbb{C}$ with $|x| < \delta$.

The following simple facts were used already in \cite{trakhtenbrot}. 

\begin{prop} \label{prop.SPpath}
  Let $G$ be an $SP$-network with poles $s$ and $t$. Then for each internal vertex $v$
of $G$ there is a path from $s$ to $t$ containing $v$.
\end{prop}
\begin{proof}
  If $G$ is a parallel $SP$-network, then since $G$ is 2-connected,
there are internally disjoint paths, a path from $v$ to $s$ and a path from $v$ to $t$. Connecting
them, we obtain a path from $s$ to $t$. 

  Suppose $G$ is a series network. 
  Let $H_1, \dots, H_k$ be the decomposition of $G$ into graphs in $\cp + \ce_2$ as in Lemma~\ref{lem.SP}.
  Then for $i = 1, \dots, k$, the network $H_i$ contains a path $P_i$ connecting its poles, and, if $v$ is an internal vertex of $H_i$, also containing $v$. Connecting each of the paths $P_i$ yields a path from $s$ to $t$ that contains $v$.
\end{proof}

\begin{prop} \label{prop.SPcycle}
  Let $G$ be a parallel $SP$-network with source $s$ and sink $t$. Then
  for each internal vertex $v$ of $G$  there are two internally disjoint paths from $s$ to $t$ such 
  that one of the paths contains $v$.
\end{prop}

\begin{proof}
 By Lemma~\ref{lem.SP} 
 the graph $G$ can be obtained by a parallel composition of two networks $S \in \cs$ and $D \in \cd$ with disjoint sets of vertices
 where $v$ is an internal vertex of $\cs$. By Proposition~\ref{prop.SPpath}, there is a path $P$ from $s$ to $t$ that contains $v$. Now $D$ contains another path
 from $s$ to $t$ internally disjoint from $P$.
\end{proof}

\begin{prop}\label{prop.v}
       Let $G \in \cd$. The network $G'$ obtained by adding a new vertex $w$ connected to both poles of $G$
       satisfies $G' \in \cp$.
\end{prop}

\begin{proof} This is an immediate consequence of Lemma~\ref{lem.SP}, see the comment after it.
\end{proof}

\bigskip




\subsection{Rooted graphs of multiple types}
\label{subsec.cbt}

Let $F(x)$ and $B(x)$ denote the exponential generating functions of
rooted connected series-parallel graphs and biconnected series-parallel graphs 
respectively. 
Then (see, e.g., \cite{momm05, gimeneznoy09})
\begin{equation}\label{eq.unitypeblocks}
  F(x) = x e^{B'(F(x))}.
\end{equation}
An analogous formula works for any addable class of graphs. 
However, it fails for classes $\cc^k$: we have to consider several types of rooted graphs instead.

Let $G$ be a coloured graph with one pointed uncoloured vertex $r$, called the \emph{root} of $G$. 
Let $C$ be the set of all colours of $G$. 
We call a colour $c$ \emph{good} for $G$, if the graph obtained from $G$ by adding
a new vertex $w$ connected to $r$ and each vertex of $G$ coloured $c$
contains no $K_4$ as a minor. Otherwise we call $c$ \emph{bad} for $G$.
We call $G$ a \emph{$C$-tree}, if the following three conditions are satisfied:
\begin{itemize}
    \item[(a)] each colour $c \in C$ is good for $G$, 
    \item[(b)] $G$ is connected and it has no cut vertex $x$ such that $G-x$ has a component without colours and without $r$, and
    \item[(c)] $r$ is not coloured, not a cut vertex of $G$ and not the only vertex of $G$.
\end{itemize}
For a positive integer $k$ and $C \subseteq [k]$, denote by $\ca_C$ the family of all
$C$-trees, see Figure~\ref{fig.RGB-tree}. We define $\ca_\emptyset = \emptyset$. If $C \ne \emptyset$, then for $n = 0, 1, 2$,  $|\ca_{C, n}|$ is equal to $0, 1$ and $2^{|C|}$ respectively.
We will now study the exponential generating functions
of $\ca_C$; in the end of this section we will use the results to obtain the growth constant of~$\cc^k$.

\begin{figure}
     \begin{center}
    \includegraphics[height=5cm]{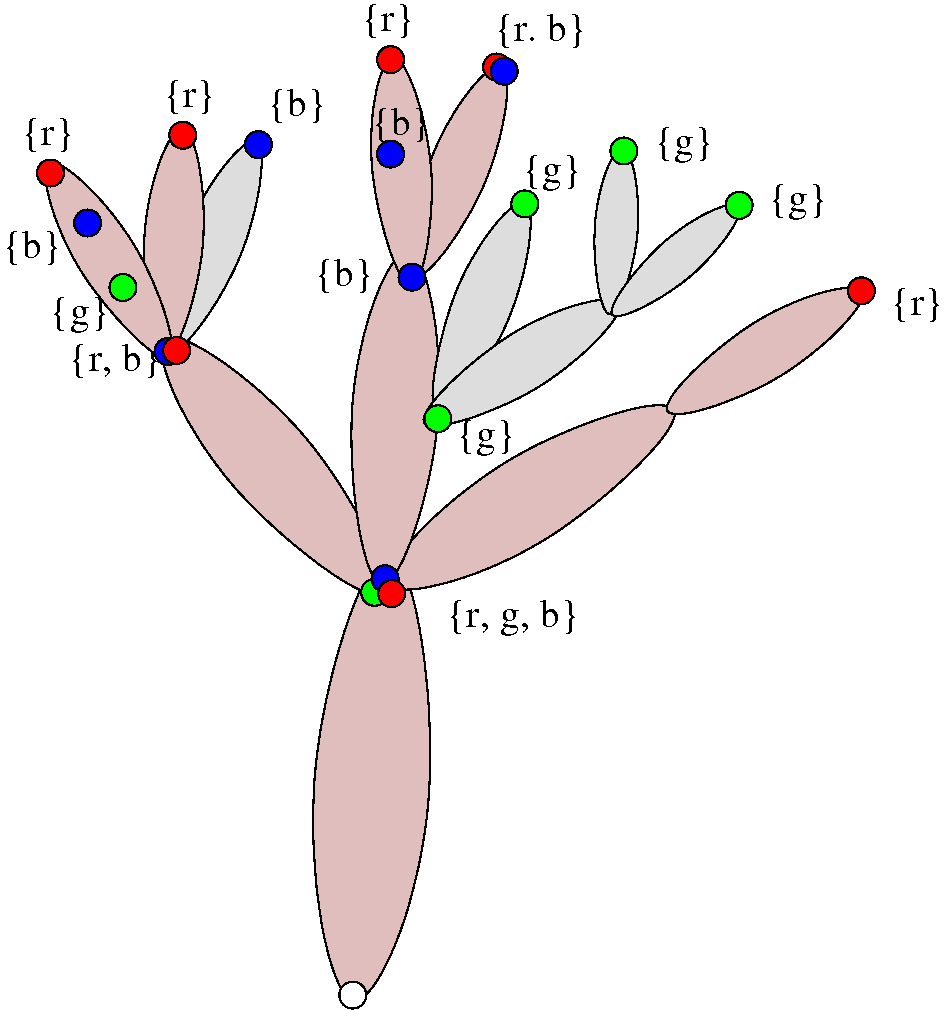}
    \end{center}
    \caption{A \textit{\{{\color{red} red}, {\color{green}green}, {\color{blue}blue}\}}-tree $G$. The elliptic shapes represent blocks, the white node is the root $r$.
             For each colour $c$ the blocks on the path from the root vertex to a vertex coloured $c$ 
             form a sequence of SP-networks joined at their poles and only the ``joints'' of these networks can have colour~$c$. All such
             blocks for any given colour $c$ form a $\{c\}$-tree which is a ``subtree'' of $G$, here the subtree for $c = red$ is highlighted.}
       \label{fig.RGB-tree}
\end{figure} 


Let $G$ be a $C$-tree with root $r$, for some $C \subseteq [k]$.
Then $G$ has a unique rooted block tree $T$ with root $r$. 
Consider any block $B$ of $G$. 
Denote by $r(B)$ the vertex of $B$ closest to $r$ in $G$.
For $v \in V(G)$, denote by $G_v$ the subgraph of $G$ induced on $v$ and the vertices of
all blocks of $G$ that are ancestors of $v$ in $T$ (if $v$ is not a cut vertex and $v \ne r$, then
it has no ancestors), with the label and colour from $v$ removed. 
We call the set of colours $\Col(v) \cup \Col(G_v)$ the \emph{type} of $v$ in $G$, and denote it by $type_G(v)$.
For any block $B$ of $G$ and any colour $c$, let $X_c(B)$ denote the set of vertices $v \in V(B)$ such that $c \in type_G(v)$. 
\begin{prop}
 If $G$ is a $C$-tree
 then for any block $B$ of $G$ and any $c \in C$ we have $X_c(B) \le 1$.
\end{prop}
\begin{proof}
    Let $G_B$ be the subgraph of $G$ induced on the vertices of $B$ and
    the vertices of all blocks that are ancestors of $B$ in the rooted block tree of $G$.
    It is easy to see that $G_B$ is  a $C'$-tree for some $C' \subseteq C$. Therefore
    it suffices to prove the claim in the case where $r(B)$ is the root $r$ of $G$.

    Suppose $X_c(B) \ge 2$. Then $G$ has a coloured minor isomorphic to a vertex-pointed 
    2-connected graph $H$
    obtained from $N(B)$ by setting $\Col(x)$ $= \Col(y) = \{c\}$ for two distinct vertices $x, y \in V(B) \setminus \{r\}$.
    Now $c$ is bad for $H$, since 
    adding a new vertex with at least three neighbours to 
    a 2-connected graph in $\ex K_4$ yields $K_4$ as a minor. It follows that $c$ is bad for $G$.
\end{proof}

\bigskip




For $k = 1, 2, \dots$, denote by $\cb_k$ the family of biconnected graphs in $\ca_{[k]}$, such that each vertex has at most one colour. Again, if $B \in \cb_k$ and $c \in [k]$, then there is exactly one vertex
in $B$ coloured~$c$. For $n=0,1,2$, $|\cb_{1,n}|$ is equal to $0, 1, 2$ respectively; also $\cb_{k,j} = 0$ if $j < k$.

For a set $C$ of positive integers, denote by $\hat{\ca}_C$ the set of all vertex-pointed graphs $G$,
such that $\Col(G) = C$ and which further satisfy the conditions (a) and (b)
of the definition of a $C$-tree. 
It is easy to check using
the definition that $|\hat{\ca}_{C, n}|$ is equal to $1$ and $4^{|C|} - 2^{|C|}$  for $n=0$ and $1$ respectively.
Each non-empty graph in $\hat{\ca}_C$ can be decomposed uniquely into a (coloured and pointed) root vertex $r$ and a set of graphs $G_1, \dots G_t$  where for $i = 1, \dots, t$, the graph
$G_i$ is a $C_i$-tree for some $C_i \subseteq C$, and $V(G_i) \cap V(G_j) = \{r\}$ for $i \ne j$.

\begin{prop}\label{prop.rootablejoin}
    Let $C_1, C_2$ be finite non-empty sets of positive integers.
    Suppose $G_1 \in \hat{\ca}_{C_1}$ and $G_2 \in \hat{\ca}_{C_2}$ have
    only their root vertex in common. 
    Then
    $G_1 \cup G_2 \in \hat{\ca}_{C_1 \cup C_2}$.
\end{prop}
\begin{proof} 
  It is easy to see that the condition (b) holds for $G$.
  Suppose (a) does not hold, i.e., $c \in C_1 \cup C_2$ is bad for $G$. Consider the graph $G^+$
  obtained by adding a new vertex $w$, connected to the root of $G$
  and each vertex coloured $c$. For $i = 1,2$ let $G_i' = G^+[V(G_i) \cup \{w\}]$.
  By Lemma~\ref{lem.2cuts} for some $i = 1,2$ we have $G_i' + rw$ has a subdivision of $K_4$, thus
  $G_i \not \in \hat{\ca}_{C_i}$. This is a contradiction.
\end{proof}

\bigskip


\begin{lemma}\label{lem.BAC} Let $C$ be a finite non-empty set of positive integers.
    The exponential generating functions
    of $\hat{\ca}_C$ and $\ca_C$ are related by
    \[
    \hat{A}_C = \sum_{S\subseteq C} (-1)^{|C|-|S|} 2^{|S|} \exp\left(\sum_{S' \subseteq S} A_{S'}\right).
    \]
\end{lemma}

Notice, that by definition $A_\emptyset(x) = 0$ and $\hat{A}_\emptyset(x) = 1$.

\medskip

\begin{proof}
    %
%
    For any set $S \subseteq C$, define 
    \[
      \tilde{\ca}_S = \cup_{T \subseteq S} \hat{\ca}_T,
    \]
    and note that $|\tilde{\ca}_{S, 0}| = 2^{|S|}$ . 
    Since $\ca_S \cap \ca_{S'} = \emptyset$ for $S \ne S'$, for any $C' \subseteq C$ 
    \[
    \tilde{A}_{C'} = 2^{|C'|} \exp \left(\sum_{S \subseteq C'} A_S\right).
    \]
    Fix a non-negative integer $n$, and for $S \subseteq C$ let $b_S = b_{S,n}$ be the number of graphs $G$ in $\tilde{\ca}_{C,n}$, such that $\Col(G) \cap S = \emptyset$. Then
    the number of graphs in $\tilde{\ca}_{C,n}$ where some colour from $C$ is missing,
    by the inclusion-exclusion principle is
    \[
    \sum_{S \subseteq C, S \ne \emptyset} (-1)^{|S|-1} b_S.
    \]
    Now $b_S = \left|\tilde{\ca}_{C\setminus S, n}\right|$, therefore, summing over all $n \ge 0$
    \begin{align*}
     &\tilde{A}_C - \hat{A}_C = \sum_{S \subseteq C, S \ne \emptyset} (-1)^{|S|-1} \tilde{A}_{C \setminus S}
     = \sum _ {S \subset C} (-1)^{|C|-|S|-1} \tilde{A}_S
    \end{align*}
    and
    \[
    \hat{A}_C = \sum _{S \subseteq C} (-1)^{|C|-|S|} \tilde{A}_S = \sum _{S \subseteq C} (-1)^{|C|-|S|} 2^{|S|} \exp\left(\sum_{S' \subseteq S} A_{S'}\right).
    \]
\end{proof}

\bigskip

%
%
 \begin{prop}\label{prop.multiconstruction}
    Let $k$ be a positive integer, and let $C$ be a set of positive integers, $|C| \ge k$.
    Let $B \in \cb_k$. For $i = 1, \dots, k$, denote by $v_i$ be the vertex of $B$ coloured $\{i\}$.
    Let $P$ be a partition of $C$ into $k$ non-empty sets $S_1, \dots, S_k$ (listed in the lexicographic order),
    and let $G_1, \dots, G_k$  
    be pairwise disjoint graphs, all disjoint from $B$,
    with $G_i \in \hat{\ca}_{S_i}$, $i = 1, \dots, k$.

    Then the graph $G$ obtained by identifying the  root $r_i$ of $G_i$
    with $v_i$ and setting $\Col_G(v_i) = \Col_{G_i}(r_i)$ for each $i$, is a $C$-tree.
\end{prop}

\begin{proof}
    We have to show the conditions (a)-(c) of the definition of the $C$-tree are satisfied. 
    It is trivial to check (b) and (c), so we will check just (a).
    Suppose it does not hold, i.e., $c \in C$ is bad for $G$. Consider the graph $G^+$ formed by adding to $G$ a new vertex $w \not \in V(G)$ and connecting $w$
    to the root $r$ of $G$ and each vertex coloured $c$. Let $S_i$ be the set containing $c$.
    Let $G_1' = G[V(G_i) \cup w]$ and $G_2' = G^+ - (G_i - v_i)$. We have $G_1' \cup G_2' = G^+$
    and $V(G_1') \cap V(G_2') = \{v_i, w\}$. Let $K'$ be a subdivision of $K_4$ in $G^+$.
    
    By Lemma~\ref{lem.2cuts}, either $K'$ is contained in $G_1'$ or $G_2'$, or the intersection of $K'$
    with $G_j'$ is a path from $v_i$ to $w$ for some $j \in \{1,2\}$. If $K'$ is a subgraph of $G_1'$,
    then $c$ is bad for $G_1$; if $K'$ is a subgraph of $G_2'$, then $c$ is bad for $B$.
    If $G_2' \cap K'$ is a path from $w$ to $v_i$, then $G_1' + w v_i$ contains a subdivision of $K_4$,
    so $c$ is bad for $G_i$. If $G_1' \cap K'$ is a path from $w$ to $v_i$, then $B + w v_i$
    contains a subdivision of $K_4$, and so $c$ is bad for $B$. In each case we get a contradiction.
\end{proof}

\bigskip

%
%

We can now use the above observations and the decomposition into blocks, similarly
as in (\ref{eq.unitypeblocks})
to give the exponential generating function for $\ca_C$.
Given a set $C \subseteq [k]$ for some positive integer $k$, let $\cp(C)$ be the set of all set partitions of $C$,
so that $|\cp([j])|$ is the $j$-th Bell number.
\begin{lemma}\label{lem.multitype}
    Let $k$ be a positive integer. For any non-empty set $C \subseteq [k]$, the exponential
    generating function of $\ca_C$ is
    \begin{equation*}
        A_C(x) = \sum_{P \in \cp(C)} B_{|P|}(x) \prod_{S \in P} \hat{A}_S(x).
\end{equation*}
\end{lemma}

\begin{proof}
    Each $C$-tree $G$ may be decomposed into the (uncoloured) block $B$
    containing its root $r$, and a set of graphs $G_v$, such that
    $v \in X = \cup_{c\in C} X_c(B)$. Since the graph $\hat{G}_v$ obtained from $G_v$ with vertex $v$ coloured
    $\Col_G(v)$ and its label removed, is isomorphic to a coloured minor of $G$, we have that $\hat{G}_v \in \hat{\ca}_{C'}$ where $C' = type_G(v)$.
    Let $v^{(1)}, \dots, v^{(t)}$ be the vertices of $X$ sorted according to their type in the lexicographic order,
    and for $v \in X$ let $ind(v)$ be the position of $v$ in this list.
    The graph $\tilde{B}$ obtained from $B$ by setting $\Col_{\tilde{B}}(v) = \{ind(v)\}$ for each $v \in X$ satisfies
    $\tilde{B} \in \cb_t$. Now using Proposition~\ref{prop.multiconstruction} we see that
    each graph in $\ca_C$ can be represented uniquely by and constructed from
    \begin{itemize}
        \item a root block $B \in \cb_t$, for some $t \in [|C|]$,
        \item a partition $P = \{S_1, \dots, S_t\}$ of $C$ into non-empty sets (indexed in the lexicographic order), and
        \item pairwise disjoint graphs $G_1, \dots, G_t$, all disjoint from $B$, where $G_i \in \hat{\ca}_{S_i}$ for $i = 1 ,\dots t$
    \end{itemize}
by identifying the root $r_i$ of $G_i$ with the vertex $v$ of $B$ coloured $\{i\}$, and colouring that vertex $\Col_{G_i}(r_i)$.
Now the exponential generating function is obtained in a standard way, see \cite{fs09}.
\end{proof}

\bigskip
We write for
shortness $\ca_R = \ca_{\{1\}}, \ca_{RG} = \ca_{\{1,2\}}$ and $\ca_{RGB} = \ca_{\{1,2,3\}}$. 
\begin{lemma}\label{lem.multitype_small} The exponential generating functions of $\ca_R, \ca_{RG}$ and $\ca_{RGB}$ satisfy 
    \begin{align*}
        &A_R = B_1 \hat{A}_R;
        \\ &A_{RG} = B_1 \hat{A}_{RG} + B_2 \hat{A}_R^2;
        \\ &A_{RGB} = B_1 \hat{A}_{RGB} + 3 B_2 \hat{A}_R \hat{A}_{RG}
      + B_3 \hat{A}_R^3.
    \end{align*}
    Here
    \begin{align*}
        &\hat{A}_R = 2 e^{A_R} - 1;
      \\&\hat{A}_{RG} = 4 e^{A_{RG} + 2 A_R} - 4 e^{A_R} + 1;
      \\&\hat{A}_{RGB} = 8 e^{A_{RGB} + 3 A_{RG} + 3 A_R} - 12 e^{A_{RG} + 2 A_R} + 6 e^{A_R} - 1.
    \end{align*}
\end{lemma}

\begin{proof}
    Notice, that by symmetry we have $A_C = A_{C'}$ whenever $|C'| = |C|$. 
    The lemma follows by Lemma~\ref{lem.BAC} and Lemma~\ref{lem.multitype}. 
\end{proof}
\bigskip

\subsection{Blocks of coloured trees (two colours)}\label{sec.T2}

In this section we present a decomposition for coloured graphs in the class~$\cb_2$. 
For a network $G$ with two poles, we denote by $G^+$ the network obtained by connecting 
the poles with an edge. We will say that a colour $c$ is \emph{bad} for a (coloured) network $G$,
if adding a new vertex $w$ to $G$ connected to 
the source 
of $G$ and each vertex coloured $c$,
we obtain a graph not in $\ex \cb$. 
Recall that in this section $\cb = \{K_4\}$.
\begin{lemma}\label{lem.S1}
    Let $\cs_1$ be the class of $\{0,1\}^1$-coloured 
    series $SP$-networks $G$ where exactly one internal vertex is coloured $\{red\}$, and 
    the colour red is good for $G^+$. 
    
    Each graph in $\cs_1$ admits a unique decomposition into two graphs in $\cd$ or a graph $\cb_2$ and a graph in $\cd$.
    The exponential generating function of $\cs_1$ is
    \[
    S_1(x) = D(x) (x D(x) + B_2(x)).
    \]
\end{lemma}


\bigskip

We will use the following simple observation.
%
%
%
\begin{lemma} \label{prop.networkize}
      Let $k$ be a positive integer and let $G$ be a $\{0,1\}^k$-coloured graph with one pointed vertex $r$ and exactly $k$ coloured vertices, so that
      for each $i \in [k]$ there is a vertex $v_i \in V(G)\setminus\{r\}$ coloured $\{i\}$.
      Denote by $G_c$ the network with source $r$ and sink $u$ obtained from $G$ by removing the colour and the label from the vertex $u$
      coloured $\{c\}$. 
      
      $G \in \cb_k$ if and only if for each $c \in [k]$ we have $N(G_c) \in \cp + \ce_2$.
\end{lemma}

   
   \begin{proof}
       If either $G \in \cb_k$ or $N(G_c) \in \cp + \ce_2$, then $G_c$ is biconnected, so  $G_c \not \in \cs$.

       ($\Rightarrow$) Suppose $N(G_c) \not \in \cp + \ce_2$ and let $u$ be the vertex coloured $\{c\}$.
                       Then $G_c + r u$ contains $K_4$ as a minor. We may replace $r u$ by the path $r w u$ where $w \not \in V(G)$ to see that the colour
                       $c$ is 
                       bad for $G$, a contradiction.

       ($\Leftarrow$) Suppose 
       we have $N(G_c) \in \cp + \ce_2$, but $c$ is not good for $G$.
                      Let $u$ be the vertex coloured $c$.
                      The assumption implies that with a new vertex $w \not \in V(G)$, the graph $G' = N(G + r w u)$ contains $K_4$ as a minor.
                      Then so does $G_c + r u$ and $N(G_c) \not \in \cd$, a contradiction. 
   \end{proof}
\bigskip

\begin{proofof}{Lemma~\ref{lem.S1}}
%
    Let $G \in \cs_1$, let $s$ and $t$ be its source and sink respectively, and let $v$ be the vertex coloured red.
    Then by Lemma~\ref{lem.SP}, $G$ may be decomposed into a sequence of (pairwise internally disjoint)
    networks $H_1, H_2, \dots, H_k \in (\cP + e)$ with $k \ge 2$ and vertices $x_1, \dots, x_{k-1}$,
    where $x_i$ is both the sink of $H_i$ and the source of $H_{i+1}$.

     Suppose $v$ is an internal vertex of some $H_j$, $2 \le j \le k$. For $j = 1,\dots, k$ denote by
 $s_j$ and $t_j$ the source and the sink of $H_j$ respectively (we have $s_1 = s$ and $t_k = t$). 
By Proposition~\ref{prop.SPcycle}, $H_j$ contains a cycle $C$ with vertices $v, s_j$ and $t_j$.
By Lemma~\ref{lem.SP} and Proposition~\ref{prop.SPpath}, there is a path $P_1$ from
$s$ to $s_j$ in $H_1 \cup \dots \cup H_{j-1}$, a path $P_2$ from $t_j$ to $t$ in
$H_{j+1}, \dots, H_k$ (which is trivial if $j = k$). Now the graph obtained from the union of $C, P_1, P_2$
and $rt$ demonstrates that the colour red is bad for $G^+$. Therefore $v$ cannot be an internal vertex 
of $H_j$, $j \ge 2$. 

So $v$ can have one of the following positions (and the cases are non-overlapping):
\begin{itemize}
    \item[(a)] $v = x_j$ for some $j \in [k-1]$;
    \item[(b)] $v$ is an internal vertex of $H_1$.
\end{itemize}
    Suppose first that (a) holds.
    Denote by $D_1$ the $SP$-network with source $s$ and sink $v$ obtained from the union (series composition) of $H_1, \dots, H_j$ and $x_1, \dots, x_{j-1}$.
    Denote by $D_2$ the network with source $v$ and sink $t$ obtained from the union (series composition) of $H_{j+1}, \dots, H_k$ 
    and the vertices $x_{j+1}, \dots, x_{k-1}$. By Lemma~\ref{lem.SP} and the comment thereafter, $D_1, D_2 \in \cd$.

    Now let $D_1', D_2' \in \cd$ be arbitrary, and let $G'$ be a network obtained by series composition of $D_1'$ and $D_2'$ by
    colouring the common pole $\{red\}$ and giving it an arbitrary label. Let $s$ be the source of $D_1'$ and $G'$, and let
    $t$ be the sink of $D_2'$ and $G'$.

    Lemma~\ref{lem.SP} and a comment after it, 
    the decomposition of a graph $\cs$ to a graph in $(\cp + \ce_2)$ and a 
    sequence of graphs in $\cz \times (\cp + e)$ is unique. 
    Therefore, if $G' \in \cs_1$ applying the decomposition of $G = G'$ into graphs $H_1, \dots, H_k$ as above,
    we recover $D_1 = D_1'$ and $D_2 = D_2'$.

    Let us check that $G' \in \cs_1$. Consider the network $\tilde{G}$ obtained from $G'^+$
    by making $v$ a sink and $t$ an internal vertex. $N(\tilde{G})$ is a parallel $SP$-network, since it is obtained
    by a series composition of the network $st$ and the SP-network $\overleftarrow{D}_2'$ and a parallel composition of the resulting network with
    the $SP$-network $D_1'$ (here $\overleftarrow{D}_2'$ denotes $D_2'$ with its source and sink swapped. This change of orientation does not change the type of
    the network). By Proposition~\ref{prop.v}, the red colour is good for $\tilde{G}$, and so it is good for $G'^+$, and $G' \in \cs_1$. 
    

Now consider the case (b). Let $\tilde{H}_1$ be a $\{0,1\}^2$-coloured graph obtained from $H_1$ by colouring
its sink green and assigning the label $x_1$. Since $\tilde{H}_1$ is a subgraph of $G^+$, if the red colour is bad for $\tilde{H}_1$, then it is also bad for $G$.
If the green colour is bad for $\tilde{H}_1$, then the path $P$ from $s$ to $x_1$ in $G^+$, where $P$ consists of the edge $s t$ and a path from $t$ to $x_1$ in $D = H_2 \cup \{x_2\} \cup \dots \cup\{x_{k-1}\} \cup H_k \in \cd$ shows that $G^+$ contains $K_4$ as a minor, which is a contradiction. 
It follows that $\tilde{H}_1 \in \cb_2$ and $D \in \cd$.

Now take an arbitrary graph $H' \in \cb_2$ with root $r$, an arbitrary network $D' \in \cd$ with source $s'$ and sink $t'$,
and identify the green vertex of $H'$ with the source of $D'$ (call this vertex $x$) to obtain
a network $G'$ with source $r$ and sink $t'$. 
Denote the vertex of $G'$ coloured $\{red\}$ by $v$.

It is easy to see using the decomposition given by Lemma~\ref{lem.SP} that if $G' \in \cs_1$, then the procedure described above applied with $G = G'$ recovers $H'$ and $D'$ of $G'$
as $\tilde{H}_1$ and $D$ respectively. It remains to show that $G' \in \cs_1$.

Consider the graph $G'' = G'^{+(w)}$ obtained by adding a new vertex $w$ to $G'^+$, such that $\Gamma(w) = \{r, v\}$. Assume that
red is bad for $G'$. Then $G''$ contains a subdivision $K'$ of $K_4$. Since $K'$ is 2-connected, it must contain both vertices $v$ and $r$.
Clearly, $\{r, x\}$ is a cut in $G$. Apply Lemma~\ref{lem.2cuts} to $G''$, and its subgraphs $\tilde{H}' = G''[V(H') \cup \{w\}]$ and $R = G'' - (\tilde{H}' - \{r,x\})$.
We consider three possible cases.

\emph{Case 1.} $K'$ is entirely contained in $\tilde{H}'$. Then $H' \not \in \cb_2$, which is a contradiction.

\emph{Case 2.} $K' \cap \tilde{H}'$ is a path from $r$ to $x$. Let $G_1$ be the network with poles $r$ and $x$ obtained by a series composition of $rt'$ and $\overleftarrow{D'}$. Then the network $G_1 + r x \in \cp$ contains $K_4$ as a minor, this is a contradiction to Lemma~\ref{lem.SP}.

\emph{Case 3.} $K' \cap R$ is a path $P$ from $r$ to $x$. Consider $K'_1 = K' \cap \tilde{H}'$. Then
               $K'_1$ is a subdivision of $K_4$ with a part of subdivided edge (i.e., the internal vertices of $P$) removed.
               
               It must be that $w \in V(K'_1)$, otherwise
               $N(G^+) \not \in \cd$, again contradicting Lemma~\ref{lem.SP}. Since $K'$ is 2-connected, $w$ must have degree 2 in $K_1'$. 
               We have that $K'_1$ is a subdivision of one of the graphs shown in Figure~\ref{fig.subdivs}, with the restriction that $rw$ cannot be subdivided.
               Importantly, in all cases, the graph $K'' = K'_1 - \{w, r, x\}$ is connected.
               The SP-network $H'$ is parallel (it contains an internal vertex $v$), so by Lemma~\ref{lem.SP}
               it can be obtained in a unique way by parallel composition of some $l\ge 2$ networks $S_1, \dots, S_l \in \cs \cup e$. 
               The connected graph $H''$ belongs to exactly one of these networks; change the indices if necessary, so that this network is $S_1$. Now, since $V(K') \cap V(S_2) \subseteq \{r, x\}$ we may use a path in $S_2$ from $r$ to $x$ to replace the path $P$ and show
               that $\tilde{H}'$ also contains a subdivision of $K_4$. 
               This demonstrates that $H' \not \in \cb_2$, which is a contradiction.
\begin{figure}
     \begin{center}
    \includegraphics[height=5cm]{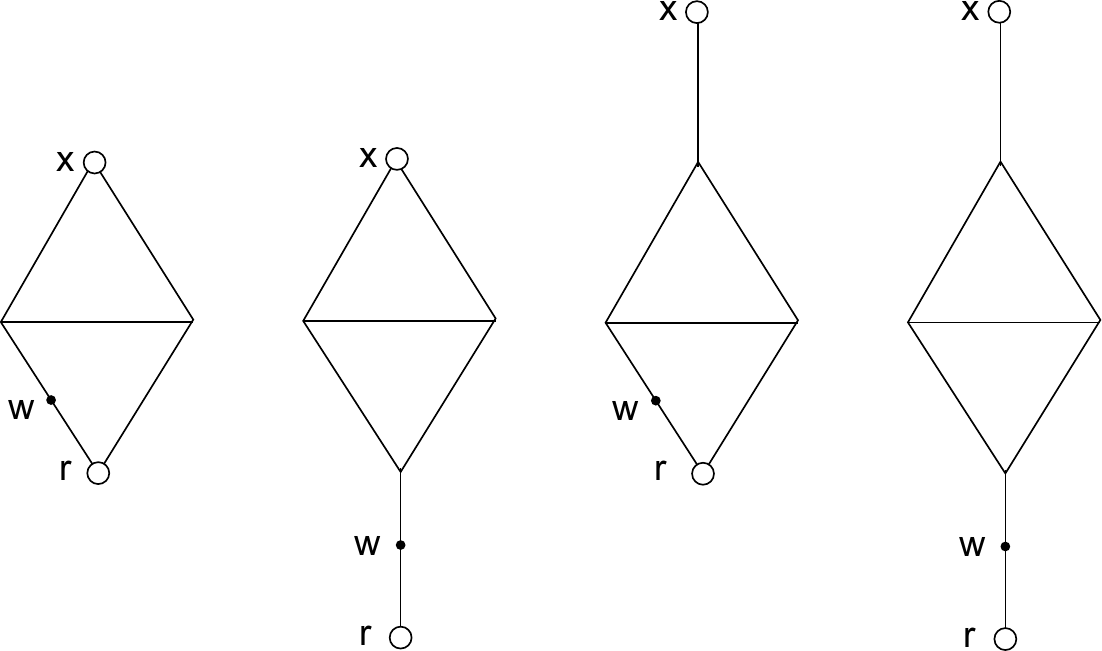}
    \end{center}
    \caption{The part $K'_1$ of the subdivision of $K_4$ contained in $\tilde{H}_1$ is a subdivision of one of these four types of graphs. Here the unlabelled vertices
             can be any vertices of~$\tilde{H}_1$.}
       \label{fig.subdivs}
\end{figure}

Combining the decompositions in each of the cases (a) and (b) yields the bijection
\[
  \cs_1 = \cz \times \cd ^2 + \cd \times \cb_2,
\]
which translates into the generating function, see \cite{fs09}, as claimed.
\end{proofof}

\begin{lemma}\label{lem.T2}
    Each graph in $\cb_2$ admits a unique decomposition into a network in $\cd$ and a network in $\cs_1$. The
    exponential generating function of $\cb_2$ satisfies
    \[
        B_2(x) = x D(x) S_1(x) 
    \]
\end{lemma}

   \begin{proof} 
       Let $G \in \cb_2$, let $r$ be the root of $G$ and let $u$ and $v$ be the vertices coloured $\{green\}$ and $\{red\}$ respectively.
       Consider the $\{0,1\}^1$-coloured network $\tilde{G}$ with poles $r$ and $u$
       obtained with $c = green$ as in Proposition~\ref{prop.networkize}. 
       $N(\tilde{G}) \in \cp$ (it has an internal vertex $v$),
       so it can be decomposed uniquely using Lemma~\ref{lem.SP} into a network $D \in \cd$ and a series graph $S$ that contains $v$,
       both with poles $s$ and $u$. Since the graph $S^+$ is isomorphic to a minor
       of $\tilde{G}$, if red is bad for $S^+$, then it is bad for $\tilde{G}$. Thus
       $S \in \cs_1$.

       Now take arbitrary networks $D' \in \cd$ and $S' \in \cs_1$ with disjoint
       sets of internal vertices
       and join them in parallel. 
       Label the sink vertex $u$ and colour it green. We claim that the resulting graph $G' \in \cb_2$.
       Suppose, not. The colour green cannot be bad for $G'$ by Proposition~\ref{prop.v}. Suppose red is bad for $G'$.
       The root $r$ of $G'$ is its unique pointed vertex (the common sink of $D'$ and $S'$).
       Using Lemma~\ref{lem.2cuts}, since $D' \in \cd$, we get that red is bad for $S' + r u$, which contradicts to the definition of $\cs_1$.
       So we have
       \[
         \cb_2 = \cz \times \cd \times \cs_2,
       \]
       and applying the standard conversion to generating functions \cite{fs09} completes the proof.
   \end{proof}

   \bigskip

   We may combine the results of this section to obtain the full picture of graphs in the class $\cb_2$, see Figure~\ref{fig.T2}.
   \begin{corollary} \label{col.T2}
  Each graph in $\cb_2$ admits a unique decomposition into three graphs in $\cd$ or two graphs in $\cd$ and a graph in $\cb_2$:
 \[
 \cb_2 = \cz^2 \times \cd ^ 3 + \cz \times \cd^2 \times \cb_2.
\]
\end{corollary}
\begin{proof}
  Combine Lemma~\ref{lem.S1} and Lemma~\ref{lem.T2}.
\end{proof}
   \begin{figure}
     \begin{center}
    \includegraphics[height=3cm]{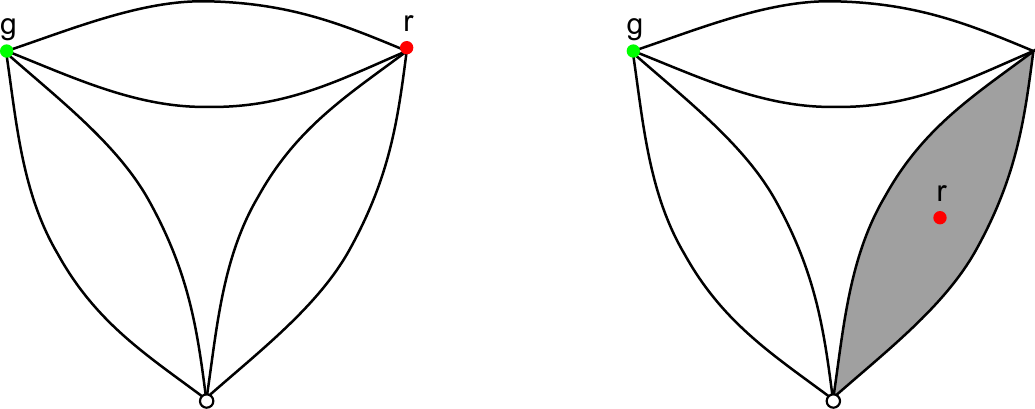}
    \end{center}
    \caption{ Graphs in the class $\cb_2$ can be decomposed into one of the following constructions. Here the white shapes represent networks in $\cd$,
             the grey shape represents a graph in $\cb_2$ where the green vertex is converted into the sink. $r$ and $g$ mark vertices coloured red and green respectively.}
       \label{fig.T2}
\end{figure}

\subsection{Blocks of coloured trees (general case)}

In this section we give a nice characterisation of the class of coloured blocks $\cb_k$ for arbitrary $k$.
It turns out that each graph in $\cb_k$ can be formed by substituting an SP-network for each edge of an ``apex tree''.

Let $k$ be a positive integer. Let $\ct_k'$ be the class of $\{0,1\}^k$-coloured Cayley trees $T$
containing exactly $k$ coloured vertices $v_1, \dots, v_k$, where vertex $v_i$ has colour $\{i\}$
for each $i = 1, \dots, k$ with the following restriction: if a vertex $u \in V(T)$ has no colour,
then it must have degree at least 3. Since every leaf is coloured by a unique colour, each $T \in \ct'_k$
has only one automorphism. Therefore, we have $|\ct_{k,n}'| = n!|(\cu\ct'_k)_n|$, where $\cu\ct'_k$ is the class of unlabelled $\{0,1\}^k$-coloured trees that can be obtained from the trees in $\ct'_k$. 

For $k = 1, 2, \dots$ the number of elements in $\cu \ct_k'$ is
\[
    1, 1, 4, 32, 396, 6692, \dots\footnote{The counts were fixed on 2019-07-13. A wrong sequence $1, 1, 4, 31, 367, \dots$ was published in the previous version of this paper and in the dissertation of the author. This only affected the last digit of $\gamma(\rd 5 K_4)$ given the table of Section~\ref{sec.conclusion}.}
\] 
 For example, all trees in the class $\cu \ct_3'$ are shown in Figure~\ref{fig.T'_3}.
Using The On-Line Encyclopedia of Integer Sequences \cite{sloane} we find that the sequence is known as A005263 and the trees in $\cu \ct_k', k \in \{1, 2, \dots,\}$ (with colours represented as integers) are called \emph{labelled Greg trees}.

Now let $\cf_k'$ be the class of all vertex-pointed graphs that can be obtained by taking a coloured tree $T \in \ct_k'$,
subdividing its edges arbitrarily (by inserting new uncoloured labelled vertices) to get a tree $T'$, and finally
adding a 
pointed
root vertex $r$ connected to each leaf of $T'$ and each uncoloured vertex of degree 2 (edges $rv$ where $v$ is coloured or has at least 3 neighbours may be included or not included). 

Let $\cf_k'(\cdot, \cd)$ denote
the class of graphs 
that can be obtained from graphs in $\cf_k'$ by replacing their edges by arbitrary networks
in $\cd$. 
Since each network has an orientation (it starts with its source and ends with its sink), in order for such replacement to be well defined for a given $G \in \cf_k'$ and $\{D_e \in \cd: e \in E(H)\}$, the edges of $G$ have to be oriented. 
We can assume that each edge in the tree $G-r$ points towards the red vertex, and each edge of $G$ adjacent to $r$ points away from $r$.
\begin{figure}
     \centering
     \includegraphics[height=5cm]{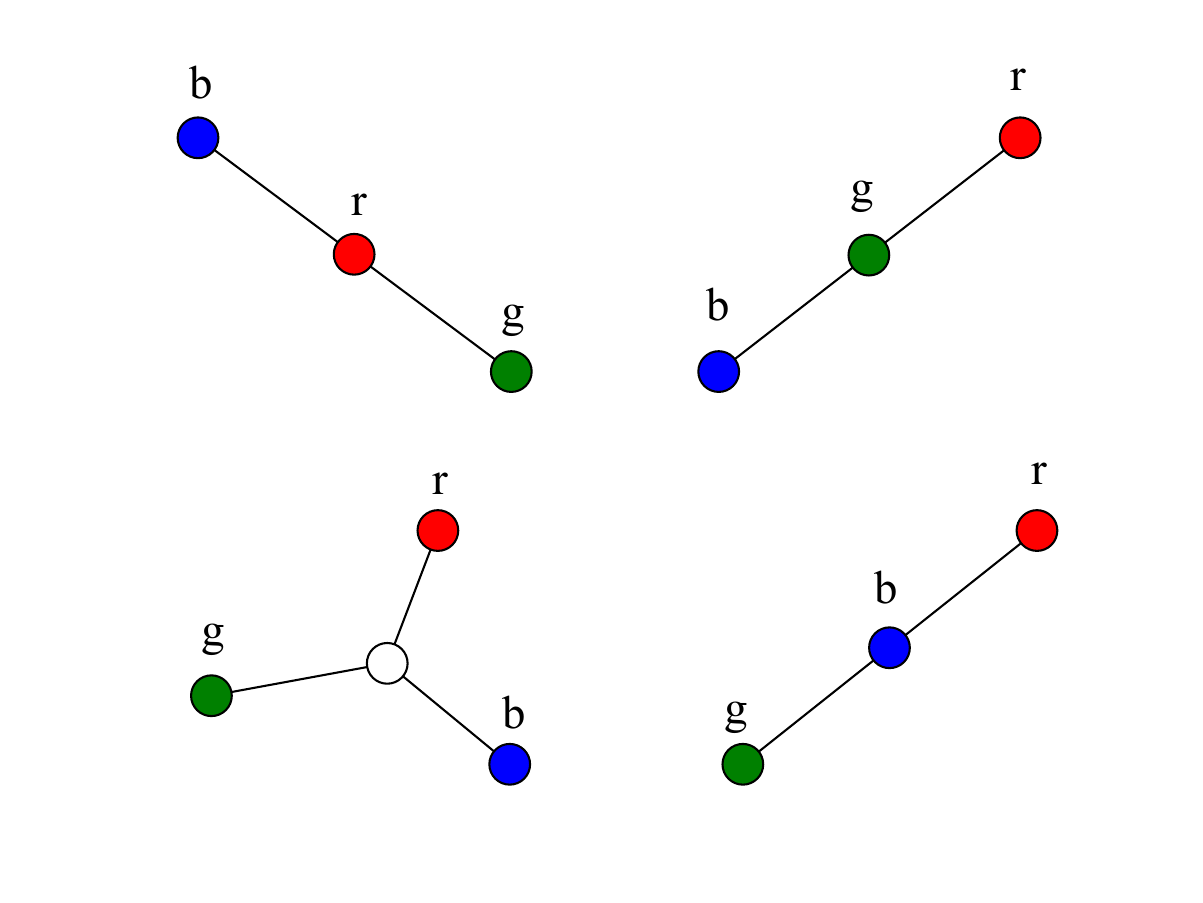}
    \caption{The isomorphism groups for the class of coloured trees $\ct_3'$ (equivalently, the labelled Greg trees with 3 black nodes).}
       \label{fig.T'_3}
\end{figure}
\begin{theorem}\label{thm.Tk}
  Let $k$ be a positive integer. Each graph in $\cb_k$ can be obtained in a unique way by substituting
an SP-network for each edge of a graph in $\cf_k'$:
\[
\cb_k = \left\{
  \begin{array}{l l}
    \cf_1'(\cdot, \cp+\ce_2) = \cz \times (\cp + \ce_2) & \quad \text{for $k=1$}\\
    \cf_k'(\cdot, \cd) & \quad \text{otherwise.}
  \end{array} \right.
\]
\end{theorem}
We will prove Theorem~\ref{thm.Tk} using the next two lemmas.

\begin{lemma}\label{lem.FkD}
    Let $k \ge 2$ be an integer. Each pair $(H, \cd_H)$, where $H \in \cf_k'$
    and $\cd_H = \{D_e \in \cd : e \in E(H)\}$ yields a unique graph $G=G(H, \cd_H) \in \cb_k$,
    where $G$ is obtained from $H$ by replacing $e$ with $D_e$ for each edge $e \in E(H)$. 
\end{lemma}

\begin{proof}
    Fix a pair $(H, \cd_H)$ 
    and let $r$ be the root of $G = G(H, \cd_H)$. Let us first prove that 
    $G \in \cb_k$. To this aim, by Lemma~\ref{prop.networkize} it suffices to show that for each colour $c\in[k]$, 
    if $v_c$ is the vertex coloured $\{c\}$, then
    the network $G_c$ with source $r$ and sink $v_c$ obtained from $G$ as in Lemma~\ref{prop.networkize} satisfies $N(G_c) \in \cp$.

    The graph $G_c$ is 2-connected, 
    so $N(G_c) \not \in \cs \cup \ce_2$.
    Suppose $G_c + r v_c$ contains a subdivision $K'$ of $K_4$. Let $B$ be the block of $G_c - r$ containing the 
    $2$-core
    of $K' - r$ (the $2$-core of a graph is the unique graph obtained by repeatedly deleting
    vertices of degree at most $1$, until no such vertices remain).
    Since each vertex $v \in V(H) \setminus \{r\}$ is a cut point of $G_c - r$, we have that $B$ is isomorphic to a subgraph of $D_e$ for some $e = xy \in H$. Furthermore, $G_c$ contains a path from the source to the sink of $D_e$ which does not
    use any internal vertex of $D_e$.
    It follows by Lemma~\ref{lem.2cuts} that $D_e$ is not an $SP$-network, a contradiction.
    
   Let us now prove that given $G = G(H, \cd_H)$ where $H \in \cf_k'$ for any $k \ge 2$
   we can always recover $H$ and $\cd_H$. We prove this claim by induction on $|V(H)|$. 

    Suppose $|V(H)| = 3$. Then the unique decomposition is provided by Corollary~\ref{col.T2},
    and the corresponding tree in $\cf_k'$ is the unique tree with two nodes $u$ and $v$, which are the vertices of $G$ coloured $\{green\}$ and $\{red\}$ respectively. 

    Now let $h \ge 4$ be an integer. Suppose that for any $\tilde{H} \in \cf_k'$ on at most $h-1$ vertices, we can always recover $\tilde{H}$ and $\cd_{\tilde{H}}$ given just the graph $G = G(\tilde {H}, \cd_{\tilde{H}})$. 
    Let $H \in \cf_k'$ be a graph on $h$ vertices with root $r$
    and let $\cd_H = \{D_e \in \cd : e \in E(H)\}$ be arbitrary.
    Denote the tree $H - r$ by $T$ (it is a subdivision of a graph in $\ct_k'$).
    Let $u$ be a coloured vertex in $G - r$,
    such that there are no two coloured components in $G - \{r, u\}$. If
   there is more than one candidate for $u$, let $u$ be such that $\Col(u) = \{j\}$
   has the largest~$j$.
    At least one such vertex
    exists since for any leaf $x$ of the tree $T$, all coloured vertices of $T-u$ (and also in $G - \{r, u\}$) are in a single component. Furthermore, for any vertex $u'$ of $T$ that is an internal vertex of $T$, 
   there are at least $2$ coloured components in $T - u'$ (and also in $G - \{r, u'\})$.
   So $u$ is the leaf vertex of $T$ with the largest colour index.

   Let $C$ be the component of $G - \{r,u\}$ containing all coloured vertices. Then the network
   $D_{ru}'$ with poles $r$ and $u$ obtained from $G - C$ is the network~$D_{ru}$. 

   Now consider the graph $\tilde{G} = G - (V(D_{ru}')\setminus \{r, u\})$.
%
   For each cut vertex $v'$ of $\tilde{G}$, let
   $C(v', u)$ denote the component of $\tilde{G} - v'$ containing $u$.
   Let $S$ be the set of cut vertices $v'$ such that either $v'$ is 
   coloured or $G - C(v', u )$ is 2-connected. Finally, call a vertex $v' \in S$ a \emph{candidate}
   if $G - C(v', u)$ does not contain any vertex from $S$. It is not difficult to see
   that there is exactly one candidate: the neighbour $v$ of $u$ in $T$.
%

   Let $D_{uv}'$ be the network obtained from 
   $\tilde{G} [V(C(v,u)) \cup \{v\}]$ by making $u$ the source and $v$ the sink.
   We can see that $D_{uv}' = D_{uv}$ (the orientation is correct, since by definition $j > 1$).
   
   Now consider the graph $\tilde{G} - V(D_{uv}' - v)$.
   If $v$ is not coloured, colour it $\{j\}$, and add a dummy path $v w r$, where $w$ is a vertex not in $\tilde{G}$. Denote the resulting graph by $\tilde{G}_2$.
   At the same time consider the graph obtained from $H - u$ by adding an edge $r v$, if it is not already there, and colouring $v$ with $\{j\}$, if it is not
   coloured.
   Let $P_{rwv}$ be the network with poles $r$ and $v$ obtained from the path $rwv$.
   If $rv \not \in E(H)$, let $\tilde{D}_{rv} = P_{rwv}$ 
   otherwise 
   define $\tilde{D}_{rv} = D_{rv} \cup \tilde P$.
   The graph $\tilde{G}_2$ can be obtained from $H'$ by replacing each edge by a network $D_e'$,
   where $D_e' = D_e$, if $e \in E(H') \setminus \{rv\}$ and $D_{rv}' = \tilde{D}_{rv}$.

   By induction (rename colours, if necessary) we may recover $H'$ and $\cd'$ uniquely. Now we see, that connecting $u$ to $v$ and $r$ in $H'$,
   returning the original colour to $v$, and removing $r v$ if $D_{rv}' = P_{rwv}$, recovers the graph $H$.
   For $e \in E(H) \setminus \{ru, rv, uv\}$, the network $D_e$ is the network $D_e'$ by induction,
   for $e \in \{rv, uv\}$, we have shown above that $D_e  = D_e'$. Finally, if $r v \in H'$,
   we obtain $D_{rv}$ as $D_{rv}' - w$.  
 \end{proof}

\bigskip

\begin{lemma}\label{lem.TFkD}
  For any integer $k \ge 2$ we have $\cb_k \subseteq \cf_k'(\cdot, \cd)$.
\end{lemma}

\begin{proof}
We use induction on $k$. 
%
%
For $k = 2$, fix $G \in \cb_2$. By Corollary~\ref{col.T2}, $G$ admits either a representation
by three SP-networks (we say that $G$ is of the first type) or by two SP-networks and a (smaller) network $G' \in \cb_2$ (we say that $G$ is of the second type), see Figure~\ref{fig.T2}.
Let $G_0 = G$. For $i = 0, \dots$, if the graph $G_i$ is of the second type, define $G_{i+1} = G_i'$. Let $j$ be the index
of the last $G_i$ that has been defined.

To prove the lemma for the case $k=2$ we apply induction on $j$. When $j = 0$, we have that $G$
is of the first type, so by Corollary~\ref{col.T2}, it is a triangle $H \in \cf_{2,2}'$ with each edge replaced by a network in $\cd$. Now let $j' \ge 1$, assume that the claim holds for $k=2$ and all $j \in \{0, \dots, j' - 1\}$, and suppose $j = j'$.
Then by Corollary~\ref{col.T2}, the graph $G = G_0$ is obtained
from $G_1$ by taking a series composition $D$ of two graphs $D_1, D_2 \in \cd$ with the common pole coloured
$\{green\}$, identifying the sink of $D$ with the green vertex $u$ of $G_1$ and removing the colour from $u$.
By induction, $G_1$ can be obtained  from a graph $H' \in \cf_2'$ 
by replacing each edge $e \in E(H')$ with a network $D_e' \in \cd$. 

Let $H$ be a graph obtained from $H'$ by inserting the vertex $u \not \in V(H')$, so that $u$ is connected to the green vertex $u'$ of $H'$ and the root, colouring 
$u$ $\{green\}$ and removing the colour from $u'$. Clearly, $H \in \cf_2'$.
Also, for $e \in E(H')$ define $D_e = D_e'$, and let $D_{ru} = D_1$ and $D_{uu'} = D_2$.
Thus $G_0$ can be obtained from the graph $H \in \cf_2'$, by replacing 
each edge $e \in E(H)$ with $D_e$. This completes the proof for the case $k = 2$.

Assume now that  we have proved the lemma for $\cb_l$ with $l \in  \{2, \dots, k-1\}$, and suppose $G \in \cb_k$, where $k \ge 3$. 
Let $u$ be the vertex of $G$ coloured $\{k\}$. Remove the colour from $u$ to obtain a graph $G' \in \cb_{k-1}$. Use induction
to find a graph $H' \in \cf_{k-1}'$ and a set of networks $\cd_{H'} = \{D_e': e \in E(H')\}$ such that $G'$ is the graph obtained
by replacing each edge $e$ of $H'$ by $D_e'$. Let $r$ be the root vertex of $H'$, write $T'  = H' - r$, and recall that 
$T'$ is a subdivision of a tree in $\ct'_{k-1}$. 

The vertex $u$ may have one of the following positions:
\begin{itemize}
    \item[(a)] $u \in V(T')$.
    \item[(b)] $u$ is an internal vertex of $D_e'$ for some $e = xy \in E(T')$.
    \item[(c)] $u$ is an internal vertex of $D_{rv}'$ for some $v \in V(T')$. 
\end{itemize}

The case (a) is easy: we let $H$ be the graph obtained from $H'$ by colouring $u$ with $\{k\}$, and let $\cd_H = \cd_{H'}$.

Consider the case (b). Suppose, $u$ is not a cut vertex of $D_e'$. By Lemma~\ref{lem.SP} and Proposition~\ref{prop.SPcycle},
$D_e$ contains a minor $M$ isomorphic to the triangle $K_3$, such that $x,y$ and $u$ all belong to different bags.
Now, since each component of $T' - xy$ contains at least one leaf of $T'$, there are paths $P_1$ and $P_2$ in  $G' - (D_e' - \{x,y\})$
from $r$ to $x$ and $y$ respectively.  Now $M$, $P_1$ and $P_2$ demonstrate that the colour $k$ is bad for $G$.
Thus, $u$ must be a cut vertex of $D_e'$. Let $H$ be the graph obtained from $H'$ by subdividing the edge $xy$ with the vertex $u$.
Let $D_{ux}$ and $D_{yu}$ be the networks with the common pole $u$ (the orientation may be reversed, if necessary), such that $D_e'$ results from the series composition of $D_{ux}$ and $D_{yu}$.
For $e \in E(H')\setminus \{e\}$, let $D_e = D_e'$, and define $\cd_H = \{D_e : e \in E(H)\}$. Then we have $G = G(H, \cd_H)$.

Now consider the case (c). Let $G_1$ be the graph obtained from $G[V(D_{rv}')]$, by colouring the vertex $u$ $\{green\}$ and the vertex $v$
$\{red\}$ and adding the edge $r v$, if $rv \not \in E(G)$.
For a $\{0,1\}^k$-coloured graph $H$ and $c \in [k]$, let $(H)_c = H_c$ be as in Lemma~\ref{prop.networkize}. By Lemma~\ref{lem.SP}, $G_1$ is 2-connected and $(G_1)_{red} \in \cp$. By
Lemma~\ref{prop.networkize}, $(G)_k \in \cp$, and since $G$ contains a path from $r$ to $v$
internally disjoint from $D_{rv}'$ which can be contracted to an edge $rv$, we get that 
$(G_1)_{green} \in \cp$. Therefore, applying Lemma~\ref{prop.networkize} second time, we see that $G_1 \in \cb_2$.

We have $N(G_1) \in \cp$ and $G_1 \in \cb_2$ by Lemma~\ref{prop.networkize}. 


Using the already proved case $k=2$, $G_1$ can be obtained from a graph $H_1 \in \cf_2'$ by replacing
each edge $e \in E(H_1)$ with an $SP$-network $D_e''$. Let $\tilde{H}_1$ be obtained from $H_1$
by setting $\Col_{\tilde{H}_1}(u) = \{k\}$ and $\Col_{\tilde{H}_1}(v) = \Col_G(v)$. 

Now, if $D_{rv}'' \in \ce_2$ and $rv \not \in E(G)$, let $H = H' \cup (\tilde{H}_1 - rv)$, 
otherwise, let $H = H' \cup \tilde{H}_1$. We can see that the tree $T = H - r$ is obtained from $T'$
by attaching at the vertex $v$ the graph $P = \tilde{H}_1-r$ (which is a path from $v$ to $u$). The vertex $u$ is coloured $\{k\}$ in $P$  
and for each vertex $x \in V(P) \setminus \{v\}$ there is an edge $rx \in E(H)$ as required by
the definition of $\cf_k'$. Since $v \in V(T')$ and $T' \in \cf_{k-1}'$, if $v$ is not coloured
it has degree at least $2$ in $T'$, and degree at least 3 in $T$. Hence $H \in \cf_k'$.
For $e \in E(H') \setminus \{rv\}$, define $D_e = D_e'$;
for $e \in E(H_1) \setminus \{rv\}$, let $D_e = D_e''$. Finally, if $rv \in E(G)$,  set $D_{rv} = D_{rv}'' - rv$.
Let $\cd_H = \{D_e: e \in E(H)\}$: we have proved that $G = G(H, \cd_H)$, as required. 
\end{proof}
\bigskip

\begin{proofof}{Theorem~\ref{thm.Tk}}
    The case $k=1$ follows by Lemma~\ref{prop.networkize}. 
    For $k \ge 2$, we combine Lemma~\ref{lem.FkD} and Lemma~\ref{lem.TFkD}.
\end{proofof}

\bigskip
Given a class $\ca$ of graphs and a parameter $X: \ca \to \mathbb{Z}_{\ge0}$, let $\ca_{n,k} = \ca_{n,k}^X$ denote 
the family of graphs $G \in \ca_n$ with $X(G) = k$. We call
\[
A(x,y) = \sum_{n\ge 0, k\ge 0} \frac {|\ca_{n,k}|} {n!} x^n y^k
\]
the \emph{bivariate generating function} of $\ca$ (where $y$ ``counts'' $X$). Below,
wherever $X$ is not specified, $y$ counts the number of edges, i.e. $X(G) = |E(G)|$.

It is not difficult to get the bivariate generating function $F_k'$ for $\cf_k'$, when $k$ is small.
\begin{lemma} \label{lem.F2F3} The bivariate generating functions of $\cf_2'$ and $\cf_3'$ are
    \begin{align*}
        &F_2(x,y) = \frac {x^2 y^3} {1-x y^2}; &F_3(x,y) = \frac {x^3 y^4 (3 - 2 x y^2) ( 1 + y)} { (1 - xy^2) ^ 3}. 
    \end{align*}
\end{lemma}
\begin{proof}
Denote by $\tilde{\ct}_k'$ the class of trees obtained by subdividing 
edges of trees in $\ct_k'$ arbitrarily.
The graphs in $\tilde{\ct}_2'$ are paths with coloured endpoints (the colours provide a unique orientation) and the univariate exponential generating function
\[
\tilde{T}_2'(x) = \sum_{n=2}^{\infty} \frac {n!} {n!} x^n = \frac {x^2} {1-x}.
\]
Each $T \in \tilde{\ct}_2'$ on $n$ vertices yields a unique fan $F \in \cf_2'$ with $2n - 1$ edges, so
\[
F_2'(x,y) = \frac {\tilde{T}_2'(x y^2)} y =  \frac {x^2 y^3} {1-x y^2}.
\]
Now consider $k=3$.
There are $3 n! (n-2)$ trees $T \in \tilde{\ct}'_{3,n}$ such that $N(T)$ is isomorphic to a path,
and $n! \binom {n-2} 2$ trees $T \in \tilde{\ct}'_{3,n}$ which are subdivided 3-stars, see Figure~\ref{fig.T'_3}.
Therefore the exponential generating function of $\tilde{\ct}_3'$ is
\[
\tilde{T}_3'(x) = \sum_{n=3}^\infty 3 (n-2) x^n + \sum_{n=4}^\infty \binom {n-2} 2 x^n = \frac {3 x^3} {(1-x)^2} + \frac {x^4} { (1-x)^3}.
\]
From each tree in $\tilde{\ct}_3'$ on $n$ vertices we can obtain two fans $F_1, F_2 \in \cf'_{3,n}$ with $2n - 1$ and $2n - 2$ edges
respectively (this is because the middle coloured vertex in the ``path'' case, and the centre of the star, in the ``star'' case
may or may not be connected to the root).
This yields the exponential generating function
\[
F_3'(x) = \tilde{T}'_3(x y^2) (y^{-1} + y^{-2}) =  \frac {x^3 y^4 (3 - 2 x y^2) ( 1 + y)} { (1 - xy^2) ^ 3},
\]
as claimed.
\end{proof}





\subsection{Growth of the class $\ca_R$}
\label{subsec.A_R}


We will use below the following fact about the class of SP-networks $\cd$.
\begin{lemma}{(Lemma 2.3 of \cite{momm05})}
 \label{lem.Datrho}
 Let $\cb$ be the class of biconnected series-parallel graphs. Then
\begin{align*} 
    &\rho(\cb) = \rho(\cd) = \rho(\cp) = \frac {(1+t_0) (t_0 - 1)^2}  {t_0^3} =0.1280.., \text{ and} 
 \\  &D(\rho(\cd)) = \frac {t_0^2} {1-t_0^2} = 1.8678..,
\end{align*}
 where $t_0 = 0.8070..$ is the unique positive solution of 
 \[
  (1-t^2)^{-1} \exp(- t^2 / (1+t)) = 2.
 \]
 \end{lemma}

By Lemma~\ref{lem.multitype_small} and Theorem~\ref{thm.Tk}
\begin{equation}\label{eq.A_R}
    A_R = B_1 (2 e^{A_R} -1) = x (P+1) (2 e^{A_R} -1).
\end{equation}
Notice, that if we add a new vertex $w$ to
any $\{red\}$-tree $G$, and connect it to the root and every vertex coloured red, we obtain
a 2-connected series-parallel graph $G'$. This follows directly from the definition of a $C$-tree:
if we delete any vertex $x \in V(G') \setminus \{w\}$, $w$ has a neighbour in each of the components
of $V(G) - \{x,w\}$, so $G' - x$ is connected. If we delete $w$ we obtain the connected graph $G$.
Thus each $\{red\}$-tree of size $n$ gives a unique 2-connected series-parallel graph of size $n+2$
(we may label the root and the new vertex $n+1$ and $n+2$ respectively). Thus,
if $\cb$ is the class of biconnected series-parallel graphs,
we have
$\rho(\cb) = \rho(\cd) \le \rho(\ca_R)$.
On the other hand, either looking at (\ref{eq.A_R}) or recalling that each $SP$-network yields a unique $\{red\}$-tree, we see that
$\rho(\ca_R) \le \rho(\cd)$. We conclude that $\rho(\ca_R) = \rho(\cd)$.

\begin{prop}\label{prop.rhoBk}
    For any positive integer $k$, we have $\rho(\cb_k) = \rho(\cd)$. 
\end{prop}
\begin{proof}
    By Lemma~\ref{prop.networkize}
    \[
      |\cb_{k,n}| \le n^k | (\cp + \ce_2)_{n-1}|,
    \]
    so $\rho(\cb_k) \ge \rho(\cp)$. By Theorem~\ref{thm.Tk},
    each graph, obtained from a non-series SP-network $G$ and a coloured path
    $P \in \tilde{\ct}_k'$, by identifying the first endpoint of $P$ with the sink of $G$ and adding an edge between the source of $G$ and second endpoint of $P$ (if it is not already there), is in $\cb_k$. So 
    \[
    |\cb_{k,n}| \ge (n)_k |(\cp+\ce_2)_{n-k}|,
    \]
    and $\rho(\cb_k) \le \rho(\cp)$. So $\rho(\cb_k) = \rho(\cp) = \rho(\cd)$ by Lemma~\ref{lem.Datrho}.
\end{proof}

\bigskip

We remind a definition from \cite{fs09}. Given two numbers $\phi, R$ with $R>1$
and $0 < \phi < \pi/2$, define
\[
\Delta(\phi, R) = \left\{z \in {\mathbb C} : |z| < R, z \ne 1, |arg(z-1)| > \phi \right\}.
\]
A domain is a \textit{$\Delta$-domain at 1} if it is $\Delta(\phi, R)$ for some $R$ and
$\phi$. For a complex number $\zeta \ne 0$, a $\Delta$-domain at $\zeta$
is the image by the mapping $z \to \zeta z$ of a $\Delta$-domain at 1.

For complex functions $f, g$ we write $f(z) = O(g(z))$ as $z \to z_0$ if
there are constants $C, \eps > 0$ such that $|f(z)| \le C |g(z)|$ for all $z$ with $|z - z_0| < \eps$.
%

The following fact is well known. 
\begin{lemma}\label{lem.cayley_tree}
    The exponential generating function 
    $R(x) = \sum_{n\ge1} \frac{n^{n-1} x^n} {n!}$ of rooted Cayley trees has a unique dominant singularity $e^{-1}$.
    $R(x)$ can be extended analytically to a $\Delta$-domain $\Delta$ at $e^{-1}$, such that for
    all $x \in \Delta$ we have $R(x) = x e ^ {R(x)}$ and
    for $x \to e^{-1}$, $x \in \Delta$ we have
    \begin{equation}\label{eq.Rsing}
      R(x) = 1 - \sqrt 2 (1 - ex)^{1/2} + O(1-ex).
  \end{equation}
  Furthermore, $R(x)$ is the unique solution $y(x)$ of $y = x e^y$, which is analytic at 0 and satisfies $R(0) = 0$.
\end{lemma}
\begin{proof}
    See, e.g., Theorem VII.3 of \cite{fs09} or Theorem 2.19 of \cite{drmota}.
    For extension to a $\Delta$-domain see, e.g., proof of Theorem~2.19 of \cite{drmota}.
    The identity $R(x) = x e ^ {R(x)}$ for $|x| < e^{-1}$ is shown, i.e., in \cite{fs09}. The identity
    then extends to the whole $\Delta$ domain  by the Identity principle
    (see, e.g., Theorem 8.12 of \cite{cplxnotes}). 
\end{proof}

\bigskip


Recall that when we omit ``$(x)$'' in identities involving exponential generating functions and do not mention otherwise,
we mean that they hold for some $\delta > 0$ and any $x \in \mathbb{C}$ with $|x| < \delta$.
(If each side is an exponential generating function of a combinatorial class, this means that the counting sequences
of both classes are identical.)
\begin{lemma}\label{lem.eval_A_R}
    We have
    \[
    A_R = R(2 B_1 e^{-B_1}) - B_1.
    \]
\end{lemma}
\begin{proof}
    We may rewrite (\ref{eq.A_R}) as
    \[
    \tilde{A}_R = E e^{\tilde{A}_R}
    \]
    where $\tilde{A}_R = A_R + B_1$ and $E = 2 B_1 e^{-B_1}$.
    By Proposition~\ref{prop.rhoBk}, $\rho(\cb_1) = \rho(\cd) > 0$, so $E$ is analytic at zero. Since $E'(0) = 2 |\cb_{1,1}| = 2 > 0$,
    $E$ has an analytic inverse $\psi_E$ at zero. 
    Thus there is $\delta > 0$ such that for all $u \in \mathbb{C}$ with $|u| < \delta$  
    we have
    \[
       f(u) = u e^{f(u)},
    \]
    where $f(u) = \tilde{A}_R(\psi_E(u))$.
    Since $f(0) = E(0) = \psi_E(0) = 0$, we conclude (using Lemma~\ref{lem.cayley_tree}) that $f(u) = R(u)$ for all $u$ with $|u| < \delta$. This implies that there is $\eps > 0$, such that for all $x \in \mathbb C$ with $|x| < \eps$ we have $\tilde{A}_R(x) = f(E(x)) = R(E(x))$, or
    \[
      A_R(x) = R(E(x))  - B_1(x).
    \]
    Since the two analytic functions are identical on an open disc, they are identical for all $x$ with $|x| < \rho(\cd)$.
\end{proof}

\subsection{Growth of the class $\ca_{RG}$}

In contrast to $\ca_R$, the exponential generating function of $\ca_{RG}$ has a dominant singularity smaller than $\rho(\cd)$.
%
%
%
\begin{lemma}\label{lem.A_RG}
    For $|x| < \rho(\cd)$ define a function $E = E(x)$ by
    \[
      E = 4 B_1 \exp \left(2 A_R - (4 e^{A_R} - 1) B_1 + (2e^{A_R} - 1)^2 B_2\right). 
    \]
    The equation $E(x) = e^{-1}$ has only one solution $x_0 = 0.086468..$ in the interval $(0, \rho(\cd))$ and
    $\rho(\ca_{RG}) = x_0$. 
\end{lemma}

\begin{proof}
    Combining Lemma~\ref{lem.multitype_small}, Lemma~\ref{lem.F2F3} and Theorem~\ref{thm.Tk} we get
    \begin{align}
       &A_{RG} = B_1 \left(4 e^{A_{RG} + 2 A_R} - 4 e^{A_R} + 1 \right) + B_2 (2e^{A_R} - 1)^2;  \label{eq.A_RG2}  
    \\ &B_1 = x (P+1); \qquad B_2 =  \frac {x^2 D^3} {1 - xD^2}. \nonumber
   \end{align}
    Denote
    \[
    B = 
    (4 e^{A_R} - 1)B_1 - (2 e^{A_R} - 1)^2 B_2.
    \]
    and
    \[
    \tilde{A}_{RG} 
    =  A_{RG} + B. 
    \]
    We may rewrite (\ref{eq.A_RG2}) as
    \begin{equation*}
        \tilde{A}_{RG} = E e^{\tilde{A}_{RG}}.
    \end{equation*}
    Since all of the functions $P, A_R, D, B_1, B_2$ have convergence radius $\rho(\cd)$, the function $E(x)$
    is analytic in the open disc $|x| < \rho(\cd)$. Furthermore, $E(x)$ is increasing
    for $x\in (0, \rho(\cd))$. To see this, notice that the Taylor coefficients
    of $E_1 = E_1(x)$ given by
    \[
    E_1 = 4 \exp \left ((2e^{A_R} -1)^2 B_2 \right)
    \]
    are non-negative, so $E_1(x)$ is continuously increasing for $x \in (0, \rho(\cd))$. Also, by (\ref{eq.A_R}) we have
    \[
    2 A_R - (4 e^{A_R} - 1) B_1  = - B_1.
    \]
    So
    \[
    E_2 = B_1 \exp \left(2 A_R - (4 e^{A_R} - 1) B_1 \right) = B_1 e^{- B_1}.
    \]
    The function 
    $B_1(x)$
    continuously increases as $x \in (0, \rho(\cd))$,
    since $B_1$ has non-negative Taylor coefficients, not all zero. Also, we have $B_1(0) = 0$.
    By Lemma~\ref{lem.Datrho}, 
     (\ref{eq.P})
    and numeric evaluation we get $B_1(\rho(\cd)) = 0.1929.. < 1$.
    Since the function $y(t) = t e^{-t}$ continuously increases for $t \in (0,1)$ we conclude that both
    $y(B_1(x))$ and $E(x) = E_1(x) E_2(x)$ continuously increase for $x \in (0, \rho(\cd))$.

    We now claim that
    \begin{equation} \label{eq.RA_RG}
        \tilde{A}_{RG}(x) = R(E(x)),
    \end{equation}
    where $R$ is the exponential generating function for rooted Cayley trees. To see why,
    first note that
    \[
    E'(0) = 4 (P(0) + 1) e^0 = 4
    \]
    and so, since $E$ is analytic at 0 and $E(0) = 0$, $E(x)$ has an analytic inverse $\psi_E(u)$ for
    $ | u| < \delta$, with some positive $\delta$, such that $\psi_E(u)=0$.
    For such $u$ we have
    \begin{equation}\label{eq.ftree}
      f(u) = u e^{f(u)}
    \end{equation}
    and we conclude as in the proof of Lemma~\ref{lem.eval_A_R} that for all $x \in \mathbb C$, $|x| < \rho$ where $\rho$
    is the radius of convergence of $R(E(x))$
    \[
    \tilde{A}_{RG}(x) = f(E(x)) = R(E(x)).
    \]
    Returning to $A_{RG}$ we have
    \begin{equation}\label{eq.eval_A_RG}
       A_{RG}(x) = R(E(x)) - B(x).
    \end{equation}
    Since $A_{RG}$ has non-negative Taylor coefficients, by Pringsheim's theorem (see, e.g., \cite{fs09}),
    it has a dominant singularity in $[0; \infty]$.  The function $E$ is continuously increasing for $x \in (0, 0.12] \subset (0, \rho(\ca_R))$
    and $E(0.12) = 0.6436.. > e^{-1}$, therefore 
    there is exactly one
    solution of $E(x) = 1/e$ in $(0, \rho(\cd))$; we call this solution $x_0$.
    Here we used (\ref{eq.D}), Lemma~\ref{lem.eval_A_R}, (\ref{eq.eval_A_RG}) in Maple, to get
    the numeric evaluation of $E(0.12)$ and solve $E(x) = 1/e$. (Let us
    note here that $D$ and $R$ have explicit functional inverses, see \cite{momm05, fs09},
    therefore $D, R, A_R, A_{RG}$ can be evaluated numerically at any point inside their disc of convergence).

    The function $A_{RG}$ is analytic for all $x < x_0$, and there is $\eps > 0$ such that $B$ and $E$ are analytic for all $x$ with $|x| < x_0 + \eps$.
    Using the fact that $E$ has an analytic inverse at $x_0$ (since $E'(x_0) > 0$) we conclude that $x_0$ must be a singularity of $A_{RG}$.
\end{proof}

\subsection{Growth of the class $\ca_{RGB}$.}

For $\ca_{RGB}$ we will apply a very similar analysis as in the previous section,
we only have to work with slightly longer formulas.
\begin{lemma}\label{lem.A_RGB}
    For $|x| < \rho(\ca_{RG})$ define 
    $E(x) = E_1(x) E_2(x)$ where
    \begin{align*}
        &E_1 = 4 \exp \{ 3 B_2 (2 e^{A_R} -1)^2 (4e^{A_{RG}+ 2 A_R} - 4e^{A_R} + 1) + B_3(2e^{A_R} - 1)^3\};  
        \\ &E_2 = 2 B_1 \exp \{3 A_{RG} + 3 A_R + B_1(6 e^{A_R} - 12 e^{A_{RG} + 2 A_R} - 1) \}.
    \end{align*}
    The equation $E(x) = e^{-1}$ has only one solution
    $x_1 = 0.044495..$ in the interval $(0, \rho(\ca_{RG}))$
    and $\rho(\ca_{RGB})  = x_1$.
\end{lemma}
\begin{proof}
    By Lemma~\ref{lem.multitype_small} we have
    \[
       A_{RGB} = 8 B_1 e^{A_{RGB} + 3 A_{RG} + 3 A_R} - B,
    \]
    where 
    \begin{align*}
        B &=   B_1 (12 e^{A_{RG} + 2 A_R} - 6 e^{A_R} + 1)
        \\ &- 3 B_2 (2e^{A_R} - 1)^2 (4e^{A_{RG}+ 2 A_R} - 4e^{A_R} + 1) - B_3 (2e^{A_R} - 1)^3.
    \end{align*}
    Setting $\tilde{A}_{RGB} = A_{RGB} + B$, we may rewrite this as
    \[
     \tilde{A}_{RGB} = E e^{\tilde{A}_{RGB}}.
    \]
    Notice, that by Lemma~\ref{lem.A_RG}, and Section~\ref{subsec.A_R}, $B_k$ (for $k \ge 1$), $A_R$ and $A_{RG}$ all convergence radius at least $\rho(\ca_{RG})$. Therefore
    $E$ is analytic (and continuous) at any point $x \in (0, \rho(\ca_{RG}))$,  
    We claim that $E(x)$ is increasing for $x \in (0, \rho(\ca_{RG}))$.
    To see why,
     recall the notation of Lemma~\ref{lem.multitype_small}, and note that 
    \[
    E_1 = 4 e^{3 B_2 \hat{A}_R^2 \hat{A}_{RG} + B_3 \hat{A}_R^3}
    \]
    is an exponential generating function for a class of combinatorial objects, so its coefficients are 
    non-negative, not all zero. Hence $E_1(x)$ is increasing for $x \in (0, \rho(\ca_{RG}))$.
    Now, by (\ref{eq.A_R}) and (\ref{eq.A_RG2}), the exponent in $E_2$ is
    \begin{align*}
       &3 A_{RG} + 3 A_R + B_1 \left(6 e^{A_R} - 12 e^{A_{RG} + 2 A_R} - 1\right)
    \\ &= 3 B_1 \left( 4e^{A_{RG}+ 2 A_R} - 4e^{A_R} + 1 \right) + 3 B_2 \left(2e^{A_R} -1\right)^2 
   \\  &+ 3 B_1 \left(2e^{A_R} - 1\right) + B_1 \left(6 e^{A_R} - 12 e^{A_{RG} + 2 A_R} - 1\right)
   \\ & = - B_1 + 3 B_2 \hat{A}_R^2.
    \end{align*}
    Thus
    \[
    E_2 = 2 B_1 e^{-B_1} e^{3 B_2 \hat{A}_R^2}
    \]
    is increasing for 
    $x \in (0, \rho(\ca_RG))$ using a similar
    argument as in the proof of Lemma~\ref{lem.A_RG}. 
    Since $E = E_1 E_2$, $E$ has the same property.
    Now since $E(0) = 0$, $E'(0) = 8 |\cb_{1,1}| = 8$, we have similarly
    as in the proof of Lemma~\ref{lem.eval_A_R}  
    \[
     \tilde{A}_{RGB}(x) = R(E(x))
    \]
    for all $x$ in the disc of convergence of $\tilde{A}_{RGB}$, or, equivalently
    \[
    A_{RGB}(x) = R(E(x)) - B(x).
    \]
    Now 
    using a numeric evaluation with 
    $0.08 < \rho(\ca_{RG})$ yields 
    $E(0.08) = 0.855.. > 1/e$,
    and
    $B(x)$ is analytic for $x < \rho(\ca_{RG})$.
    It follows similarly as in Lemma~\ref{lem.A_RG}
    that the smallest positive number $x_1$
    such that $E(x_1) = 1/e$ is a dominant singularity of $R(E(x))$.
    Numerically solving with Maple yields $x_1 = 0.044495..$. 
    (Here the numeric evaluation of $E(x)$ for $x \in (0, \ca_{RG})$ can
    be easily carried out using the inverse functions of $R$ and $D$,
    Lemma~\ref{lem.eval_A_R} and (\ref{eq.eval_A_RG}).)
    Since the convergence radius of $B(x)$ is at least $\rho(\ca_{RG})$,
    it follows that $x_1$ is the convergence radius of $\ca_{RGB}$.
%
%
\end{proof}

\subsection{Completing the proofs}

\begin{proofof}{Lemma~\ref{lem.part2main}}
    Let $\cf$ be the class of rooted series-parallel graphs.
    Bodirsky, Gim\'{e}nez, Kang and Noy \cite{momm05} showed that the
    functional inverse $\psi_F$ of $F$ satisfies
    \[
     \psi_F(u) = u e^{-B'(u)}
    \]
    where
    \[
    B(x) = \frac 1 2 \ln (1 + x D(x)) - \frac {x D(x) (x^2 D(x)^2 + x D(x) + 2 - 2 x)} {4 (1 + x D(x))}
    \]
    is the exponential generating function of biconnected
    series-parallel graphs. 
    Using Theorem~3.4 of \cite{momm05},
    $\psi_F(u)$ is continuously increasing for $u \in [0, u_0)$, where $u_0 = F(\rho(\cf)) = 0.127969..$ (denoted $\tau(1)$ in \cite{momm05}) and $\psi_F(u_0) = \rho(\cf) = 0.11021..$. 

    Recall that $\ca=\cc^{\bullet3}$ is the class of all $3$-rootable graphs rooted at a
    $3$-rootable vertex. Then (cf. (\ref{eq.unitypeblocks}) and the proof of Lemma~\ref{lem.BAC})
    \[
    \ca =  2^3 \times \cf \times \prod_{S \subseteq [3]} \SET(\ca_S(\cf)).
    \]
    Therefore
    \begin{equation} \label{eq.rrr}
    A(x) =  8 e^{A_{RGB}(F(x))} F(x) e^{\sum_{S\subset [3]} A_S(F(x))}. 
   \end{equation}
    By Lemma~\ref{lem.A_RGB} $\rho(\ca_{RGB}) = 0.044.. <  F(\rho(\cf)) = u_0 = 0.1279..$. 
    Let $\rho$ be
    the unique solution in $(0, u_0)$ of
    \[
       F(u) = \rho(\ca_{RGB}),
    \]
    so that $\rho = \psi_F(\rho(\ca_{RGB})) = 0.042509..$ 
    
    Since $F$ and $A_{RGB}$ have non-negative coefficients and $F'(\rho) \ne 0$,
    $\rho$ is a singularity of $e^{A_{RGB}(F(x))}$. 
    Moreover, since $\rho(\ca_S) > \rho(\ca_{RGB})$ for any $S \subset [3]$ by Lemma~\ref{lem.A_RG} and Lemma~\ref{lem.A_RGB}, we have that
    \[
    g(x) = F(x) e^{\sum_{S\subset [3]} A_S(F(x))}
    \]
    is analytic at $\rho$ and $g(\rho) \ne 0$. It follows, see \cite{fs09}, that the radius of convergence of $A(x)$ is $\rho$, and so $\gamupper(\ca) = \rho^{-1}= 23.524122..$. 
    By Lemma~\ref{lem.gcrooted}, $\rho^{-1}$ is the growth constant of $\ca = \cc^{\bullet3}$.

    Now 
     \[
       \gamupper(\cc^2) \le \gamma(\apex(\ex K_4)) \le 2 \gamma(\ex K_4) < \rho^{-1},
  \]
 since $\gamma(\ex K_4) = 9.07..$ by \cite{momm05}. Also $\gamma(\cc^2)$ exists and is equal to $\gamupper(\cc^2)$
 by Lemma~\ref{lem.gcrd}.
    Applying Lemma~\ref{lem.gcindstep} completes the proof.
\end{proofof}

\section{Counting tree-like graphs}
\label{sec.unrooting}
\subsection{Substituting edges, internal vertices and leaves of Cayley trees}

Let $\ct'$ be a class of trees. Let $\cd, \ci, \cL$ be arbitrary non-empty classes of labelled objects. As before,
we assume that all classes are closed under isomorphism of the labels.
Although we use the same symbol to denote the class of series-parallel networks, in this section $\cd$ will be an arbitrary class.
We will consider the class $\ct'(\cd, \ci, \cL)$ obtained from trees in $\ct'$ by attaching to leaves, internal vertices and edges objects from $\cL$, $\ci$, $\cd$ respectively. More precisely,  denote by $L(T)$ and $I(T)$ the sets of
all labelled leaves and labelled internal nodes of a tree $T$ respectively (in this section, a node of $T$ is called a leaf, if its degree is at most one; otherwise it is called an internal node).
Then $\ct'(\cd, \ci, \cL)$ is the class of all tuples $(T, \cd', \ci', \cL')$ where $T \in \ct'$, 
$\cd' = \{D_e : e \in E(T)\}$, $\ci' = \{I_v: v \in I(T)\}$ and $\cL' = \{L_v: v \in L(T)\}$ are families of objects from $\cd, \ci$ and $\cL$ respectively, and the sets of labels of each object in $\{T\} \cup \cL_T \cup \ci_T \cup \cd_T$ are pairwise disjoint.


Suppose $\ci, \cL$ are classes of vertex-pointed graphs, $\cd$ is a class of networks, and graphs in $\ct'$
have at most one pointed (unlabelled) vertex.
Then each object $\alpha = (T, \cd_T, \ci_T, \cL_T) \in \ct'(\cd, \ci, \cL)$ corresponds naturally to a graph $G(\alpha)$ defined as follows, see Figure~\ref{fig.leaf-attached}.
Starting with $T$, identify the pointed vertex of $L_v$ with the node $v$ of $T$ for each $v \in L(T)$, identify the pointed vertex of $I_u$ with the node $u$ of $T$ for each $u \in I(T)$ and replace
each edge $uv \in E(T)$ by the network $D_{uv}$. To carry out the last substitution, fix a rule for orientation of edges of $T$. For instance, identify the source and the sink of $D_{uv}$ with the smaller and the larger of $\{u,v\}$ respectively; a pointed vertex can be assumed to be smaller than any labelled vertex.

Let $\ct$ be the class of (unrooted) Cayley trees. For example, if $\cz_C$ is the class of graphs consisting of a single pointed vertex coloured $C$ and $\ce_2$ is the class of trivial networks of size $0$ containing a single edge, then the class $\ct(\ce_2, \cz_\emptyset, \cz_{\{red\}} \cup \cz_{\{green\}})$ is isomorphic to the class of all unrooted Cayley trees where the leaves are coloured either red or green.


For $\alpha \in \ct'(\cd, \ci, \cL)$, let $T(\alpha)$ denote the underlying tree $T$.
Our aim in this section is to enumerate general ``supercritical''
classes $\ct(\cd, \ci, \cL)$ and obtain results on the underlying tree size: an application of this will be one of the key elements in the proof of Theorem~\ref{thm.K4}. 

\begin{figure}
     \begin{center}
    \includegraphics[height=5cm]{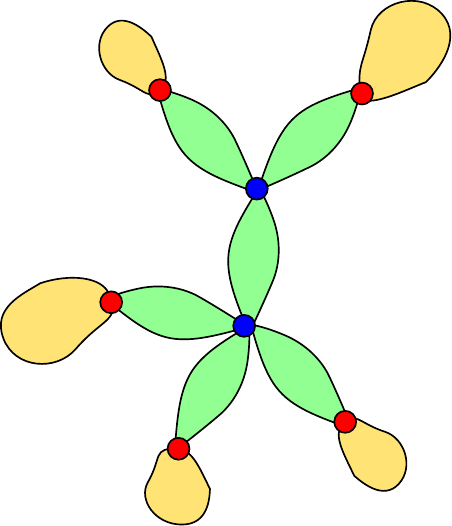}
    \end{center}
    \caption{Theorem~\ref{thm.treeleaf} characterises the growth of general ``supercritical'' class of graphs $\ct(\cd, \ci, \cL)$ obtained
    by replacing edges, internal nodes and leaves of Cayley trees by objects from classes $\cd, \ci$ and $\cL$ respectively.}
    \label{fig.leaf-attached}
\end{figure}
\begin{theorem}\label{thm.treeleaf}
    Let $\cd, \ci, \cL$ be non-empty classes of
    labelled objects.
    Let $\ca = \ct(\cd, \ci, \cL)$ and $\ca' = \cR(\cd, \ci, \cL)$. 
    Suppose $\rho = \rho(\ca) < \min(\rho(\cd), \rho(\ci), \rho(\cL))$ and 
    assume there are positive integers $i,j,k$ with $gcd(k-i,j-i) = 1$ such that 
    $\ca$ contains an object of each of the sizes $i, j$ and $k$.
    
    There are constants $a> 0$ and $c >0$ such that 
    the following holds. 
    Let $R_n \in_u \ca'$ or  $R_n \in_u \ca$ and let 
    $Y_n = |V(T(R_n))|$.
    \begin{enumerate}[ \indent 1)]
    \item For any $\eps > 0$, $\pr \left( \left |\frac {Y_n} n  - a \right| > \eps \right) = e^{-\Omega(n)}$;
    \item $|\ca_n| = (1 + o(1)) (an)^{-1} |\ca_n'| = c a^{-1} n^{-5/2} n! \rho^{-n} (1 + o(1))$;
    \item $A$ and $A'$ converge at $\rho$.
    \end{enumerate}
\end{theorem}
To prove the theorem, we will need some preliminary results and a technical lemma. Let $\cd, \ci, \cL$ be as
in Theorem~\ref{thm.treeleaf}.
Let $\ct_1$ be the class of Cayley trees pointed at a leaf and containing at least two vertices.
Consider the class $\ca_1 = \ct_1(\cd, \ci, \cL)$ with the bivariate generating function $A_1(x,s)$ where
the second variable $s$ counts the size of the underlying tree (which is the number of its nodes minus one). Then, writing $D=D(x)$, $I=I(x)$ and $L=L(x)$,
    \[
    A_1(x,s)  =  s x D L + s x D I \left(e^{A_1(x,s)} - 1 \right).
    \]
Consider additionally the class $\ca_2$ 
with specification
\[
  \ca_2 = \ca_1 - \cz \times \cd \times \cL + \cz \times \cd \times \ci,
\]
Alternatively, $\ca_2$ is the class $\ct_1'(\cd, \ci, \cL)$, where $\ct_1'$ is the same as $\ct_1$,
except that when we have a tree of size one (i.e., isomorphic to $K_2$), then its (unique) labelled vertex
$u$ is treated as an internal vertex and an object from $\ci$, rather than from $\cL$ is attached to it.

The bivariate generating function $A_2(x,s)$ of $\ca_2$  
satisfies
\begin{equation}\label{eq.bivbiv}
    A_2(x,s) = s x D I e^{sx D (L-I)} e ^ {A_2(x,s)}. 
\end{equation}

Call a class $\ca$ \emph{aperiodic}, if there are positive integers $i, j, k$ such that $i < j < k$,
 $|\ca_i|, |\ca_j|, |\ca_k| > 0$ and $gcd(k-i, j-i) = 1$.  
\begin{lemma}\label{lem.aperiodic}
    Let $\cd, \ci, \cL$ be non-empty classes of labelled objects.
    If any of the classes $\ca = \ct(\cd, \ci, \cL), \ca_1, \ca_2$ is aperiodic,
    then all of them are.
\end{lemma}

\medskip

\begin{proof}
    Given $\alpha \in \ct'(\cd, \ci, \cL)$,
    and $x \in V^*(T(\alpha)) \cup E(T(\alpha))$ we denote by $Obj_x(\alpha)$
    the object in $\cd \cup \ci \cup \cd$ associated with $x$. Here $V^*(T) \subseteq V(T)$ denotes
    the set of labelled vertices of $T$.

    \textit{1) Proof of $\ca$ aperiodic $\implies$ $\ca_2$ aperiodic.} 
    Consider an object $\alpha$ obtained from the path $P_2=uvw$ where $u$ and $v$ are poles,
    with associated objects $D_{uv}, D_{vw} \in \cd$, $I_v \in \ci$ such that
    the label sets of $D_{uv}, D_{vw}, I_v$ are pairwise disjoint, and disjoint from $u$.
    Let $\alpha_1, \alpha_2, \alpha_3 \in \ca$ be objects of sizes $i_1, i_2, i_3$ respectively, such that
    $\gcd(i_3-i_1, i_2-i_1) = 1$. Let $l \in \{1,2,3\}$; we can assume that the label set of $\alpha_l$ is disjoint
    from the label set of $\alpha$.  Construct a new object $\alpha_l'$ from $\alpha_l$ and $\alpha$
    as follows. If $|V(T(\alpha_l))| \ne 2$, then let $x$ be a vertex of $T$ with $d_T(v) \ne 1$. Merge
    $\alpha$ and $\alpha_l$ by identifying the vertex $u$ of $T(\alpha)$ with $x$.
    If $V(T(\alpha_l))$ has two vertices, say $a$ and $b$, then let $\alpha_l'$ be an object with underlying tree
    on edges $\{wv, va, vb\}$ by identifying $u$ and $a$, so that $Obj_{vb}(\alpha_l') = Obj_{ab}(\alpha)$ and other associations are inherited from
    $\alpha$ and $\alpha_l$, see Figure~\ref{fig.gcd}. 
    We have $|\alpha_l'| = |\alpha_l| + |\alpha|$, and $\alpha_l \in \ca_2$.
    Thus $\ca_2$ contains objects of sizes $i_l' = i_l + |\alpha|$ for $l = 1, 2, 3$ and $\gcd(i_3'-i_1', i_2'-i_1') = 
    \gcd(i_3 - i_2, i_2 - i_1) = 1$.

\begin{figure}
     \begin{center}
    \includegraphics[height=5cm]{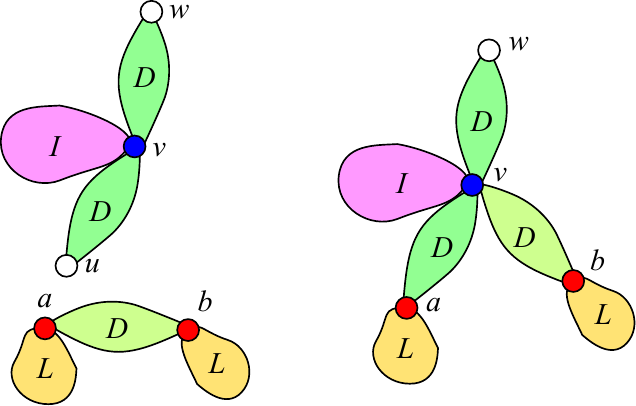}
    \end{center}
    \caption{Top left: the object $\alpha$, bottom left: the object $\alpha_l$ in the case $|T(\alpha_l)|=2$, 
    right: the object in $\ca_1$ of size $|\alpha| + |\alpha_l|$ obtained by merging $\alpha_l$ and $\alpha$.}
    \label{fig.gcd}
\end{figure}

    \textit{2) Proof of $\ca_2$ aperiodic $\implies$ $\ca_1$ aperiodic.} By definition, the only objects in the symmetric difference of $\ca_1$ and $\ca_2$
    are those, where the underlying tree has only one edge.

    Let $\alpha \in \ca_1$ be such that $|V(T(\alpha))| = 2$. $T(\alpha)$ contains one labelled and one 
    pointed vertex.
    Let $\alpha_1, \alpha_2, \alpha_3 \in \ca_2$ be objects of sizes $i_1, i_2, i_3$ respectively, such that
    $\gcd(i_3-i_1, i_2-i_1) = 1$. We may assume that the set of labels of $\alpha$ is disjoint from the
    set of labels of $\alpha_l$ for $l = 1,2,3$.
    From $\alpha_l$ we may obtain a new object as follows.
    Let $u$ be an internal vertex of $T(\alpha_l)$, if $|V(T(\alpha_l))| \ge 3$, otherwise,
    let $u$ be the unique labelled vertex of $T(\alpha_l)$.
    Merge the objects $\alpha$ and $\alpha_l$ by identifying the pointed vertex of $T(\alpha)$ with $u$,
    so that all the associated objects are inherited from the relevant tree. 
    In particular, associate with $u$ the object $Obj_u(\alpha_l) \in \ci$. Call
    the resulting structure $\alpha_l'$, and note that $|\alpha_l'| =  |\alpha_l| + |\alpha|$ and $\alpha_l' \in \ca_1$
    since $u$ is an internal vertex of $T(\alpha_l')$.
    Similarly as above, it follows that $\ca_1$ is aperiodic.

 \textit{3) Proof of $\ca_1$ aperiodic $\implies$ $\ca$ aperiodic.} From any object $\alpha \in \ca_1$
 we may obtain an object in $\ca$ by labelling the pointed vertex $u$ of $\alpha$ and
 associating with $u$ an object in $\cL$ of some fixed size. Now the claim follows similarly as in the previous cases.
%
%
\end{proof}

\medskip

\begin{lemma}\label{lem.samecr}
    Let $\cd, \ci, \cL$ be non-empty classes of labelled objects. For $\ca = \ct(\cd,\ci,\cL)$,
    $\rho(\ca) = \rho(\ca_1) = \rho(\ca_2)$.
\end{lemma}
\begin{proof}
    Constructions as in Lemma~\ref{lem.aperiodic} show that for 
    \[
      (\cc', \cc'') \in \{(\ca, \ca_2), (\ca_2, \ca_1), (\ca_1,\ca)\},
    \]
    there is a positive integer $s$, such that
    from any object in $\cc_n'$ we can construct a unique object in $\cc_{n+s}''$. So
    $|\cc_{n+s}''| \ge |\cc_n'|$, $\rho(\cc') \le \rho(\cc'')$ and the claim follows.
\end{proof}

\medskip


For the function $f = f(x,s)$ below we denote by $f_x$ and $f_s$ its
partial derivatives with respect to $x$ and $s$. 

\begin{lemma}\label{lem.treeleaf}
    Let $\cd, \ci, \cL$ be non-empty classes of labelled objects. 
    Let $\ca_2$ be the class with  the bivariate generating function 
    $A_2(x,s)$ given in (\ref{eq.bivbiv}).
    Suppose that 
    $\rho(\ca_2) < m = \min(\rho(\cd), \rho(\ci), \rho(\cL))$ and $\ca_2$ is aperiodic. 

    There is $\delta > 0$ such that the following holds. For any
    fixed $s \in [1-\delta; 1+\delta]$ we have
    \begin{equation}\label{eq.A2R}
    A_2(x,s) = R(f(x,s)),
\end{equation}
    where $f(x,s) = s x D(x) I(x) e^{ s x D(x)(L(x)-I(x))}$ and $R$ is the Cayley tree function.
    The function $A_2(x,s)$ has 
    a dominant singularity
    at $\rho(s)$, which is the
    smallest
    number in $(0, m)$ such that
    \[
    f(\rho(s),s) = e^{-1}.
    \]
    Let $\rho = \rho(1)$. 
    We have  $f_s(\rho,1) > 0$,  ${f_x(\rho,1)} > 0$ and $0 < \rho D(\rho) I(\rho) < 1$,
    $\rho(t)$ is continuously differentiable for $t \in [1-\delta, 1+\delta]$ and $\rho'(1) = - f_s(\rho,1)/f_x(\rho,1)$.
    Furthermore $A_2(x,s)$ is analytic in a $\Delta$-domain $\Delta'$ at $\rho(s)$ and for 
    $x \to \rho$, $x \in \Delta'$ we have 
    \begin{equation}\label{eq.A_2}
     A_2(x, s) = 1 - c(s) (1 - x/\rho(s))^{1/2} + O((1 - x/\rho(s)))
 \end{equation}
    where $c(s) =  (2 e \rho(s) f_x(\rho(s), s)) ^ {1/2}$ is positive. 
\end{lemma}

It is not difficult to modify the proof and show that $O()$ holds uniformly for some $\delta > 0$ and
$s \in [1-\delta, 1+\delta]$.

\medskip

\begin{proof}
    We will write, for shortness, $D = D(x)$, $I = I(x)$ and 
    $L = L(x)$. 

    Fix $s > 0$. We have $f(0,s) = 0$ and since $m > 0$,
    $f(x,s)$ is analytic at $0$. Define $F^{[s]}(z,w) = f(z,s)e^w - w$.
    Then, see (\ref{eq.bivbiv}), the points $(x, A_2(x,s))$ are solutions of $F^{[s]}(x, y) = 0$.
    We have  $R(f(0,s)) = 0$ and for $x$ in a neighbourhood of $0$, see Section~\ref{subsec.A_R},
    \[
    R(f(x,s)) = f(x,s) e^{R(f(x,s))},
    \]
    so $(x,R(f(x,s)))$ are also solutions of $F^{[s]}(x,y)  = 0$.
    Since the derivative of $F^{[s]}$ with respect to $w$ satisfies
    $F^{[s]}_w(0,0) = -1 \ne 0$ and $F^{[s]}(0,0) = 0$, (\ref{eq.A2R}) and the fact that $\rho(\ca_2) > 0$ follow by the Analytic 
    Implicit Function Theorem (Theorem~B.4 
    of \cite{fs09}) and the Identity principle.
    
    To prove the rest of the lemma we will apply
    the ``smooth implicit function schema'' 
    and a theorem of Meir and Moon \cite{fs09, meirmoon}. 
    The function $f(x,s)$ (and $F^{[s]}$) can have negative coefficients, 
    therefore we will work with the function $A_1(x,s)$, which
    satisfies
    \[
    A_1(x,s) = G^{[s]}(x, A_1(x,s)) \quad\mbox{where}\quad G^{[s]}(x,w) = s x  D L  + s x D I (e^w-1).
    \]
    (Alternatively, one could apply Theorem~2 of \cite{meirmoon} directly to~$F^{[s]}$.)


    Suppose $f(x,1) < e^{-1}$ for all $x \in (0, m)$.
    Then
    for any $x \in  (0, m)$,  $A_2$ 
    is analytic at $x$.
    Since by the Pringsheim's theorem (see \cite{fs09}), $A_2(x,1)$ has a dominant singularity in $(0, \infty)$,
    we conclude that $\rho(\ca_2) \ge m$, a contradiction. 
    By continuity of $f(x,1)$, there exists a smallest positive $\rho \in (0,m)$, such that
    $f(\rho,1) = e^{-1}$.

An important observation is that $\rho D(\rho) I(\rho) < 1$. Suppose, this is false.
$x D(x) I(x)$ continuously increases for $x \in (0,m)$ (it counts a non-empty combinatorial class),
so there is a unique positive $x_0  \in (0, \rho]$ such that $x_0 D(x_0) I(x_0) = 1$.
Since the function $h(z) = z e^{-z}$ is increasing for $z \in [0, 1)$ and 
$e^{x L(x) I(x)}$ is increasing for $x \in (0,m)$, we have that 
$f(x,1) = h(x D(x) I(x)) e^{x L(x) I(x)}$ is increasing for $x \in (0, x_0)$. 
But
\[
f(x_0,1) = e^{-1} e^{x_0 D(x_0) L(x_0)} = e^{-1} e^{L(x_0)/I(x_0)} > e^{-1},
\]
so $\rho < x_0$, a contradiction.
So $\rho D(\rho) I(\rho) < 1$, $\rho < m$, and we can further conclude
that $f(x,1)$ is continuously increasing
for $x \in (0; \rho + \eps_1)$ for some $\eps_1 > 0$.
This implies that 
$f_x(\rho,1) > 0$. Furthermore, for $|x| < m$
\[
f_s(x,s) = x D I e^{s x D(L-I)} (s x D L + 1 - s x D I),
\]
and so $f_s(\rho, 1) > 0$. 



Now consider a function $\tilde{F}(x,s) = f(x,s) - e^{-1}$, as a real function. Since
$\tilde{F}(\rho, 1) = 0$, $\tilde{F}_x(\rho,1) = f_x(\rho,1) > 0$
and $\tilde{F}_s(\rho,1) = f_s(\rho,1) > 0$, by the Implicit Function
Theorem (see, e.g., \cite{rudin}, Theorem~9.28), there is $\delta_1 > 0$
and a function $\rho(t): \mathbb{R} \to \mathbb{R}$, such that
for $t \in [1-\delta_1, 1+ \delta_1]$, $\rho(1) = \rho$, $\rho(t)$ is continuously differentiable,
$\tilde{F}(\rho(t), t) = 0$,
$\rho'(1) = - \frac {\tilde{F}_s(\rho,1)} {\tilde{F}_x(\rho,1)} = - \frac {f_s(\rho,1)} {f_x(\rho,1)}$,
and $\{(t,\rho(t)): t \in (1-\delta_1, 1+\delta_1)\}$ contains all the solutions of $\tilde{F}(x,t) = 0$
in the region $[\rho - \delta_1, \rho+\delta_1] \times [1-\delta_1, 1+\delta_1]$.

Since $\rho < m$, $\rho I(\rho) D(\rho) < 1$,
$\rho(t)$ is continuous at $t=1$ and $x I(x) D(x)$ is analytic at $x = \rho$, we see that
there is $\delta \in (0, \delta_1)$, such that for $t \in [1-\delta, 1+\delta]$,
$s \rho(t) I(\rho(t)) D(\rho(t)) < 1$ and $\rho(t) < m$. 
Now, since it is a product of two continuously increasing functions, 
$f(x, t) = h(t x D(x) I(x)) e^{t x D(x) L(x)}$ increases for $x \in (0, \rho(t))$, so
$\rho(t)$ is the smallest positive solution of $f(x,t) = e^{-1}$.


Assume $s \in [1-\delta, 1+\delta]$.
    Define $\tau(s) = 1 + s \rho(s) D(\rho(s)) \left(L(\rho(s)) -  I(\rho(s))\right)$. We claim that
    the function $G^{[s]}$ satisfies the ``smooth implicit function schema'' (Definition VII.4 of \cite{fs09}).
    Indeed, $\rho(s) < m$, $\tau(s) < \infty$, the condition ($I_1$) is satisfied, since $G$ is (bivariate) analytic for $|x| < m$ and $|w| < \infty$.
    The condition ($I_2$) follows since $G^{[s]}(0,0) = G_w^{[s]}(0,0) = 0$, and
    for any positive integer $m$,  $[w^m] G^{[s]}(x,w) = (m!)^{-1} s x D(x) I(x)$ has non-negative coefficients,
    not all zero, since $\cd, \ci$ are non-empty. 
    It remains to check the condition ($I_3$). 
    \[
    G^{[s]}(\rho(s), \tau(s)) = \tau(s)\quad\mbox{and}\quad G_w^{[s]}(\rho(s), \tau(s)) = 1.
    \]
    Both identities follow after a simple calculation using the definition of $\tau(s)$
    and the fact that $f(\rho(s), s) = e^{-1}$.
    
    Furthermore, by Lemma~\ref{lem.aperiodic},
    there are positive integers $i_1, i_2, i_3$, such that $|\ca_{1, i_l}| > 0$,
    $\gcd(i_3-i_1, i_2-i_1)=1$ and,
    denoting $\ca_{1, n, m}$ the set of objects $\alpha \in \ca_1$ with $|\alpha| = n$ and
    $|V(\alpha)| = m$,
    \[
    [x^{i_l}] A_1(x,s) = [x^{i_l}] (n!)^{-1} \sum_{m \ge 1} |\ca_{1, i_l,m}| s^m > 0
    \]
    for $l = 1,2,3$.
    Hence we can apply Theorem~VII.3 of \cite{fs09}, which yields that $\rho(s)$ is the unique dominant singularity 
    of $A_1(x,s)$
    and there is a $\Delta$-domain $\Delta$ at $\rho(s)$, such that for $x \to \rho_s$, $x \in \Delta$ 
    \[
    A_1(x, s) = \tau(s) - c(s) (1 - z/\rho(s))^{1/2} + O( (1 - x/\rho(s)))
    \]
    where
    \[
    c(s) = \left( \frac {2 \rho(s) G^{[s]}_x(\rho(s),\tau(s))} {G^{[s]}_{ww}(\rho(s), \tau(s))} \right)^{1/2} = (2 e \rho(s) f_x(\rho(s),s))^{1/2}.
    \]
    The last equality follows using $f(\rho(s), s) = e^{-1}$, $G^{[s]}_{ww}(\rho(s), \tau(s)) = 1$ and comparing
    $G^{[s]}_x(\rho(s), \tau(s))$ and $f_x(\rho(s),s)$ term by term.


    Finally, let us show that there is a $\Delta$-domain at $\rho(s)$, such that $y(x) = A_1(x,s)$ is analytic at each $x \in \Delta'$ . By Lemma~VII.3 of \cite{fs09}, $y$ is analytic in a region $D_0 = \{z \in {\mathbb C}: |z - \rho(s)| < r, |\arg(z) - \rho(s)| > \theta\}$ for some
    $r > 0$ and $0 < \theta < \frac \pi 2$. By Note VII.17 of \cite{fs09} $y$ is analytic at
    any $\xi$ with $|\xi| = \rho(s)$ and $\xi \ne \rho(s)$, i.e., there exists
    an open ball $B_\xi$ centered at $\xi$ and an analytic continuation of $y$ in $B_\xi$. 
    By compactness (see, e.g., proof of Theorem~2.19 of \cite{drmota}), we may pick a finite number of points $\xi$, such that the respective balls cover all points $x \in D_0$ with $|x| = \rho(s)$, and the union of these finite balls together with
    $\{z: |z| < \rho(s)\}$ contains some $\Delta$-domain $\Delta'$.

\end{proof}
\bigskip

\begin{proofof}{Theorem~\ref{thm.treeleaf}}
    In Example IX.25 of \cite{fs09}, Flajolet and Sedgewick give a  
    bivariate generating function $\tilde{R}(x,z)$ for rooted Cayley trees where the variable $z$ counts leaves (the root is counted as a leaf only in the case $n=1$)
\begin{equation}\label{eq.cayleyleaf}
    \tilde{R}(x,z) =  xz  + x \left(e^{\tilde{R}(x,z)} - 1 \right)
\end{equation}
which they show how to express using the usual (univariate) Cayley tree function
\begin{equation*}
    \tilde{R}(x,z) = x (z-1) + R(x e^{x (z-1)}).
\end{equation*}
Let $R(x,z)$ be the bivariate generating function
for rooted Cayley trees, where $z$ counts leaves, and the root is also counted as a leaf whenever its degree is at most one. Then, considering the
cases when the root has 0, 1 or 2 children separately and using $x e^{R(x)} = R(x)$ we get
\begin{align}
    R(x,z) &= zx +  z x \tilde{R}(x,z) + x(e^{\tilde{R}(x,z)} - \tilde{R}(x,z) -1) \nonumber
   \\      &=R\left(x e^{x(z-1)}\right) (x (z-1) + 1) + x^2(z-1)^2 + x(z-1). \label{eq.cayleyleafuni}
\end{align}
Let us add a variable $w$ that counts internal nodes:
\[
 R(x,w,z) = R(x w,z/w).
\]

And another variable $y$, that counts edges:
\[
 R(x,y,w,z) =  R(xyw, z/w) /y.
\]
We get using (\ref{eq.cayleyleafuni})
\[
R(x,y,w,z) = (xy (z-w) + 1) R \left(x y w e^{x y (z-w)} \right) y^{-1} + y (x(z-w))^2 + x (z-w).
\]

Then the bivariate generating function $\ca'$, where for $\alpha \in \ca'$, the variable $s$ counts the size of the underlying
tree $T(\alpha)$ is
\begin{align*}
    &A'(x,s) = R(sx, D, I, L) 
\\    &= (s x D(L-I) + 1)  \frac {R(f(x,s))} D 
      + (s x (L-I))^2 D + sx (L-I).
\end{align*}
Here $D = D(x)$, $I = I(x)$, $L = L(x)$ are the exponential generating functions of $\cd, \ci$ and $\cL$ respectively,
and $f(x,s) = s x D I e^{sxD(L-I)}$ as  before.

By Lemma~\ref{lem.aperiodic}, $\ca_2$ is aperiodic. Denote $m = \min (\rho(\cd), \rho(\ci), \rho(\cL))$. By Lemma~\ref{lem.samecr}, $\rho(\ca_2) = \rho(\ca) < m$.
Therefore we can apply Lemma~\ref{lem.treeleaf} to the class $\ca_2$ and its bivariate generating function
$A_2(x,s) = R(f(x,s))$. 

Let $\rho = \rho(1) > 0$ be as in Lemma~\ref{lem.treeleaf}. 
Our constant $a$ will be
\[
    a= -\frac {\rho'(1)} \rho = \frac {f_s(\rho,1)}  {\rho f_x(\rho,1)}.
\]
By Lemma~\ref{lem.treeleaf}, $a$ is positive.

Fix a small $\eps \in (0, \min(0.5, 0.5 a))$. Let $R_n' \in_u \ca'$ be a uniformly random
construction of size $n$ and let $X_n' = |V(T(R_n'))|$. The key part of the proof will be to show that
\begin{equation}\label{eq.nrootable}
\pr( |X_n' - an| > \eps n) = e^{-\Omega(n)}. 
\end{equation}
Let $\delta$ be given by Lemma~\ref{lem.treeleaf} applied with $\cd, \ci, \cL$.
We can assume that $\delta < \min \left(\frac 1 2, \frac {\eps} {16 a }, \frac \eps {16 a^2}\right)$. 
Fix $ s \in \{1 -\delta, 1, 1+\delta\}$.
By Lemma~\ref{lem.treeleaf}, we may also assume that $\delta$ is small enough that
\begin{align}
|\rho(s) - \rho - \rho'(1)| < \frac {\eps \rho \delta} 2 \quad\mbox{and}\quad 0 < \rho(s) < m. \label{eq.rhoisgood}
\end{align}

We can write $A'(x,s) = E_1 D^{-1} A_2(x,s) + E_2$ where
\begin{align*}
    &E_1 = E_1(x,s) = {s x D (L - I) + 1};
 \\ &E_2 = E_2(x,s) = (s x (L - I))^2 D + sx (L-I).
\end{align*}
Using (\ref{eq.rhoisgood}) $D(\rho(s)) > 0$, so $E_1 D^{-1}$ and $E_2$ are analytic at $x = \rho(s)$. 
So we have 
as $x \to \rho(s)$
\begin{align*}
    &E_1 D^{-1} = E_1(\rho(s)) D(\rho(s))^{-1} + O(x - \rho(s)),
    &E_2 = E_2 (\rho(s)) + O(x - \rho(s)).
\end{align*}
Using Lemma~\ref{lem.treeleaf}, $R(x) = x e^{R(x)}$, the fact that $m < \rho(s)$ and writing
\[
    A_2(x,s) = E_1 D^{-1} R(f(x,s)) + E_2 = E_1 s x I e^{s x D (L -I) + R(f(x,s))} + E_2
\]
we get (see \cite{fs09}) that $\rho(s)$ is the unique dominant singularity of $A'(x,s)$ 
and there is a $\Delta$-domain $\Delta'$ at $\rho(s)$, such that for $x \to \rho(s)$, $x \in \Delta'$
    \[
        A'(x,s) = c_0(s) - c_1(s) \left(1 - \frac x {\rho(s)} \right)^{1/2} + O \left(1-\frac x {\rho(s)}\right),
   \]
   with $c_1(s) = E_1(\rho(s),s) c(s) D(\rho(s))^{-1}$, $c_0(s) = c_1(s) + E_2(\rho(s),s)$.


Now by the ``Transfer method'' of Flajolet and Odlyzko (Theorem~VI.1 and Theorem~VI.3 of~\cite{fs09})
\[
    [x^n/n!] A'(x,s) = \frac {c_1(s)} {2 \sqrt \pi} n^{-3/2} \rho(s)^{-n} \left( 1 + O(n^{-1/2}) \right),
\]
the probability generating function of $X_n'$ at $s$ 
satisfies
\[
    \E s^{X_n'} = \frac{[x^n] A'(x,s)} {[x^n] A'(x,1)} = \left(\frac \rho {\rho(s)} \right)^n (1 + O(n^{-1/2})), 
\]
By Markov's inequality, for $s = 1-\delta$ 
\begin{align*}
    &\pr(X_n' \le (a-\eps) n) = \pr(s^{X_n'} \ge s^{(a-\eps) n}) \le \frac {\E s^{X_n'}} {s^{(a-\eps) n}}
    \\ &= \exp \left( (p_1 - (a - \eps) ) n \ln s + o(1) \right) = e^{-\Omega(n)}.
\end{align*}
since by (\ref{eq.rhoisgood}) and our choice of $\eps, \delta$ 
\begin{align*}
    &p_1 = \frac{ \ln \rho - \ln \rho(1-\delta)} {\ln (1-\delta)} \ge \frac{\ln (1 + a \delta - \delta\eps/2)} {- \ln (1-\delta)}
    \ge \frac {\ln(1 + a \delta - \delta \eps/2)} {\delta + \delta^2} 
    \\ &\ge \frac { a \delta  - \delta\eps/2 - (a \delta)^2} {\delta + \delta^2}
    \ge \left( a - \eps/2 - \delta a^2 \right) (1 - \delta)
    \\ & \ge a - \eps/2 - \delta a^2 - \delta a > a - \eps.
\end{align*}
Here we used simple inequalities $b - b^2 \le \ln (1 + b) \le b$
and $1/(1+b) \ge 1 - b$, which are valid for any $b \in (-0.5, 0.5)$.

Now taking $s = 1 + \delta$, we get by Markov's inequality
\begin{align*}
 & \pr(X_n' \ge (a+\eps)n) \le \frac {\E s^{X_n'}} {s^{(a+\eps)n}} = \left(\frac {\rho} {\rho(s) s^{a+\eps}}\right)^n (1+o(1))  
    \\ &= \exp \left( (p_2 - (a + \eps)) n \ln s + o(1) \right) = e^{-\Omega(n)} 
\end{align*}
since similarly as above
\begin{align*}
    &p_2 = \frac{\ln \rho - \ln \rho(s)} {\ln (1 +\delta)} 
    \le \frac{-\ln(1-a\delta-\eps \delta/2)} {\ln(1+\delta)} 
    \le \frac {a\delta + \delta\eps/2 + (a\delta + \delta\eps)^2} {\delta - \delta^2}
    \\ &\le (a+\eps/2 + 4 a^2 \delta) (1 + 2 \delta) \le a+\eps/2 + 
       4 a^2 \delta + 4 a \delta  < a + \eps.
\end{align*}

This completes the proof of (\ref{eq.nrootable}) and yields 1) with $Y_n = X_n'$. 

Let us finish the proof of the theorem. 
Let $a_n(k,l)$ (respectively, $a_n'(k,l)$) be the number of objects $\alpha$ in $\ct(\cd, \ci, \cL)_n$ (respectively, $\cR(\cd, \ci, \cL)_n$)
such that $|V(T(\alpha))| \in [k,l]$. Also let $A_n = a_n(1,n)$ and $A_n'  = a_n'(1,n)$.
%
Since each unrooted tree $T$ corresponds to exactly $|V(T)|$ rooted trees
\begin{equation}\label{eq.rootunroot}
A_n \le A_n' \le n A_n \quad \mbox{and} \quad k a_n(k,l) \le a_n'(k,l) \le l a_n(k,l) .
\end{equation}
By (\ref{eq.nrootable}) we have
\[
  d_n = A_n' \pr (|X_n' - an| > \eps n) = A_n' e^{-\Omega(n)}.
\]
For $R_n \in_u \ca$, let $X_n = X(R_n)$. The fact that 1) holds for $Y_n = X_n$ follows since
\[
\pr \left( |X_n - an| > \eps n \right) \le \frac { n d_n } {A_n'} = e^{-\Omega(n)}.
\]
Let us now show 2). 
Observe that since $X_n' \le n$, by \ref{eq.nrootable} it must be $a \in (0,1]$.
Let $\eps' = \min(\eps, 1-a)$. 
By (\ref{eq.rootunroot}) 
\[
\frac { a_n'\left( (a-\eps)n, (a+\eps')n \right)} { (a+\eps') n} \le a_n\left( (a-\eps)n, (a+\eps')n \right) \le \frac { a_n'\left( (a-\eps)n, (a+\eps')n \right)} { (a-\eps) n}. 
\]
So
\[
a_n\left( (a-\eps)n, (a+\eps')n \right) \in \left( \frac { (A_n' - d_n) (1-2\eps/a)} {an}, \frac {A_n' (1+2\eps/a)} {an} \right)
\]
and also 
\[
A_n - a_n\left( (a-\eps)n, (a+\eps')n \right) \le d_n = e^{-\Omega(n)} A_n'. 
\]
So 
\[
\left| A_n - \frac {A_n'} {an} \right| \le \frac{ A_n'} {an} \left( \frac {2 \eps} {a} +  e^{-\Omega(n)} \right).
\]
Letting $\eps$ go to zero shows that $A_n \sim \frac {A_n'} {an}$. Thus 2) follows with 
\[
    c = \frac {c_1(1)} {2 \sqrt \pi} = \frac{\rho D(\rho) (L(\rho) - I(\rho)) + 1} {D(\rho)} 
                                      \left( \frac {e \rho f_x(\rho,1) } \pi \right)^{1/2}.
\]
Finally, since $R(e^{-1}) = 1$, we have that 
\[
A'(\rho) = A'(\rho, 1) = c_0(\rho) + c_1(\rho)
\]
is finite. 
Using Lemma~\ref{lem.samecr} we have $\rho(\ca) = \rho$. By (\ref{eq.rootunroot}), the coefficients
of $A(x)$ are dominated by the coefficients of $A'(x)$ and so $A(\rho) \le A'(\rho) = c_0(\rho) + c_1(\rho)$.
\end{proofof}

\subsection{The case $\cb = \{K_4\}$}

\label{subsec.unrootedstr}
In this section we will have $\cb = \{K_4\}$ fixed and $l$ a positive integer.
Recall from the proof of Lemma~\ref{lem.unrootable},
that $G \in \cc^l$ is called nice if there is a vertex $x \in V(G)$
such that $G - x$ has at least two components containing all colours $[l]$. 
In this case, we call the vertex $x$ \emph{nice} in $G$.
For $G \in \crd l$ we say that a vertex $x$ is nice if it is
nice in its connected component.
Also recall that by $\cu = \cu^{<l>}$ we denote the class of graphs in $\cc^l$ that are not nice.
By Lemma~\ref{lem.unrootable}, Lemma~\ref{lem.gcrooted} and Lemma~\ref{lem.gcindstep}, 
we know that $\gamma(\cc^l)$ exists and $\gamupper(\cu) < \gamma(\cc^l)$.
%
%
Consider a graph $G \in \cc^l$. 
Repeatedly ``trim off''
``pendant uncoloured subgraphs'' from $G$ (i.e. for $x \in V(G)$ such that $G-x$ has an uncoloured component $H$, delete $V(H)$ from $G$)
until no such subgraphs remain. We call the remaining graph $G'$
the \emph{coloured core} of $G$. 

Suppose $G$ has a nice vertex $r$. Then since all coloured vertices remain in $G'$,
$G'$ is also nice.
Consider the rooted block tree $T_r$ of $G'$. 
Recall that the nodes of $T_r$ are $\{r\} \cup X \cup \ch$,
where $X$ is the set of cut points of $G'$ and $\ch$ is the set of blocks of $G'$. 

We call a block $B$ of $G'$ \emph{simple} if there are at most two coloured components
in $G' - E(B)$. If there are more than two coloured components in $G' - E(B)$, we call $B$
\emph{complex}. Suppose $y$ is a nice vertex and $y \ne r$. Then, using Proposition~\ref{prop.SPcycle} we see that every block node
on the path $P_{yr}$ from $y$ to $r$ in $T_r$ must correspond to a simple block $B$. 
By Theorem~\ref{thm.Tk}, the edges of $B$
form either a single edge or a parallel $SP$-network, where only the poles can be coloured. Furthermore,
the poles $s$ and $t$ of $B$ must be nice. 

All paths $\{ P_{yr}: \mbox{$y$ is nice} \}$ form an (unrooted) subtree $T'$ of $T_r$,
and the blocks corresponding to them form a connected subgraph of $G'$.
Since a block node $B$ of $T'$ corresponds to a simple block,
it has only two neighbours in $T'$. Therefore
we may consider an unrooted tree $T$ with $V(T) = \{y \in V(G'): y \mbox{ is nice} \}$
and $E(T) = \{xy :  T' \mbox{ contains a path  } x B y \mbox{ for some } B \in \ch\}$.
The trees $T'$ and $T$ do not depend on which nice vertex $G$ is initially rooted at.
We call $T$ the \emph{nice core tree} of $G$. 

Fix 
$v \in V(T)$. 
Consider, the graph $G_v'$ induced on $v$ and the vertices of those components of  $G' - v$ that 
do not contain any nice vertex, and pointed at the vertex $v$ (by retaining the colour of $v$).

First suppose that $v$ is a leaf node of $T$. By Proposition~\ref{prop.rootablejoin},
$G_v'$ admits a unique decomposition to a $\{0,1\}^l$-coloured vertex (i.e. the root) and
some graphs $H_1, \dots, H_t$, where 
$H_i$ is a $C_i$-tree with root $v$, such that $C_i \subseteq [l]$, $C_i \ne \emptyset$.
The requirement that $v$ is nice implies that at least one graph $H_i$ must be an $[l]$-tree
with the restriction that $H_i$ does not have a subgraph $H'$ rooted at a cut vertex $v' \ne v$, 
such that $H'$ is an $[l]$-tree (otherwise $v'$ would also be nice and $v$ would not be a leaf vertex). The generating function counting the class $\bar{\ca}$ of such graphs $H_i$ is 
the same as the one given in Lemma~\ref{lem.multitype}, with the only difference that we 
do not allow an $[l]$-tree to be attached to the root block, so using Lemma~\ref{lem.BAC} and Lemma~\ref{lem.multitype}
\begin{align*}
&\bar{A} = A_{[l]} - 2^l B_1 (e^{A_{[l]}} - 1) \exp \left(\sum_{S \subset [l]} A_S\right)  
\\ & = B_1 
\left( 2^l \exp \left(\sum_{S \subset [l]} A_S\right) + \sum_{S \subset [l]} 2^{|S|} (-1)^{l-|S|} \exp\left(\sum_{S'\subseteq S} A_{S'}\right) \right)
 \\ &+ \sum_{P \in \cp([l]), |P| \ge 2} B_{|P|} \prod_{S \in P} \hat{A}_S.
\end{align*}

If $T$ has at least two vertices, then   $G' - G'_v$ is non-empty and has a component that contains all colours $[l]$. 
Then the exponential generating function for the class $\cL_2$ of all possible graphs $G_v'$ is 
\[
L_2 = 2^l (e^{\bar{A}} - 1)  \exp \left( \sum_{C \subset [l]} {A_C} \right).
\]
If $T$ has only one vertex, then there must be at least two graphs $H_i, H_j \in \bar{\ca}$,
so the exponential generating function for the class $\cL_1$ of all possible graphs $G_v'$ is
\[
L_1 = 2^l (e^{\bar{A}} - \bar{A} - 1)  \exp\left(\sum_{C \subset [l]} A_C \right).
\]
Now suppose $v$ is an internal node of $T$. Then $G' - G'_v$ has at least
two components containing other nice vertices, and each such component contains all  colours $[l]$.
Therefore $G'_v$ is in the class $\ci$ with the exponential generating function
\[
I = 2^l \exp \left( \bar{A} + \sum_{C \subset [l]} A_C \right).
\]
Observe that our expressions for $\bar{A}, L_1, L_2, I$ all are given in terms of
analytic functions of $A_C$, $C \subset [l]$ and $B_i$, $i \le l$.
Thus, by Lemma~\ref{lem.part2main}, Lemma~\ref{lem.Datrho} and Proposition~\ref{prop.rhoBk} 
the convergence radii of each of these functions are at least $\rho(\ca_{l-1}) > \rho(\ca_l)$.

If $G$ is a graph, $v$ is a vertex of $G$ and $\ca$ is a class of vertex-pointed graphs,
$G'$ is obtained from $G$ by \textit{attaching} a graph $H \in \ca$ at $v$ if 
$G' = G \cup H$, $V(G') \cap V(H) = \{v\}$ and we assume that $v$ inherits the label 
of $G$. 

Recall that $\cf$ denotes the class of all rooted series-parallel graphs. Let $\cf_\circ$ denote 
the class of all vertex-pointed series-parallel graphs, so that $F_\circ(x) = F(x)/x$.
For two pointed graphs $G_1$ and $G_2$ with disjoint sets of labels, let $G_1 \times G_2$ be the pointed graph
obtained by identifying their roots. Call classes $\cd_1, \cd_2$ of pointed graphs \emph{uniquely mergeable},
if $G_1 \times G_2 \ne G_1' \times G_2'$ for all $G_1, G_1' \in \cd_1$, $G_2, G_2' \in \cd_2$, where
$G_1 \times G_2$ and $G_1' \times G_2'$ are defined. If $\cd_1$ and $\cd_2$ are uniquely mergeable,
we will identify with the combinatorial class $\cd_1 \times \cd_2$, the class of all graphs
$G_1 \times G_2$, where $G_1 \in \cd_1$, $G_2 \in \cd_2$, and $G_1 \times G_2$ is defined.

Obviously, the classes $\cf_\circ $ and  $\ca$ are uniquely mergeable, when $\ca$ is $\cL_1$, $\cL_2$ or $\ci$:
the vertices of $G \in \cf_\circ \times \ca$ that belong to the graph $G_1 \in \cf_\circ$ are exactly
the root $r$ of $G$ and those vertices that are in the uncoloured components of $G - r$.

Given a class of graphs $\ca$ and a class of rooted graphs $\cc$, we denote by $\ca(\cc)$ the
class obtained from graphs in $\ca$ by replacing each vertex by a graph in $\cb$.
Let $\cp_+ = \cp \cup \ce_2$, be the class of non-series $SP$-networks.
The above observations imply that each graph $G \in \cu$ can be constructed as follows.
\begin{itemize}
    \item Take a tree $T$ of size at least one from the set of all unrooted trees~$\ct$ (i.e. a nice core).
    \item Replace each edge $e$ of $T$ by a network $D_e \in \cp_+(\cf)$ (to fix the orientation, we may assume that edges of $T$ are oriented away from the
          node with the smallest label in~$T$). 
      \item Attach at each leaf node of $T$ a graph in $\cf_\circ \times \cL_1(\cf)$ (respectively in $\cf_\circ \times \cL_2(\cf)$) if $T$ has one node (respectively, at least two nodes). 
      \item Attach at each internal node of $T$ a graph in $\cf_\circ \times \ci(\cf)$.
\end{itemize}
It is easy to see (for example, by fixing a root and comparing this construction with the construction of an $[l]$-tree) that the above decomposition is unique and the construction always yields a graph in $\cc^l \setminus \cu$. 


\begin{lemma}\label{lem.nice} Consider $\cb = \{K_4\}$.
    Let $l \ge 2$ an integer, let $R_n \in_u \cc^l$. Let $Y_n$ denote the number of nice
    vertices in $R_n$.
    There is a positive constant $a_l$, such that
    \[
    \pr( |Y_n - a_l n| > \eps n) = e^{-\Omega(n)}.
    \] 
\end{lemma}
\begin{proof}
    Combining Corollary~\ref{col.part2main} with Lemma~\ref{lem.gamupper} we get that $\rho(\cc ^ l) < \rho(\cc ^ {l-1})$. 
    Write 
    \[
    \tilde{\ct} = \ct(\cp_+(\cf), \cf_\circ \times \ci(\cf), \cf_\circ \times \cL_2(\cf)), 
    \]
    and notice that $\tilde \ct$ is aperiodic, since it
    contains, for example, all Cayley trees, where each node has colour $[l]$.
The construction given in this section above yields the following identity
\begin{equation} \label{eq.cceq}
    \cc^l + \cz \times (\cf_\circ \times \cL_2(\cf)) = \tilde{\ct} + \cz \times (\cf_\circ \times \cL_1(\cf)) + \cu.
\end{equation}
We will prove that the convergence radii of
the exponential generating functions of $\cu$, $\cp_+(\cf)$, $\cf_\circ \times \ci(\cf)$, $\cf_\circ \times \cL_1(\cf)$ and $\cf_\circ \times \cL_2(\cf)$ are all at least $\rho(\cc^{l-1})$. 
This implies that $|\cc^l_n| = |{\tilde \ct_n}| (1 + e^{-\Omega(n)})$ 
and the claim
follows by Theorem~\ref{thm.treeleaf}. 

Consider the class $\bar{\cc}$ of rooted graphs obtained from graphs in $\ca_{[l-1]}$ by replacing each labelled vertex
by a graph in $\cf$ and labelling the root. Then $\bar{C}(x) = x A_{[l-1]}(F(x))$ and $\bar{\cc} \subseteq \cc^{\bullet (l-1)}$ (defined in Section~\ref{subsec.proofgc}). Using Lemma~\ref{lem.gcindstep}, ${\bar \rho} := \rho({\bar \cc}) \ge \rho (\cc^{\bullet (l-1)}) = \rho(\cc^{l-1})$. By \cite{momm05} the functional inverse $\psi_F(x)$ of $F$ is increasing for $x \in (0, x_0)$ and $F(x_0) > \rho(\cd)$,
where $x_0 = F(\rho(\cf)) = 0.1279..$ (denoted $\tau(1)$ in \cite{momm05}). By Proposition~\ref{prop.rhoBk} $\rho(\cd) = \rho(\cb_{l-1}) \ge \rho(\ca_{[l-1]})$, so 
$x_0 \ge \rho(\ca_{[l-1]})$. We see (using e.g., Section VI.9 of \cite{fs09}) that ${\bar \rho} = \psi_F(\rho(\ca_{l-1}))$.

By our construction above, for $i \in \{1,2\}$, $\rho(\cL_i) \ge \rho(\ca_{[l-1]})$, therefore $\rho(\cz \times (\cf_0 \times \cL_i(\cf))) \ge \psi_F(\rho(\ca_{[l-1]})) = {\bar \rho} \ge \rho(\cc^{l-1})$. Since each graph in $\cp_{+,n}$ yields a unique graph in
$\cf_{n+2}$, $\rho(\cp_+(\cf)) \ge \rho(\cf) \ge \rho(\cc^{l-1})$.
Finally, $\rho(\cu) \ge \rho(\cu') \ge  \rho(\cc^{l-1})$ by Lemma~\ref{lem.unrootable}.
\end{proof}

\section{Structure of random graphs in $\ex (k+1) K_4$}
\label{sec.unrootedstructure}
\subsection{Proof of Theorem~\ref{thm.K4} and Theorem~\ref{thm.K4struct}}
\label{subsec.unrootedstructure}
Let $H$ be a fixed connected coloured graph on vertices $\{1,\dots,h\}$.
Following \cite{msw05}, we say that $H$ \emph{appears} in $G$ at $W \subseteq V(G)$ if (a) the increasing bijection
from $\{1,\dots, h\}$ to $W$ gives an isomorphism between $H$ and $G[W]$
and (b) there is exactly one edge in $G$ between $W$ and the rest of $G$,
and it is incident with the smallest element of $W$. We let $f_H(G)$ denote the number of sets $W$ such that $H$ appears at $W$ in $G$. 

Let $\ca$ be a class of (coloured) graphs and let $H$ be a connected graph, rooted at $r \in V(H)$. Let $G \in \ca$, and let $S \subseteq V(G)$.
Suppose $G$ and $S$ have the following property:
if we take any number of pairwise disjoint copies of $H$, all disjoint from
$G$, and add an edge between the root of each copy and a vertex in $S$ then the
resulting graph is still in $\ca$. The set $S$ is called an \emph{$H$-attachable} subset of $G$ (with respect to $\ca$).

The next lemma and its proof is just an adaptation of Theorem 4.1 of \cite{msw05} for graphs where not necessarily all of the vertices form an $H$-attachable set.
\begin{lemma}\label{lem.pendant} Let $\cc$ be a non-empty class of (coloured) graphs, and suppose $\gamma(\cc) = c \in [e^{-1}; \infty)$. Let $H$ be a connected (coloured) graph on the vertex set $\{1,\dots,h\}$ rooted at $1$. 
    Suppose there are  constants $a \in (0,1)$, $N_0 > 0$ and $d > 0$ such that the probability that 
    $R_n \in_u \cc$ has an $H$-attachable subset (with respect to $\cc$) of size at least $a n$  is at least $1 - e^{-dn}$ all $n \ge N_0$.
    Fix $\alpha$, such that $\alpha < d$ and  $\alpha \le a/(9 e^2 c^h (h+2) h!)$. Then there exists $n_0$ such that
    \[
    \pr(f_H(R_n) \le \alpha n) \le e^{-\alpha n} \mbox { for all } n \ge n_0.
    \]
\end{lemma}


\begin{proof} The proof is a 
    simple modification
    of the proof of Theorem 4.1 of \cite{msw05}. We skip some of the
 details and refer the reader for them to \cite{msw05}. 
 Write $\beta = a^{-1} e^2 c^h (h+2) h!$ and let $\eps \in (0, 1/3)$ be such that
$(\alpha \beta)^\alpha = 1 - 3 \eps$. Let $f(n)$ denote the number of graphs
in $\cc_n$. 
Then since $\gamma(\cc) = c$ there is $n_1 \ge N_0$ such that for each $n \ge n_1$ we have $e^{-\alpha n} \ge 2 e^{-d n}$ and
\begin{equation}\label{eq.est1}
2 (1-\eps)^n n! c^n \le f(n) \le (1+\eps)^n n! c^n.
\end{equation}
Assume that for infinitely many $n \ge n_1$
the claim of the lemma does not hold: that is, at least $e^{-\alpha n}$ fraction 
of graphs  in $G \in \cc_n$ ``have few pendant appearances'', i.e., $f_H(G) \le \alpha n$. Let $\tilde{\cc}_n \subseteq \cc_n$
consist of those graphs in $\cc_n$ that have few pendant appearances and 
an $H$-attachable subset with at least $an$ vertices. Then 
\[
|\tilde{\cc}_n| \ge f(n) ( e^{-\alpha n} - e^{-d n}) \ge e^{-\alpha n} (1 - \eps)^n n! c^n.
\]
 Let $\delta \in (0, 1)$ be given by $\delta = \alpha h$. We can construct a graph $G$ on $\lceil n(1+\delta) \rceil$ vertices
by putting a graph $G_0$ isomorphic to a graph in $\tilde{\cc}_n$ on some $n$ of these vertices, and adding $\lfloor \alpha n \rfloor$ disjoint copies of $H$
on the remaining $\lceil \delta n \rceil$ vertices, so that for each added copy $H'$ of $H$
there is an edge between 
the least vertex of $H'$ and some $y \in S_0$, where $S_0$
is the largest $H$-attachable subset of $G_0$. The number of such
constructions $b_{\lceil(1+\delta) n\rceil}$ satisfies
\begin{align*}
    b_{\lceil (1+\delta) n\rceil} &\ge \binom { \lceil (1 + \delta) n \rceil} n  |\tilde{\cc}_n| \binom {\lceil \delta n \rceil} {h, \dots, h} \frac {(an)^{\lfloor \alpha n \rfloor}} {\lfloor \alpha n \rfloor!} 
  \\   &\ge
    \lceil (1 + \delta) n\rceil ! e^{-\alpha n} (1-\eps)^n c^n  \frac {a^{\lfloor \alpha n \rfloor}} {h! (h! \alpha)^{\lfloor \alpha n \rfloor}}.
\end{align*}
Now \cite{msw05} show that each graph in this way is constructed at most $\binom {\lfloor (h+2) \alpha n \rfloor} {\lfloor \alpha n \rfloor} \le e((h+2) e)^{\lfloor \alpha n \rfloor}$ times. This,
after a similar calculation as in \cite{msw05} yields that
\begin{align*}
  &f( \lceil(1+\delta)n\rceil) \ge 
    \frac {b_{\lceil(1+\delta) n\rceil}} {e ((h+2) e)^{\lfloor \alpha n \rfloor}}
\ge c' f ( \lceil (1+\delta) n \rceil) \left(\frac {1-\eps} {(1-3\eps) (1+\eps)^2}\right)^n,
\end{align*}
for some constant $c' > 0$ which does not depend on $n$. Since $\frac {1-\eps} {(1-3\eps) (1+\eps)^2} > 1$, 
our assumption cannot hold for infinitely many $n$, a contradiction.
\end{proof}

\begin{corollary}\label{col.pendantK4} 
    Let $\cb = \{K_4\}$, let $h\ge 1$ and $l \ge 2$ be integers and suppose $H \in \cc^l_h$ is rootable at 1. 
    There is a constant $a = a(l, h)  > 0$ such that
    the random graph $R_n \in_u {\crd l}$ satisfies
    \[
    \pr(f_H(R_n) \ge a n) \ge 1 - e^{-\Omega(n)}.
    \]
\end{corollary}
\begin{proof}
    By Proposition~\ref{prop.rootablejoin} and the decomposition of Section~\ref{subsec.unrootedstr} the set of nice vertices of a graph $G \in \crd l$ is $H$-attachable. 
    The claim follows by Lemma~\ref{lem.nice} and Lemma~\ref{lem.pendant}.
\end{proof}

\bigskip

Fix positive integers 
$r$ and $l$, $r > l$.
Recall the class $\tilde{\ca} = {\tilde \ca}^{<\cb, l, r>}$ defined in the proof of Lemma~\ref{lem.doubleblockers}: $\tilde{\ca}$ is
the class of $\{0,1\}^r$-coloured graphs 
 corresponding to the class of graphs that have
 an
 $(l, 2, \cb)$-double
 blocker of size $r$, and $\tilde{\cc}$ is the class of connected such graphs.
 (In this section $\cb = \{K_4\}$ is fixed.)


Suppose $H$ is an induced subgraph of a coloured graph $G$. Similarly as in  \cite{cmcdvk2011, cmcdvk2012},
we will call $H$ a \emph{spike} of $G$ if all of the following hold:
\begin{itemize}
    \item $H$ is a path $v_1 \dots v_{l+1}$;
    \item there is only one edge between $V(H)$ and $V(G - H)$, and this edge is $u v_1$ where $u \in V(G - H)$;
    \item $\Col_H(v_1) = \dots = \Col_H(v_{l+1}) = \{1, \dots, l, x\}$ where $x \in \{l+1, \dots, r\}$;
    \item $u < v$ for each $v \in V(H)$.
\end{itemize}
It is easy to see that two different spikes must be pairwise disjoint.
\begin{lemma}\label{lem.conspikes} 
    Let $\cb = \{K_4\}$, let  $r$ and $l$ be positive integers, $r > l$ and
    consider the random graph $R_n \in_u \tilde{\cc}$. 
    There is a constant $a' = a'(r, l)$ such that
    \[
    \pr(R_n \mbox { has less than $a' n$ spikes}) \le e^{-\Omega(n)}.
    \]
\end{lemma}
\begin{proof}
    Let $H$ be a $\{0,1\}^{l+1}$-coloured path on the vertex set $[l+2]$
    such that one of its endpoints is $1$ and for $v \in \{2, \dots, l+2\}$, we have $\Col_H(v) = [l+1]$.
    By Corollary~\ref{col.pendantK4}, there are positive constants $a, c$ and  $C$, such 
    that the number of graphs in $\cc^{l+1}_n$ with at most $a n$ pendant
    appearances $H$ is at most
    $C e^{-c n}|\cc^{l+1}_n|$ for every $n$.




    Let $N = (1+\aw 2 (\ex K_4))^{l-1} = 3^{l-1}$.
    In the proof of Lemma~\ref{lem.doubleblockers} we have shown that each graph in $\tilde{\cc}_n$, as well as some other
    graphs, can be obtained as follows.
    \begin{itemize}
        \item Pick $\kappa \in [N]$ and  $j, m \in \{0, \dots, N-1\}$;
        \item choose a partition $\cs$ of $[n+j]$ into $\kappa$ sets $V_1 ,\dots, V_{\kappa}$;
        \item for each $i = 1, \dots, \kappa$ put an arbitrary graph $H_i \in \cc^{l+1}$ on $V_i$;
        \item for each $i = 1, \dots, \kappa$ choose $q_i \in  \{l+1, \dots, r\}^\kappa$ and map the colour $l+1$ in $H_i$ to $q_i$;
        \item choose a set $J$ of $m$ edges between the components $H_1, \dots, H_\kappa$ and add them
            to the resulting graph;
        \item finally, contract all edges $J$, so that the vertex resulting from a contraction of an edge $e = xy$ receives $max(x,y)$ as a label.
    \end{itemize}
    Consider the set $\mathcal{M}(n)$ of all possible constructions that yield a graph on the vertex set $[n]$ (the
    (multi-)set of the resulting graphs contains~$\tilde{\cc}_n$). 
    
    By Lemma~\ref{lem.gcrd} the class $\cc^{l+1}$ has a growth constant $\gamma$.
    By Lemma~\ref{lem.part2main} and the proof of Lemma~\ref{lem.doubleblockers} 
    \[
    |\mathcal{M}(n)| = n! \gamma^{n (1 + o(1))} \quad \mbox{and}\quad |\tilde{\cc}_n| = n! \gamma^{n (1 +o(1))}.
    \]

    Fix $\kappa = \kappa_0$, $m = m_0$
    $j = j_0$, $\cs = \cs_0$, $q = q_0$ and $J = J_0$ such 
    that there is at least one construction in $\mathcal{M}(n)$ with these parameters. Then every choice 
    of the graphs in $\{H_i\}$ yields a construction in $\mathcal{M}(n)$. In particular, writing $n_i = |V_i|$, there are in total
    $t_0 = \prod_{i=1}^{\kappa_0} |\cc^{l+1}_{n_i}|$ constructions in $\mathcal{M}(n)$ with these parameters.

    Note that the largest set $V_j$ in $\cs_0$ always contains $n' \ge n / \kappa_0 \ge n / N$ elements.
    It is easy to see that each pendant appearance of $H$ in $H_j$ yields a spike in $H_j$. 
    There are at most $C e^{-cn'} |\cc^{l+1}_{n_i}|$ ways
    to choose the graph $H_j$ so, that $H_j$ has less than $ a n'$ spikes.
    If $H_j$ has more than $a n'$ spikes, then the graph $G$ resulting from the construction
    has at least $a n' - 2 |J_0| \ge an' - 2N$ spikes, since the spikes are disjoint and each edge in $J$ can touch at most two spikes.

    Therefore there are at most
    \[
    C e^{-cn'} t_0 \le C e^{-(c/ N) n} t_0
    \]
    ways to finish the construction by choosing $H_1, \dots, H_{\kappa_0}$, 
    so that the resulting graph $G$ has less than $(a / N) n - 2N$ spikes.

    Since this bound holds for every $\kappa_0, m_0, j_0, \cs_0, q_0$ and $J_0$, we get by the law of
    total probability that the number of constructions in $\mathcal{M}(n)$ that yield a graph
    with at most $(a/N)n - N$ spikes is at most
    \[
    C e^{-(c/N) n} |\mathcal{M}(n)| \le  n! e^{-(c/N) n}  \gamma^{n + o(n)}.
    \]
    So for any $a' < a / N$ and $n$ large enough
    \[
    \pr(R_n \mbox { has less than $a' n$ spikes}) \le C e^{-(c/N) n + o(n)} = e^{-\Omega(n)}.
    \]
\end{proof}

\begin{lemma}\label{lem.disconspikes} Let $\cb = \{K_4\}$, let $l, r$ and $K$ be positive integers, $r > l$. Then for $R_n \in_u \tilde{\ca}$ we have
    \[
    \pr(R_n \mbox { has at most $K$ spikes}) \le e^{-\Omega(n)}.
    \]
\end{lemma}

\begin{proof} Let $\tilde{\ca}_1$ be the class of graphs in $\ca$
    that have at most $K$ spikes, and let $\tilde{\cc}_1$ be the class
    of graphs in $\tilde{\cc}$ that have at most $K$ spikes.
    Then $\tilde{\ca}_1 \subseteq \SET(\tilde{\cc}_1)$ and
    \[
    \tilde{A}_1(x) \le e^{\tilde{C}_1(x)}.
    \]
    Using Lemma~\ref{lem.conspikes}, 
    \[
     \gamupper(\tilde{\ca}_1) \le \gamupper(\tilde{\cc}_1) < \gamma(\tilde{\cc}) = \gamma(\tilde{\ca}).
    \]
    The lemma follows by (\ref{eq.small}).
\end{proof}

\bigskip

    For an $r$-coloured graph $G$, $S_1, S_2 \subseteq [r]$, 
    and sets $Q_1, Q_2$ disjoint from $V(H)$, such that $|S_i| = |Q_i|$ for $i=1,2$,
    we denote by $G^{S_1 \to Q_1, S_2 \to Q_2}$ the graph obtained by adding to $G$ new vertices $Q_1 \cup Q_2$,
    and for $i=1,2$ adding an edge between $q_i^{(j)}$ and each vertex coloured $s_i^{(j)}$. Here $q^{(j)}_i$ and
    $s^{(j)}_i$ is the $j$-th smallest element in $S_i$ and $Q_i$ respectively.

\begin{lemma}\label{lem.rdmain} Let $k$ be a positive integer. 
    Then 
    \begin{equation}\label{eq.rdmain}
    |(\ex (k+1) K_4)_n| = (1 + e^{-\Omega(n)}) |(\rd {2k+1} K_4)_n|.
    \end{equation}
\end{lemma}
\begin{proof}
    By Lemma~\ref{lem.gamupper}, Lemma~\ref{lem.part2main}, Lemma~\ref{lem.notcomplex} and Theorem~\ref{thm.gc}, there is a constant $r = r(k) > 2k$ such that 
    all but an exponentially small fraction of graphs from $\ex (k+1) K_4$,
    have a $(2k, 2, K_4)$-double blocker of size $r$. Here we will show 
    that all but an exponentially small fraction of graphs in the latter class
    have a redundant blocker of size $2k+1$. Since each graph in $\rd {2k+1} K_4$
    is in the class $\ex (k+1) K_4$, the claim will follow.
%
    We will use the idea of the proof of the main result of \cite{cmcdvk2012}.

    Fix $n \ge r$. All graphs in $(\ex (k+1) K_4)_n$ that have a $(2k, 2, K_4)$-double blocker (and some other graphs) can be constructed as follows.
    \begin{itemize}
        \item Choose $Q \subseteq [n]$ of size $r$ and $S \subset Q$ of size $2k$.
        \item Put an arbitrary graph in $\tilde{G} \in \tilde{\ca} = \tilde{\ca}^{<\{K_4\}, 2k, r>}$ on $[n] \setminus Q$.
        \item Put an arbitrary graph $H$ on $Q$.
        \item Let $G = \tilde{G}^{\{1,\dots,2k\} \to S, \{2k+1, \dots, r\} \to Q \setminus S}$. This finishes the construction of a
            graph $G$. 
    \end{itemize}
    Suppose $\tilde{G}$ has more than $ K = r (k + 2r + 12)$ spikes. Then there is a (smallest) colour $q \in \{2k+1, \dots, r\}$
    such that there are at least $k + 2r + 12$ spikes coloured $\{1,\dots,2k, q\}$. Let $x$ be the vertex in $Q\setminus S$
    whose neighbours in 
    $G$
    are the vertices coloured $q$
    in $\tilde{G}$. 
    Denote $S' = S \cup \{x\}$.

    Suppose $G$ contains a $\cb$-critical subgraph $H'$ (that is, a subdivision of $K_4$) which has at most one vertex in $S'$. 
    By Lemma~5.3 of \cite{cmcdvk2012},
    $H'$ can touch at most $2 (r + |E(K_4)|) = 2r + 12$ spikes in $\tilde{G}$. Thus there are
    at least $k$ spikes that are disjoint from $H$. Form an arbitrary maximal matching in 
    the set $S' \setminus V(H)$: the matching has exactly $k$ pairs. For $y,z \in S'$ 
    and a spike $P$ in $\tilde G$, we have that $G[V(P) \cup \{y,z\}] \not \in \ex K_4$. Thus, we can produce
    $k$ disjoint minors in $\cb$ for each pair in the matching. The graph $H$ yields $(k+1)$-st disjoint
    minor in $\cb$. 

    Thus, whenever $\tilde{G}$ has at least $K$ spikes and $G \in \ex (k+1) \cb$, we have that
    $S'$ is a redundant blocker for $G$.
    So each construction $G$
    such that $G \in \ex (k+1) K_4 \setminus \rd {2k+1} K_4$ is formed by taking
    a graph $\tilde{G}$ with at most $K$ spikes in the second step.
    
    By Lemma~\ref{lem.disconspikes},
    the number of choices for $\tilde{G}$, such that $\tilde{G}$ has less than $K$ spikes is
    at most
    \[
        e^{-\Omega(n)} |\ca_{n-r}|. 
    \]
    Therefore if $\cd = \ex (k+1) K_4 \setminus \rd {2k+1} K_4$, we have for $n \ge r$, $n \to \infty$
    \[
        |\cd_n| \le \binom n r \binom r {2k} 2^{\binom {2k} 2}  e^{-\Omega(n)} |{\tilde \ca}_{n-r}|
    \]
    and $\gamupper(\cd) < \gamma({\tilde \ca}) = \gamma(\ex (k+1) K_4)$.
\end{proof}

\begin{lemma} \label{lem.rdcount}
    Let $\cb = \{K_4\}$, and let $l \ge 2$ be an integer. 
    We have
    \[
        |(\rd l \cb)_{n+l}| = a_n  (1 - e^{-\Omega(n)}) \quad \mbox{where} \quad  
    a_n = 2^{\binom l 2}  \binom {n+l} {l} |\crd {l, n}|. 
    \]
\end{lemma}
\begin{proof}
     Consider the following constructions
     of graphs on $[n+l]$: first pick a set $Q \subseteq[n+l]$ of size $l$, next take a graph $G_0 \in \crd l$ 
     with $V(G_0) = [n+l] \setminus Q$ and an arbitrary graph $H$
     with $V(H) = Q$. 
     Let 
     $G = G_0^Q \cup H$.
     Each graph in $(\rd l \cb)_{n+l}$ can be obtained in this way, so $|(\rd {l} \cb)_{n+l}| \le a_n$. 
     We aim to bound the number of constructions
     that can be obtained twice, i.e., the ones which have two or more different redundant $K_4$-blockers
     of size $l$.

     If $G_0$ has at least 
     $l+1$
     spikes then $Q$ is a unique redundant blocker of size $l$.
     Indeed, if $Q'$ is another such blocker, $Q' \ne Q$, then take a vertex $z \in Q \setminus Q'$
     and any $x \in Q \setminus \{z\}$.
     
     Now $x, z$ and
     the vertices of any spike $S$ induce a minor of $K_4$. Therefore, since $Q' \setminus \{x\}$ 
     must still be a $\cb$-blocker for $G$, $Q'$ must contain a vertex from each of the $l+1$ spikes, and so $|Q'| > l$,
     a contradiction. 

     Thus every construction where $Q$ is not the unique redundant blocker is obtained when $G_0$ has at most $l$ spikes.
     By Corollary~\ref{col.pendantK4}, there are at most
     \[
     2^{\binom l 2}\binom {n+l} {l} e^{-\Omega(n)} |\crd {l, n}| = a_n e^{-\Omega(n)} 
     \]
     such constructions, so the number of graphs in $(\rd l \cb)_n$ that have a unique redundant blocker is at least
    $a_n (1 - e^{-\Omega(n)})$.
\end{proof}

%

\begin{lemma}\label{lem.sizeAl} Let $l \ge 2$ be an integer and let $\crd l$ be the class of coloured
    graphs defined for $\cb = \{K_4\}$ in Section~\ref{subsec.b2k}. Let $\rho_l = \rho(\crd l)$. Then $A_l(\rho_l) < \infty$ and there is a constant $a_l > 0$ such that 
    \[
    |\crd {l,n}| = a_l n^{-5/2} n! \rho_l^{-n} \left( 1 + o(1) \right).  
    \]
\end{lemma}

\begin{proof}
    By the exponential formula we have $\rho(\cc^l) = \rho_l$.
    We will show that
    
    \bigskip 
    \begin{doublespace}
    \begin{tabular}{c l}
       (*)    & $C^l$ converges to some positive constant at $\rho_l$. \\ 
       (**)   & $|\cc^l_n| > 0$ for all $n \ge 1$. \\ 
       (***)  & $\cc^l$ is smooth, i.e. ${|\cc^l_{n+1}|} / {n |\cc^l_n|} \to \rho_l^{-1}$. \\ 
       (****) & For any $w = w(n) \to \infty$ we have
 \end{tabular}
 \end{doublespace}
 \vskip -0.5 cm
 \begin{align*} 
     &S(n, w) = \sum_{k = w}^{n-w} \binom n k |\cc^l_k| |\cc^l_{n-k}| = o(|\cc^l_n|).
 \end{align*}
    Then $A_l(\rho_l) \le  e^{C^l(\rho_l)} < \infty$ and
    Theorem~2 of Bell, Bender, Cameron and Richmond \cite{bbcr2000}  yields
    \[
      |\crd {l,n}| = \frac 1 {A_l(\rho_l)} |\cc^l_n| \left( 1+ o(1) \right).
    \]

    We will use the identity (\ref{eq.cceq}) and the notation from the proof of Lemma~\ref{lem.nice}. 
    There we have shown 
    that the 
    class $\tilde{\ct}$ 
    has $\rho(\tilde{\ct}) = \rho_l$ and by Lemma~\ref{thm.treeleaf}
    $\tilde{T}(\rho_l) < \infty$. Since $\rho(\cf_\circ)$, $\rho(\cL_1(\cf))$, $\rho(\cL_2(\cf))$, and $\rho(\cu)$
    are all strictly larger than $\rho_l$,
    we have by (\ref{eq.cceq}) that $C^l(\rho_l) < \infty$. Since 
    the coefficients of $C^l$ are non-negative, (*) follows. The condition (**) is obvious ($\cc^l$ includes, i.e., every uncoloured path on $n$ vertices).
    Furthermore,
    by (\ref{eq.cceq}) and Theorem~\ref{thm.treeleaf}
    \begin{equation} \label{eq.sizeCl}
    |\cc^l_n| = c n^{-5/2} n! \rho_l^{-n} (1 + o(1)) 
    \end{equation}
    for some constant $c > 0$, so (***) follows. 

    Finally, let us prove (****). Let $w = w(n) \to \infty$. We may assume $2 \le w(n) \le n/2$ for all $n$.
    By (\ref{eq.sizeCl}) 
    for any $\eps > 0$, for all sufficiently large $j$ we have $|\cc^l_j| \le  (c+\eps) j^{-5/2} j! \rho_l^{-j}$. So
    for $n$ sufficiently large
    \[
    S(n) = \sum_{k=w}^{n-w} \binom n k |\cc^l_k| |\cc^l_{n-k}| \le (c+\eps)^2 \rho_l^{-n} n! \sum_{k=w}^{n-w} k^{-5/2} (n-k)^{-5/2}.
    \]
    Now symmetry and a standard approximation of a sum by an integral gives for $w' = w-1$
    \[
    f(n) = \sum_{k=w}^{n-w} k^{-5/2} (n-k)^{-5/2} \le 2 n^{-4} \int_{x = \frac 1 2}^{1 - \frac {w'} n} x^{-5/2} (1-x)^{-5/2} dx.
    \]
    Since for $t \in (1/2, 1)$
    \[
        \int_{1/2}^t x^{-5/2} (1-x)^{-5/2} dx = - \frac { 2 (1 + 6t -24 t^2 + 16 t^3)} {3 t^{3/2} (1-t)^{3/2}}
    \]
    we have
    \[
    f(n) \le \frac {4 (n^3 + 6w'n^2 - 24 w'^2 n + 16 w'^3)} {3 n^4 (n-w')^{3/2} w'^{3/2}} = O\left(n^{-5/2} w^{-3/2}\right). 
    \]
    Thus $S(n) = O(|\cc^l_n| w^{-3/2}) = o(|\cc^l_n|)$. This completes the proof.
\end{proof}

\bigskip

We are now ready to prove Theorem~\ref{thm.K4}.

\bigskip

\begin{proofof}{Theorem~\ref{thm.K4}}
    To replace $\Omega$ by $\Theta$ in the result of Lemma~\ref{lem.rdmain},
    note that by Lemma~\ref{lem.apex_not_rd} and Theorem~1.2 of \cite{cmcdvk2012}, $(\ex (k+1) K_4)_n$ contains
    at least
    \[
    |(\apexp k K_4)_n \setminus (\rd {2k+1} K_4)_n| 
    = n! (2^k\gamma(\ex K_4))^{n + o(n)}
    \]
    graphs that do not have a redundant $K_4$-blocker of size $2k+1$.
    
    The theorem follows by Lemma~\ref{lem.part2main}, Lemma~\ref{lem.rdmain}, Lemma~\ref{lem.rdcount} and Lemma~\ref{lem.sizeAl}. 
\end{proofof}

\bigskip

\begin{proofof}{Theorem~\ref{thm.K4struct}}
    Let $R_n'$ be a random construction as in the proof of Lemma~\ref{lem.rdcount},
    where we pick the set $Q$ of size $2k+1$, the graph $G_0 \in \crd {l,n}$
    and the graph $H$ on $Q$ uniformly at random. 
    Then Theorem~\ref{thm.K4} and the proof of Lemma~\ref{lem.rdcount} imply that 
    the total variation distance between $R_n$ and $R_n'$ satisfies
    \[
    d_{TV}(R_n, R_n') = e^{-\Theta(n)}.
    \]
    Therefore it is enough to prove the theorem for the random graph $R_n'$.
    By Corollary~\ref{col.pendantK4}, there is a constant $a_k > 0$, such 
    that the graph $G_0$ has at least $a_k n$ spikes (and so, each vertex in $Q$ has degree at least $a_k n$)
    with probability $1 - e^{-\Omega(n)}.$

    Suppose there is a blocker $Q'$ of $R_n'$ and at least two distinct vertices $x,y \in Q' \setminus Q$.
    Then any spike and $\{x,y\}$ induces a minor $K_4$, thus every such blocker must have
    at least $a_k n$ vertices with probability at least $1-e^{-\Omega(n)}$.
    Similarly $Q$ is with probability $1-e^{-\Omega(n)}$ a unique redundant $K_4$-blocker for $R_n$ of size $2k+1$.
    This finishes the proof of \textit{(a)}.

    Finally, \textit{(b)} follows by a result of \cite{cmcd09} (restated in 
    a slightly more convenient form in \cite{cmcdvk2012}).
    More precisely, by Lemma~6.2 of \cite{cmcdvk2012}, the graph $Frag(R_n)$
    obtained from $R_n$ by removing its (lexicographically) largest
    component is in the class $\ex K_4$ with probability $1 - e^{-\Omega(n)}$.
    The class $\ex (k+1) K_4$ is  bridge-addable and by Theorem~\ref{thm.K4} it is smooth and $\rho = \gamma(\ex (k+1) K_4)^{-1} > 0$.
    Therefore by \cite{cmcd09}, see Lemma~6.3 of \cite{cmcdvk2012},
    $Frag(R_n)$ converges in total variation to the ``Boltzmann-Poisson'' random graph 
    with parameters $\cb$ and $\rho$, in particular $\pr(|V(Frag(R_n))|=0) \to p_k = A(\rho)^{-1}$.
\end{proofof}

\subsection{An illustration}
\label{subsec.illustration}

We present Figure~\ref{fig.Ctree2} illustrating some properties of typical graphs in $\ex 2 K_4$.

\begin{figure}[h]
     \begin{center}
    \includegraphics[height=7cm]{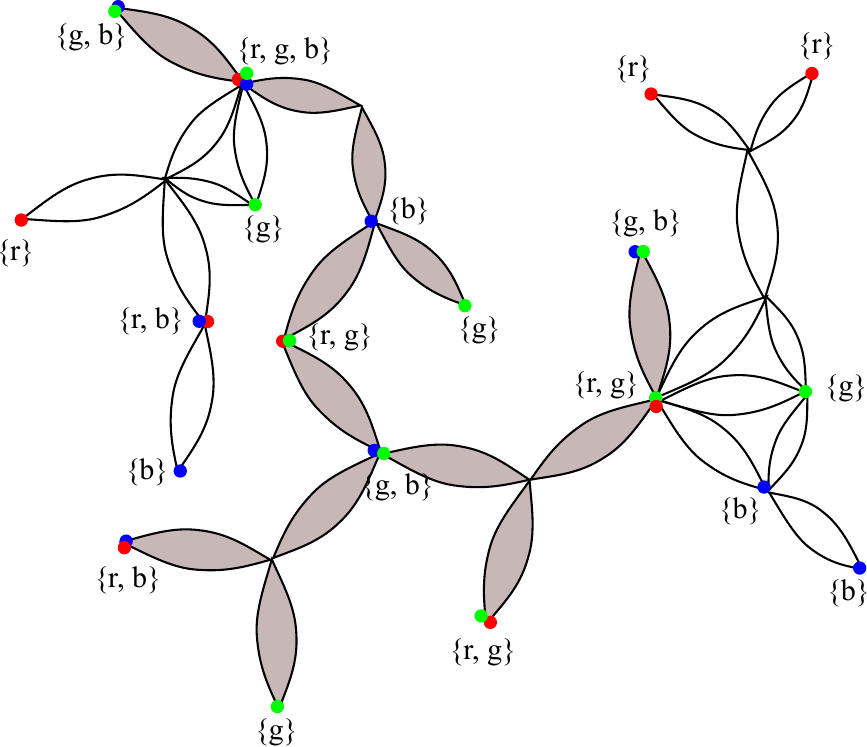}
    \end{center}
    \caption{Structure of (coloured cores of) graphs in $\ex 2 K_4$.}
    \label{fig.Ctree2}
\end{figure}

The picture shows just
          the coloured core of a graph in $\ex 2 K_4$ (it is typically of linear size). The leaf-like shapes are arbitrary series-parallel networks, most of them small.  Attach arbitrary rooted series-parallel graphs at arbitrary vertices of this;
          add a redundant blocker - $\{x,y,z\}$ - and connect $x$ to each red point, $y$ to each green point and $z$ to each blue point (there can also be arbitrary edges on $\{x,y,z\}$). 
          Unlike in planar graphs, there are vertices of different types: colours can occur only at the ``joints''; only the joints of the grey blocks can have arbitrary colour. For example, the top-right vertex cannot have colour blue, otherwise $z$ alone would form a minor $K_4$, and $\{x,y,z\}$ would not be a redundant $K_4$-blocker.

\section{The class $\ex (k+1) \{K_{2,3}, K_4\}$}
\label{sec.outerplanar}

The class $\ex (k+1) \{K_{2,3}, K_4\}$ is a subclass of $\ex (k+1) K_4$. 
Even though it does not satisfy the condition of Theorem~\ref{thm.gc},
we can still adapt most of the techniques from the preceding sections 
(the proofs are simpler for this case).

\subsection{Coloured cores and paths}
\label{subsec.outerplanarcores}


In this section we fix $\cb = \{K_{2,3}, K_4\}$.
To avoid an additional index, we will accordingly write
$\cc^{l} = \cc^{l, \cb}$, $\crd l = \crd {l,\cb}$,
and 
assume that the definitions such as ``good colour'',
``nice vertex'' and ``nice graph''
are with respect to $\cb = \{K_{2,3}, K_4\}$.

\begin{lemma}\label{lem.outerplanarpath}
    Let $\cb = \{K_{2,3}, K_4\}$, let $l$ be a positive integer
    and let $G \in \cc^l$ be nice. 
    The nice core tree of $G$ is a path.
\end{lemma}

\begin{proof}
	In Section~\ref{subsec.unrootedstr} we showed that the nice core tree $T$ of $G \in \cc^{l, \cb} \subset \cc^{l, \{K_4\}}$
    is a tree. If $T$ has at least three leaves, we can produce a minor $K_{2,3}$ by adding
    a new vertex connected to each of these leaves, thus every colour is bad for $G$.
\end{proof}

\bigskip

Let $\tilde{\cd}$ be the 
class of all biconnected outerplanar
networks $G$ 
such that adding a new vertex connected
to each pole of $G$ gives an outerplanar graph. Let $G \in \tilde{\cd}$
be 2-connected.
Then $G$ has a unique Hamilton cycle $H$  (see, e.g., \cite{bps08}).
The poles $s$ and $t$ of $G$ must be neighbours in $H$, otherwise adding a new vertex connected to each pole yields
a minor $K_{2,3}$. Conversely, if we take an arbitrary 2-connected outerplanar graph $G$ and pick an oriented edge $s t$ from
its Hamilton cycle $H$, the graph with source $s$ and sink $t$ obtained from $G$ is from the class $\tilde{\cd}$: this follows by the relation of 2-connected outerplanar graphs and polygon dissections \cite{bps08}. Therefore from each
2-connected rooted outerplanar graph we can obtain two networks in $\tilde{\cd}$ (the
root becomes the source and either
the left or the right neighbour of the root on the Hamilton cycle becomes the sink) with $n-2$ vertices. It follows that
\[
\tilde{D}(x) = \frac {2 B(x)} {x^2} - 1,
\]
where $B(x)$ is the exponential generating function of rooted 
biconnected
 outerplanar graphs (which contains $K_2$). Bernasconi, Panagiotou and
Steger \cite{bps08} show that
\[
B(x) = \frac 1 2 (D(x) + x^2)
\]
where $D(x)$ is the exponential generating function for polygon dissections.
Thus $\tilde{D}(x) = D(x)/x^2$ and we get by (4.1) of \cite{bps08}
\begin{equation}\label{eq.outerplanarnetworks}
    \tilde{D}(x) = \frac 1 {4x} \left( 1 + x - \sqrt{x^2 - 6x +1} \right).
\end{equation}
Solving quadratic equations and using the ``first principle'' from \cite{fs09} we get that 
\begin{equation} \label{eq.rhoDrho}
	\rho(\tilde{\cd}) = 3 - 2 \sqrt 2 \quad \mbox{ and } \quad
\rho(\tilde{D}) \tilde{D}(\rho(\tilde{\cd})) = 1 - \frac {\sqrt 2} 2 = 0.292\dots
\end{equation}

\begin{lemma}\label{lem.structouterplanar} Let $\cb = \{K_{2,3}, K_4\}$, let $l \ge 2$ be an integer, and let $\tilde{\cc}^l$ be the class of coloured cores of nice graphs in $\cc^l$.
    Then
	the	class $\tilde{\cc}^l$ has exponential generating function
            \begin{equation}\label{eq.outerplanarc}
                \tilde{C}^l(x) =  
                                \frac {2^{l-1} x L(x)^2} {1 - 2^l x \tilde{D}(x)},
            \end{equation}
            where 
	    $L$ is the exponential generating function of a class $\cL = \cL^{<l>}$
	    with $\rho(\cL) \ge \rho( {\tilde \cc}^{l-1})$.
\end{lemma}

\begin{figure}
     \begin{center}
    \includegraphics[height=2.5cm]{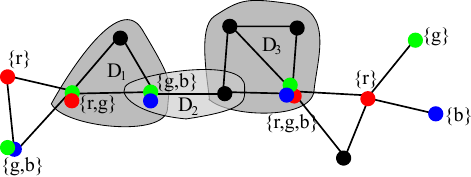}
    \end{center}
    \caption{Coloured core of a graph in $\cc^3$ in the case $\cb =\{K_{2,3}, K_4\}$.
    Replacing the three
    grey networks (which belong to the class $\tilde{\cd}$) with edges we obtain its nice core tree (a path of length three). The
             graphs attached to the endpoints of this path belong to the class $\cL$.}
    \label{fig.outerplanarcore}
\end{figure}

\begin{proof}
    Let $\cL = \cL^{<l>}$ be the class of all pointed $\{0,1\}^l$-coloured connected graphs $G$ 
    satisfying the following conditions: 
   a) each colour is good for $G$;
   b) $G-r$ is connected, $\Col(G-r) = [l]$ and $\Col(r) = \emptyset$, where $r=r(G)$ is the root of $G$;
   c) for any $x \in V(G)$, each component of $G-x$
      contains at least one colour or $r$, and 
      at most one component contains
      all colours $[l]$
      or $r$.
 %
%
%
    The class $\cL$ corresponds to the class $\cL_2$ from Section~\ref{subsec.unrootedstr}.
   
    We will now prove that $\rho(\cL) \ge \rho({\tilde \cc}^{l-1})$. 
    We first claim that for each positive integer $i$ and each $n = 0, 1, \dots$, 
    \begin{equation} \label{eq.ccup}
        |{\tilde \cc}^i_n| \ge |{\tilde \cc}^{i-1}_n|.
    \end{equation}
    Indeed, for any graph in $G \in {\tilde \cc}^{i}_n$ we can assign a graph $G' \in {\tilde \cc}^{i+1}_n$ by picking
    the lexicographically minimal pair of vertices coloured red and adding colour $i+1$ to their colour sets. Such a pair always exists since $G$ is nice, and since the colour red is good for $G$, the colour $i+1$ must be good for $G'$.

 
Let $\cc$ be the class of all pointed graphs $H$, such that 
\begin{enumerate}
    \item $\Col(H) \subset [l]$;
    \item if we add to $H$ a new vertex $w$ coloured $\Col(H)$ and connected to the root of $H$
        and label the root of $H$ arbitrarily, we obtain a graph in $\tilde{\cc}^l$.
\end{enumerate}
The coefficients of $C(x)$ are dominated by those of
$\sum_{j \in [l-1]} \binom l j x^2 {\tilde C}^j(x)'$, therefore $\rho(\cc) \ge \min_{j \in [l-1]} \rho({\tilde \cc}^j) \ge
\rho({\tilde \cc}^{l-1})$. The last inequality follows using (\ref{eq.ccup}).
Recall also that we denote by $\cz_\emptyset$ the class of graphs consisting of a single pointed uncoloured vertex only.

Now let $G \in \cL$.
$G$ can be constructed as follows:
    take a pointed 
    biconnected
    outerplanar graph $B$ ($B$ is the block of $G$ containing the root with its colours removed),
    colour $t \le 2l$ of its (non-root) vertices $v_1, \dots, v_t$ with some subsets of $[l]$
    and identify $v_i$ with the root of a pointed graph 
    $G_i \in \cc \cup \cz_\emptyset$ (so that labels of $G$, $G_1, \dots, G_t$ are disjoint). To verify the last statement, note first that
    the number of cut vertices and coloured vertices in $B$ is bounded by $2l$, since each
    component of $G - E(B)$ can contain at most two components containing colour
    $c$ for each $c \in [l]$ (see Section~\ref{subsec.cbt}).
    Secondly, suppose $u$ is a cut vertex of $G$, and let $G_u$ the graph obtained from the component of $G - E(B)$ containing $u$
    by pointing the vertex $u$. The graph $H$ obtained from $G$ by contracting all vertices
    in $G - G_u$ into a single vertex $w$, setting $\Col_H(w) = \Col_G(G_u - u) \subset [l]$ and $\Col_H(u) = \emptyset$ is a coloured minor of $G$: this shows that indeed $G_u \in \cc$.

    Let $\cb_\circ$ denote the class of pointed biconnected outerplanar graphs. By the properties of biconnected outerplanar graphs
    discussed above in this section we see that
    $\rho(\cb_\circ) = \rho({\tilde \cd}) \ge \rho(\tilde{\cc}^1)$. (The last inequality follows since from each graph in $\tilde \cd$
    we may obtain a graph in ${\tilde \cc}^1$ by labelling its poles and colouring them $\{red\}$.)

    The above observations imply that the coefficients of $L(x)$ are bounded by the coefficients of
    \[
        \sum_{t=0}^{2l} x^t B_\circ(x)^{(t+1)} (2^l (1 + C(x)))^t.
    \]
    Therefore $\rho(\cL) \ge \min(\rho(\cc), \rho(\cb_\circ)) \ge \rho({\tilde \cc}^{l-1})$.

%
   
    Now let ${\tilde \cc}^{l,1}$ denote the subclass of graphs in ${\tilde \cc}^l$ with the nice tree of size 1, and let
    ${\tilde \cc}^{l,2} = {\tilde \cc}^l \setminus {\tilde \cc}^{l,1}$.
    Using Lemma~\ref{lem.outerplanarpath} and Section~\ref{subsec.unrootedstr}, each graph in ${\tilde \cc}^{l,2}$ can be obtained
    as a series composition of $r \ge 1$ outerplanar networks $D_1, \dots, D_r$, where only the poles of these networks
    can be coloured (with arbitrary colours in $2^{[l]})$, by attaching two rooted graphs $L', L''$ at the source of $D_1$ and the sink of $D_r$ respectively (attaching pendant
coloured graphs to vertices corresponding to internal vertices of the nice core path is forbidden,
otherwise some colour would be bad for the resulting graph), see Figure~\ref{fig.outerplanarcore}. Let $c \in [l]$. Since each pole of a network $D_i$ is nice, 
there are two disjoint
    paths from each of the poles ending with a vertex coloured $c$. This shows that $D_i \in \tilde{\cd}$.
    
    It is easy to see that $L',L''$ must always be from the class $\cL$. On the other hand, every such construction gives a valid graph in ${\tilde \cc}^{l,2}$, and every graph in $\tilde{\cc}^{l,2}$ is obtained exactly twice (reversing each network and their order in the sequence and swapping $L'$ with $L''$ gives the same graph).
    
    Therefore
    \[
        2 \times \tilde{\cc}^{l,2} =  (2^l \times \cz \times \cL)^2 \times \tilde{\cd} \times \SEQ ( (2^l \times \cz) \times \tilde{\cd}).
     \]
     Similarly 
     \[
         \tilde{\cc}^{l,1} = 2^l \times \cz \times \SET_2(\cL).
     \]
     Finally, the identity $\tilde{\cc}^{l} = \tilde{\cc}^{l,1} + \tilde{\cc}^{l,2}$ and a standard conversion to exponential generating functions \cite{fs09}  yields (\ref{eq.outerplanarc}).
%
\end{proof}

\bigskip



\begin{lemma}\label{lem.rho.tildeC}
    Let $l \ge 2$ be an integer and let $\cL^{<L>}, {\tilde \cc}^l$ be as in Lemma~\ref{lem.structouterplanar}. We have $\rho(\cL^{<l>}) > \rho ( {\tilde \cc}^ l) = r_l$, where
    \[
    r_l = \frac 1 {2^l} \left( 1 - \frac 1 {2^l-1} \right). 
    \]
    There are constants $R_1 = R_1(l) > r_l$ and $c_l' > 0$ and a function $g_1(x) = g_{1,l}(x)$ analytic for $x \in \mathbb{C}$ with $|x| < R_1$
    such that
    \begin{equation}\label{eq.exptildec}
        {\tilde C}^l(x) = \frac {c_l'} {1 - x/r_l} + g_1(x).
    \end{equation}
\end{lemma}

\begin{proof}
    A simple calculation shows that the unique 
    solution 
    of $2^l x \tilde{D}(x) = 1$ is $r_l$. Notice also that
    since $\tilde{\cd}$ contains a network isomorphic to $K_2$ and a network
    isomorphic to an arbitrary cycle, we have
    $[x^n] 2^l x \tilde{D}(x) > 0$ for $n = 1, 2, \dots$ (that is, $2^l x \tilde{D}(x)$ is strongly aperiodic, see \cite{fs09}).

    We will also use that $r_2 = \frac 1 6$ and $r_{j+1} < r_j$ for any integer $j = 2, 3, \dots$. Define
    \[
    c_l' = \frac {r_l L(r_l)^2} {2 (\tilde{D}(r_l) + r_l \tilde{D}'(r_l)) }.
    \]
    We prove the claim by induction on $l$. First consider the case $l = 2$. By Lemma~\ref{lem.structouterplanar}
    $\rho(\cL^{<2>}) \ge \rho( {\tilde \cc}^1) = \rho(\tilde \cd) = 3 - 2 \sqrt 2 > \frac 1 6$.
    We can write ${\tilde C}^2(x) = h(x) f(x)$ where $f(x) = (1 - 2^2 x {\tilde D}(x))^{-1}$ and
    $h(x) = 2^2 x L^{<l>}(x)^2$.
    By (\ref{eq.rhoDrho}), $f(x)$  corresponds
    to a supercritical sequence schema and $h(x)$ is analytic in $\Delta = \{x \in \mathbb C: |x| < R_1\}$ for some $R_1 > r_2$, $h(r_2) > 0$.

    We get using Theorem~V.1 of \cite{fs09}, its proof and properties of meromorphic functions that
    $\tilde{C}^2$ is meromorphic and has only one pole $r_2$ (which is simple) in $\Delta$,
    where it satisfies (\ref{eq.exptildec}).

    The proof of the general case $l \ge 3$ follows similarly, since using Lemma~\ref{lem.structouterplanar} and
    induction we have $\rho(\cL^{<l>}) \ge \rho({\tilde \cc}^{l-1}) = r_{l-1} > r_l$, $2^l \rho(\tilde \cd) {\tilde D}(\rho(\tilde \cd)) > 1$ and the convergence
    radius of $2^l x \tilde{D}(x)$ is  $\rho(\tilde \cd) > r_2 > r_l$.
\end{proof}

\bigskip

\subsection{Proof of Theorem~\ref{thm.outerplanar}}

It remains to collect and combine the analytic results for classes related with $\ex \cb$.
Denote by $\cf$ the class of connected rooted outerplanar networks (we reuse the symbol from the previous section). 
\cite{momm05} show that the functional inverse of its exponential generating function $F$ is
\begin{equation}\label{eq.inv_rooted_outerplanar}
    \psi_F(u) = u e^{\frac 1 8 \left( \sqrt{1-6u+u^2} - 5u - 1 \right)}.
\end{equation}

\begin{lemma}\label{lem.outerplanarCl}
    Let $l \ge 2$ and let $\cb = \{K_{2,3}, K_4\}$. 
    Define $\sigma_l = \psi_F(r_l)$, where 
    $r_l$ is as in Lemma~\ref{lem.rho.tildeC}.
    There are constants $c_l > 0$, $R > \sigma_l$ and a function $g(x)$ analytic
    in $\{z \in \mathbb C: |x| < R\}$ such that
    \[
    C^l(x) = \frac {c_l} {1 - x/\sigma_l} + g(x).
    \]
\end{lemma}

\begin{proof}
%
    The class $\bar{\cc} = \bar{\cc}^{<l>}$ of nice graphs in $\cc^l$ has 
    exponential generating function
    \[
    \bar{C}(x) = \tilde{C}(F(x)).
    \]
    By \cite{momm05}, $\tau = F(\rho(\cf))$ is the smallest positive solution
    of $3u^4 - 28 u^3 + 70 u^2 - 58 u + 8 = 0$, and a numeric evaluation yields that
    $\tau = 0.170\dots > 1/6 \ge r_l$. Furthermore,
    clearly $|\cf_n| > 0$ for $n = 1, 2,\dots$.
    Thus by \ref{eq.inv_rooted_outerplanar}
      $\sigma_l = \psi_F(r_l)$ 
    is the smallest positive solution of $F(x) = r_l$
    and the unique dominant singularity of $\bar{C}(x)$.

    Using Lemma~\ref{lem.rho.tildeC} we get
    \[
        \bar{C}(x) = \frac {c_l'} {1 - F(x)/r_l} + g_1(F(x))
    \]
    and $g_1(x)$ is analytic for $|x| < R_1$ where $R_1 > r_l$.
    Since $F$ has convergence radius larger than $\sigma_l$, there is $\eps > 0$, such 
    that $F(x) < R_1$ for $x \in (0, \sigma_l + \eps)$. By the triangle inequality,
    for any $t \in \mathbb C$, $|t| < \sigma_1 + \eps$ we have $|F(t)| \le F(|t|) < R_1$
    so $g_1(F(x))$ is analytic at $x = t$.

    Now applying the supercritical 
    composition schema (Theorem~V.1 of \cite{fs09}) to the function $c_l' (1-F(x)/r_l)^{-1}$, we see
    that there is $R_2 \in (\sigma_l, \sigma_l + \eps)$ such that
    $\bar{C}(x)$ satisfies for $x \in \Delta := \{z \in \mathbb C, |z| < R_2\}$ 
    \[
    \bar{C}(x) = \frac {c_l} {1 - x/\sigma_l} + g_2(x), \quad 
      c_l = \frac {r_l c_l'} {\sigma_l F'(\sigma_l)} 
    \]
    for some function $g_2(x)$ which is analytic in $\Delta$.

    Finally, to obtain $C^l(x)$ we have to add to $\bar{C}(x)$ the exponential generating
    function $U^{<l>}(x)$ of graphs in $\cc^l$ that are not nice. An argument analogous
    to the one presented in the proof of Lemma~\ref{lem.unrootable} shows that $\rho(\cu^{<l>}) \ge \rho(\cc^{l-1})$.
    Furthermore, \cite{momm05} showed that the convergence radius of the class of outerplanar graphs $\rho(\ex \cb) = 0.1365..$.
    Now in the case $l=2$ the lemma follows, since $\rho(\cu^2) \ge \rho(\cc^1) \ge \rho(\ex \cb)> \sigma_2 = 0.1353..$,
    so $U^{<2>}(x)$ is analytic at any $t$ with $|t| < R:=\min (\rho(\ex \cb), R_2)$. Since $(r_l, l=1,2,\dots)$ is strictly decreasing
    and $\psi_F$ is increasing for $x \in (0, \rho(\cf))$, we have that $(\sigma_l, l = 1, 2, \dots)$ is strictly decreasing.
    Therefore have that $\rho(\cu^{<l>}) \ge \rho(\cc^{l-1}) = \sigma_{l-1} > \sigma_l$ by induction, and the lemma follows similarly
    as in the case $l = 2$.
\end{proof}

\begin{lemma}\label{lem.outerplanarAl}
    Let $l \ge 2$ be an integer and let $\cb = \{K_{2,3}, K_4\}$. Let $c_l, \sigma_l$ and $g$ be
    as in Lemma~\ref{lem.outerplanarCl}. Then
    \[
    |\cc^l_n| = c_l n! \sigma_l^{-n} (1 + o(1))
    \]
    and 
    \[
    |\crd {l,n}| = b_l n^{-3/4} e^{2 (c_l n)^{1/2}} n! \sigma_l^{-n} (1 + o(1))
    \]
    where
    \[
    b_l = \frac {c_l^{1/4} e^{c_l/2 + g(\sigma_l)}} {2 \pi^{1/2}}.
    \]
\end{lemma}

\begin{proof}
    The lemma follows by Lemma~\ref{lem.outerplanarCl} and Proposition~23 of \cite{bmw2013}.
\end{proof}

\bigskip

\begin{proofof}{Theorem~\ref{thm.outerplanar}}
    By \cite{momm05}, there are computable
    constants $h$ and $\gamma$, with $\gamma^{-1} = 0.1365..$
    such that the
    number of outerplanar graphs on vertex set $[n]$ is
    \[
    h n^{-3/2} \gamma^n n! (1 + o(1)).
    \]
    Using Lemma~\ref{lem.outerplanarCl}  
    and  Theorem~1.2 of \cite{cmcdvk2012}
    \[
    \gamma(\rd 3 \cb) = \sigma_3^{-1}  = 10.482.. < \gamma(\apex (\ex \cb)) = 2 \gamma = 14.642..
    \]
    Therefore by 
    Theorem~\ref{thm.main1} we have
    \[
    |(\ex 2 \cb)_n| = | (\apex (\ex \cb))_n| (1 + e^{-\Theta(n)}),
    \]
    and using Theorem~1.2 of \cite{cmcdvk2012} 
    \[
    |(\ex 2 \cb)_n|  = \frac {h} {2 \gamma} n^{-3/2} n! (2 \gamma)^n (1 + o(1)).
    \] 
    Now 
    \[
    \gamma(\rd 5 \cb) = \sigma_5^{-1} = 34.099..
    \]
    and $\gamupper(\apex(\ex 2 \cb)) \le 4 \gamma = 29.2.. < \gamma(\rd 5 \cb)$.
    By Lemma~\ref{lem.gamupper} and Lemma~\ref{lem.outerplanarAl}
    \[
    \gamma_k' = \gamma(\ex (k+1) \cb) = \gamma(\rd {2k+1} \cb) < \gamma(\apexp k (\ex \cb))
    \]
    for each $k = 2, 3, \dots$.

    Furthermore, using Lemma~\ref{lem.outerplanarAl} there are constants $b_{2k+1}, c_{2k+1} > 0$ such that
    \begin{align*}
    &|(\ex (k + 1) \cb)_{n+2k+1}| \ge 2^{\binom {2k+1} 2} |\crd {2k+1, n}| 
    \\ &= 2^{\binom {2k+1} 2} b_{2k+1} n^{-3/4} 
    \exp{\left(2 (c_{2k+1} n)^{1/2}\right)} n! (\gamma'_k)^n (1 + o(1))
    \\ & = e^{\Omega(\sqrt {n+2k+1})} (n+2k+1)! (\gamma_k')^{n+2k+1}.
    \end{align*}
    Finally, the values of $\gamma_k' = \sigma_{2k+1}^{-1}$ for $k = 2, 3, \dots$
    can be obtained using the closed-form expression $\sigma_l = \psi_F(r_l)$.
\end{proofof}

\section{Concluding remarks}
\label{sec.conclusion}

As the length of this paper indicates, analysis of classes without $k+1$ disjoint minors in $\cb$ becomes
more involved as the excluded minors get more complicated. In this work we concentrated on 
families of sets $\cb$, not covered by the results of \cite{cmcdvk2012}:
we found that indeed the highest-level structure of typical graphs in such cases
may obey a different pattern.

There are a few possible directions of further research. 
One can conjecture that for $\cb$ as in Theorem~\ref{thm.main1}, and perhaps for more general $\cb$, 
all but an exponentially small proportion of graphs in $(\ex (k+1) \cb)_n$ belong
to one of the classes $\rd {2k+1} \cb$ or $\apexp k (\ex \cb)$.
For certain $\cb$ our results imply part of this conjecture, and we gave specific
examples where this conjecture holds. To advance it further, one would need
to develop a general way of comparing growth constants for two or more candidate subclasses.
It is not clear whether this can be done without knowing the specific structure and
generating functions for $\rd {2k+1} \cb$.

It seems plausible, that for classes $\ca$ with $\aw 2 (\ca) \le j$ an analogue of Theorem~1.2 is true 
with a more general kind of redundant blockers. One can go even further and formulate conjectures as in \cite{cmcdvk2012} about classes $\ex (k+1) \cb$ in the case
when $\ex \cb$ contains all $j$-fans, $j \ge 2$, but not all $(j+1)$-fans. 
Yet another level of complexity would be to obtain any results 
in the case when $\cb$ does not contain a planar graph.


In Section~\ref{sec.part2} we proved decompositions for the class $\rd l K_4$ for general $l=1,2,\dots$. 
In the following table we present growth constants for $l$ up to $5$ obtained automatically with the help of Maple (and a simple
program to enumerate trees in $\ct_l'$ by size and number of leaves).
We explicitly proved validity of the numerical estimates up to $l=3$ in this paper.
\begin{center}
    \begin{tabular}{ | r | l | l | l |}
    \hline
    $l$ & $\gamma(\rd l K_4)$ & Comment & $\gamma(\rd l K_4)/(2^l e)$  \\ 
    \hline
    1 & 9.073311.. &  $=\gamma(\ex K_4)$, \cite{momm05} & 1.67..\\ 
    2 & 12.677273.. &  &1.17..\\
    3 & 23.524122.. & $=\gamma(\ex 2 K_4)$ & 1.08.. \\
    4 & 45.5488.. & & 1.05..\\
    5 & 89.5512.. & $=\gamma(\ex 3 K_4)$ & 1.03.. \\
    \hline
  \end{tabular}
\end{center}
The last column shows the ratio of the growth constant of $\rd l K_4$ and
the growth constant of the class $\apexp l( \ex K_3)$ (see \cite{cmcdvk2012}), where $\ex K_3$ is the class
of forests of labelled trees. Not surprisingly, the numerical estimates indicate that
this ratio approaches 1 as $l$ increases. A similar situation can be observed with the ratio $\gamma(\rd l \{ K_{2,3}, K_4 \})/2^l$.
This prompts the following questions:
is it possible that for some $k = k(n) \to \infty$, 
a typical graph from $(\ex (k + 1) K_4)_n$
consists of a forest and $2k+1$ apex vertices with probability $1 - o(1)$? 
Can this be generalised?

\end{document}